\documentclass[11pt,a4paper]{amsart}
\usepackage[utf8]{inputenc}
\usepackage[english]{babel}
\usepackage{amsmath}
\usepackage{url}
\usepackage{amsfonts}
\usepackage{amssymb}
\usepackage{mathtools}
\usepackage{xcolor}
\usepackage{nicefrac}
\usepackage[left=3cm,right=3cm,top=3cm,bottom=3cm]{geometry}
\title[Bulk-surface systems with Langmuir type adsorption]{Analysis of bulk-surface reaction-sorption-diffusion systems with Langmuir type adsorption}
\author[B.~Augner]{Björn Augner}
\address{Technische Universität Darmstadt, Fachbereich Mathematik, Fachgebiet Mathematische Modellierung und Analysis, Schlossgartenstr.\ 7, 64289 Darmstadt}
\email{augner@mma.tu-darmstadt.de}
\author[D.~Bothe]{Dieter Bothe}
\address{Technische Universität Darmstadt, Fachbereich Mathematik, Fachgebiet Mathematische Modellierung und Analysis, Alarich-Weiss-Str.\ 10, 64287 Darmstadt}
\email{bothe@mma.tu-darmstadt.de}
 
 \newtheorem{theorem}{Theorem}[section]
 \newtheorem{remark}[theorem]{Remark}
 
 \newtheorem{corollary}[theorem]{Corollary}
 \newtheorem{lemma}[theorem]{Lemma}

 \newtheorem{assumption}[theorem]{Assumption}
 \newtheorem{example}[theorem]{Example}
 \newtheorem{definition}[theorem]{Definition}

 \DeclareMathOperator{\R}{\mathbb{R}}
 \DeclareMathOperator{\N}{\mathbb{N}}

 \DeclareMathOperator{\dv}{\operatorname{div}}
 \DeclareMathOperator{\F}{\mathbb{F}}

 \DeclareMathOperator{\B}{\mathcal{B}}
 \DeclareMathOperator{\dom}{\mathrm{D}}

 \DeclareMathOperator{\sign}{\mathrm{sign}}
 
 \DeclareMathOperator{\BUC}{\mathrm{BUC}}
 \DeclarePairedDelimiterX{\inp}[2]{(}{)}{#1 \mid #2}

 \newcommand{\bb}[1]{\boldsymbol{#1}}
 \newcommand{\fs}[1]{\mathbb{#1}}

 \newcommand{\dd}{\, \mathrm{d}}
 \newcommand{\norm}[1]{\big\| #1 \big\|} 
 \newcommand{\seminorm}[1]{\big[ #1 \big]} 
 \newcommand{\abs}[1]{\big| #1 \big|}
 \newcommand{\Lip}{\operatorname{Lip}}
 \renewcommand{\vec}[1]{\boldsymbol{#1}}
 \newcommand{\diag}{\operatorname{diag}}
 \newcommand{\LL}{\mathrm{L}}
 \newcommand{\WW}{\mathrm{W}}
 \newcommand{\HH}{\mathrm{H}}
 \newcommand{\CC}{\mathrm{C}}
 \newcommand{\BB}{\mathrm{B}}

 \newcommand{\additionalremarks}[1]{}
 \newcommand{\ee}{\mathrm{e}}

 \newcommand{\HTcal}{\mathcal{HT}}

 \newcommand{\Ecal}{\mathcal{E}}
 \newcommand{\Dcal}{\mathcal{D}}
 \newcommand{\Tcal}{\mathcal{T}}
 \newcommand{\ausgrauen}[1]{}



\begin{document}
 \allowdisplaybreaks[3]
 \maketitle
 
 Version of \today.
 
 \begin{abstract}
  We consider a class of bulk-surface reaction-adsorption-diffusion systems, i.e.\ a coupled systems of reaction-diffusion systems on a bounded domain $\Omega \subseteq \R^d$ (bulk phase) and its boundary $\Sigma = \partial \Omega$ (surface phase), which are coupled via nonlinear normal flux boundary conditions.
  In particular, this class includes a heterogeneous catalysis model with Fickian bulk and surface diffusion and nonlinear adsorption of Langmuir type, i.e.\ transport from the bulk phase to the active surface, and desorption.
  For this model, we obtain well-posedness, positivity and global-in-time existence of solutions under some realistic structural conditions on the chemical reaction network and the sorption model.
 \end{abstract}

\section{Introduction}
\label{sec:introduction}

 Over the last decades, bulk/surface-systems have attracted increasing attention in Life Science and Engineering, as an increasing number of effects, especially in Chemistry and Biology have been observed that involve and can be modelled and explained by volume/surface interaction: For example, subprocesses which take place localized on some surface or interface may influence processes within a fluid phase which is in contact with the surface or interface. More generally, interaction of processes on subdomains with distinct physical or environmental properties contribute to a wide range of phenomena: They give raise to polarization and pattern formation in cell biology, may increase the spreading speeds for animals, but also for immune biology, and  in chemical engineering new catalytic materials with optimized structures facilitate an effective production of desired products and bulk/surface interaction is also the core phenomenon in wetting processes.
 When mathematically modelling such processes, such systems can be abstractly described by a open set $\Omega \subseteq \R^d$ and a d-1--dimensional manifold $\Sigma$, which is typically (part of) the boundary $\partial \Omega$ (\emph{surface}) of the domain $\Omega$ or which divides disconnected regions of the open set $\Omega$ (\emph{interface}).
 For instance, in cell biology, the formation of \emph{polarities}, i.e.\ a break in symmetry, such as front/back-polarization as a preparatory step for cell migration, is believed to be governed by a
cascade of signalling components.
 As a main contribution to cell migration, a localized peak of myosin concentration can be observed in the front of a cell. At the same time at the tail end the otherwise, i.e.\ for unpolarized cells, uniformly distributed concentration of actin is elevated.
The localisation of myosin at the front then triggers a local extension of the cell, the elevated actin concentration at the back a local contraction, so that effectively the cell moves in direction of its front, see e.g.\ \cite{RappelEdelstein-Keshet_2017} for an up-to-date review on modelling of single cell polarity, and the main targets of research connected to it.
As one of the from mathematical analysis viewpoint important works, let us exemplary describe the setting in M.~Castanho and M.~Fernandes \cite{CastanhoFernandes_2006}.
 These consider a disc-shaped model domain for the \emph{intracellular medium}, which is separated from the annulus-shaped \emph{extracellular medium} by the semipermeable \emph{cell membrane}.
 This means that only some particular chemical substances may cross the cell membrane. Thus, some chemical components are bound to the intracellular or to the extracellular medium, respectively: For example, a \emph{ligand}, i.e.\ a molecule with can form a receptor-ligand complex with a certain \emph{receptor}, may be present in the extracellular medium, but is not able to cross the cell membrane. The receptor, on the other hand, may be produced at the cell membrane, but may only desorp into the intercellular medium. Ligand-receptor complexes, then, can only be formed when the ligand coming from the extracellular medium attaches to a receptor which is present at the cell membrane (Eley--Rideal mechanism), or by first adsorbing at the cell membrane, and subsequent activation, diffusion and  formation of a ligand-receptor complex with a receptor (Langmuir--Hinshelwood mechanism).
The authors of \cite{CastanhoFernandes_2006} then proceed to show that a constant production rate of receptors at the cell membrane can lead to uniformly positive stationary solutions of the corresponding bulk/surface-system of reaction-diffusion equations.
For further read and related topics, we refer to references \cite{AnguigeRoeger_2017}, \cite{CuEKMaPoMa_2019}, \cite{EKHoZaDu_2013}, \cite{GGG_2014}, \cite{NRV_2020} and \cite{StGeSmHoRa_2020}.
\newline
Field-road models, on the other hand, have been proposed by Berestycki, Roquejoffre and Rossi \cite{BRR_2013} to model the propagation of wolves in certain regions of Canada, where the wolves employ men-made corridors (\emph{road}) to invade new habitats (\emph{field}). In this type of models, the field $\Omega \subseteq \R^n$ corresponds to a region where the wolves usually live, and where they are assumed to diffuse, reproduce and die at certain rates, modelled by the reaction-diffusion equation
 \[
  \partial_t u - d \Delta u
   = f(u),
   \quad
   t > 0, \, \vec x \in \Omega,
 \]
where $u(t,\vec x)$ decribes the population density at time $t \geq 0$ and position $\vec x \in \Omega$.
In the context of field-road models, the reproduction process $f(u)$ is typically modelled by a Kolmogorov--Petrovskii--Pisconov-model (\emph{KPP-model}, for short), i.e.\ a function $f \in \CC^1(\R)$ such that
 \[
  f(0) = f(1) = 0
   \quad \text{and} \quad
  0 < f(s) \leq s f'(0)
   \quad \text{for } s \in (0,1),
   \quad
  f(s) < 0 \quad \text{for } s > 1.
 \]
One particular important example is given by the choice $f(s) = s (1-s)$ from Fisher's equation which corresponds to logistic growth for the corresponding ODE.
The KPP-equation is well-studied, in particular it is known that, e.g.\ considering $\Omega = \R^n$, there is a finite \emph{asymptotic speed of propagation} $c_\mathrm{KPP} = 2 \sqrt{d f'(0)}$.
 In the field-road model, this asymptotic speed of propagation may (or may not) be affected by additional diffusion on the road $\Gamma \subseteq \partial \Omega$, where members of the species may diffuse with a surface diffusion coefficient $d_\Gamma > 0$, which is typically larger than $d$.
 The exchange rates between the field and the road are then modelled as being linear w.r.t.\ the trace of the field population and the road population, respectively, viz.\ the normal flux condition
  \[
   - d \partial_{\vec \nu}|_\Gamma
    = \nu u|_\Gamma - \mu u_\Gamma,
    \quad
    t > 0, \, \vec x \in \Gamma.
  \]
In absence of natality or mortality on the road, the surface reaction-diffusion equation on the road reads as
  \[
   \partial_t u_\Gamma - d_\Gamma \Delta_\Gamma u_\Gamma
    = \nu u|_\Gamma - \mu u_\Gamma,
    \quad
    t > 0, \, \vec x \in \Gamma
  \]
 with $u_\Gamma(t,\vec x)$ denoting the population density on the road at time $t \geq 0$ and position $\vec x \in \Gamma$.
 For the model case $\Omega = \R \times (0,\infty)$ and $\Gamma = \partial \Omega = \R \times \{0\}$ it has then been shown in \cite{BRR_2013} that for every bounded and non-negative initial datum $(u^0, u_\Gamma^0)$ there exists a unique, bounded and non-negative solution $(u, u_\Gamma)$ to the field-road KPP-reaction-diffusion model, and, moreover, the diffusion on the field road may affect the asymptotic speed $c_\ast$ of propagation in the following way:
 For surface diffusion coefficients $d_\Gamma > 0$ below some certain threshold $d_\Gamma^\mathrm{crit} > 0$, the asymptotic speed of propagation for the field-road model will be the same as for the full-space KPP-reaction-diffusion equation, whereas for values $d_\Gamma > d_\Gamma^\mathrm{crit}$ above this threshold, the asymptotic speed of propagation is elevated, at least in directions $\pm \vec e_1$ as these are the directions of the road.
 In \cite{BRR_2013a}, the same authors amended their model by a drift term and mortality on the road.
 Tellini \cite{Tellini_2016}, Rossi, Tellini and Valdinoci \cite{RoTeVa_2017} and \cite{Chipot_Zhang_2021} considered other field domains and various choices for the roads as well as different choices for boundary conditions at the non-road parts of the field's boundary.
 Whereas the former two considered existence and parameter-dependency of the asymptotic speed of spreading, the latter focussed on existence of positive stationary solutions.
 Berestycki, Ducasse and Rossi \cite{BDL_2020} consider the case, where an ecological niche on the road that is moving with some constant velocity $c_\mathrm{niche} \in \R$.

In chemical engineering, volume/surface-models are known and applied in practise for long.
Refining chemicals, polymerization as well as environmental protection \cite{Keil_2013} are only some of the tremendous applications of \emph{catalysts}. To give a definition, we may e.g.\ follow Keil \cite{Keil_2013} who defines \emph{catalysts} as "materials that \emph{accelerate} reactions and convert reactents \emph{preferably} to desired products \emph{without being consumed} themselves.''
This means that, \emph{activity}, i.e.\ the ability to accelerate chemical reactions, and \emph{selectivity}, i.e.\ the ability to produce close to none undesired by-products, are the most important key figures of a catalyst, and, moreover, the catalyst acts within \emph{catalytic circles}, in which the catalysts stands at the start and the end of such a circle.
While \emph{homogeneous} catalysts, especially those which are available at low cost, are still used in some industrial areas, the search for new, more effective, i.e.\ more active and more selective, catalysts nowadays is primarily focussed on \emph{heterogeneous} catalysts, especially catalytic surfaces, usually consisting of a support material, on which a catalytic surface is prepared.
The support material helps to regulate heat and electron transfer, so that the catalysts remains in a regime where it operates most efficiently.
Often activity and selectivity of a solid catalyst are achieved by a porous structure which contains single atoms of nanoparticles of since $< 20 \mathrm{nm}$ \cite{Keil_2013}.
Models for such a pore structure range from straight pores, curved pores, randomly oriented pores in the Johnson--Steward model \cite{JoSt_1965} up to Bethe-lattices, cf.\ \cite{Bethe_1935} which have been introduced by Beeckmann and Froment \cite{BeFr_1980} and Reyes and Jensen \cite{ReJe_1985} into the context of heterogeneous catalysis.
These structural models can then be combined with numerical methods such as Monte-Carlo simulations to derive realistic free energy models for the adsorbed species, which can then be employed in continuum-thermodynamic models and combined with fluid flow models.
For more information on heterogeneous catalysis, we refer to the overview articles \cite{ReKi-Mi_2010} by A.~Renken and L.~Kiwi-Minsker, \cite{Schloegel_2015} by R.~Schlögel, and the handbook of heterogeneous catalysis \cite{Handbook_HetCat} edited by G.~Ertl, H.~Knözinger and J.~Weitkamp as well as to the books by T.~Whirt \cite{Whirt_2008} and M.G.~White \cite{White_1990}. For an overview on recent advances in computational methods for heterogeneous catalysis, see B.W.J.~Chen, L.~Xu and M.~Mavrikakis \cite{ChXuMa_2021}.

Another highly relevant example of bulk-surface systems/processes is the field of dynamic wetting. There, a liquid is in contact with a solid support and the so-called contact line, i.e. the triple-line at which liquid, solid and the surrounding gas phase meet, moves along the solid wall. While standard engineering descriptions try to model the moving contact line via a (semi-)empirical relation between the contact angle, i.e. the angle between the liquid-gas and the liquid-solid interface, and the speed of normal displacement of the contact line due to its motion along the solid, more sophisticated models include the exchange of mass between the bulk liquid and the liquid-solid interface as well as the liquid-gas surface. See \cite{YS93} and the monograph \cite{YS-book} for this so-called Interface Formation Model. For more information see \cite{FBK} and the literature which is cited in these references.
All of the models also include the two-phase Navier-Stokes equations. This makes the overall model rather complex and mathematically demanding. While the analysis of such models is out of the scope of the present paper, a more fundamental mathematical analysis of bulk-surface reaction-sorption-diffusion systems can help to understand their properties.

 In the present manuscript, we consider bulk-surface reaction-sorption-diffusion models and mathematically analyse the well-posedness (in the sense of existence and uniqueness of solutions and continuous dependence on the initial datum) for a large class of such bulk-surface systems.
 In particular, we provide a global-in-time existence result under some structural conditions on the system, viz.\ a \emph{triangular structure} (also known as \emph{intermediate sum condition} \cite{MorTan22+}) of the chemical reaction network and \emph{polynomial boundedness} conditions on the chemical reaction network in the bulk phase and on the active surface.
 In view of mass conservation, it is natural to consider a normal flux condition for the coupling between the bulk phase and the active surface.
 The class of allowed non-linear relations will include the special case of a nonlinear \emph{Langmuir} type model for adsorption and desorption at the active surface.
 We include the local-in-time well-posedness analysis for more general classes of sorption models, which may, e.g., correspond to more general models for the surface chemical potentials.

\section{A brief literature overview}

 We next give a brief overview on some relevant papers on the mathematical theory of reaction-diffusion-systems with an emphasis on results related to local-in-time-existence, to regularity and to global-in-time existence of solutions and works which are related to the present one.
 The different branches in the theory of reaction-diffusion-systems are, therefore, so plenty that we do not even try to give a complete overview over all of them, but rather focus on the development of the theory of reaction-diffusion-systems from a mathematical analyst's point of view with a particular bias on regularity and well-posedness results (in the sense of initial-boundary value problems) both of local and global-in-time type.
 That means, we focus on the analysis of reaction-diffusion system.
 Due to their historical importance for the development of the research field, we will start with some of the milestones in the theory for pure bulk diffusion.
 Though, most of the time, no bulk/surface interaction has been considered in these works, the techniques developed there, nevertheless, provide useful orientation for the bulk/surface theory.
 Here, the fundamental works of Ladyshenskaja, Solonnikov and Ural'ceva \cite{LaSoUr68} remain up to today a tremendous corner-stone, which provide optimal regularity results in Hölder space and Sobolev space settings to linear, but possibly inhomogeneous parabolic equations and systems.
 Further results in this perspective can be found in N.D.~Alikakos' article \cite{Ali79}.
 H.~Amann extensively studied quasilinear parabolic equations, with an emphasis on reaction-diffusion systems in \cite{Amann_1990} and on global-in-time existence of solutions in \cite{Amann_1989}.
 In the context of quasilinear equations, let us also mention the seminal paper \cite{Pruess_2003} by J.~Prüss.
 R.~Denk, M.~Hieber and J.~Prüss in their works \cite{DeHiPr03}, \cite{DeHiPr07} treated the linear vector-valued case.
 Based on such $\LL_p$-maximal regularity results and boot-strapping methods, the Boundedness-Implies-Continuation property has been employed in works such as \cite{Ama83}, \cite{Ama85} by H.~Amann, \cite{Hen81} by D.~Henry, \cite{Rot84} by F.~Rothe, \cite{HolMarPie87}, \cite{MarPie91}, \cite{Pierre_2010} by M.~Pierre and co-authors as well as \cite{BoRo_2015}, \cite{QuiSou07} and \cite{Souplet_2018}, typically in combination with $\LL_p$-estimates derived by duality methods, to obtain global-in-time existence results.
 Note that for the more realistic case of a Maxwell--Stefan diffusion model for multi-component fluid mixtures, the mathematical analysis is more involved due to its quasi-linear nature and the summation constraint on the diffusive mass fluxes, see the contribution \cite{Bothe_2001} of the second author of this paper.
 
 For bulk/surface systems, some of the abstract optimal regularity results for linear parabolic systems have been generalised to the vector-valued case with bulk-surface-coupling in \cite{DePrZa08} by R.~Denk, J.~Prüss and R.~Zacher.
 For the case of a bulk reaction-diffusion system with nonlinear boundary conditions resulting from a fast surface chemistry limit, see \cite{AugBot21a} and \cite{AugBot21b} by the paper's authors.
 \newline
 In recent years, bulk-surface reaction-diffusion systems have drawn increasing attention of several research groups, see the aforemention active research areas in life science and engineering.
 Obviously this is just a short glimpse on an abundant reservoir of literature on this topic from that perspective.
 The mathematical analysts' perspective on bulk/surface systems seems to be comparably underrepresented, given the tremendous importance of the subject in view of developments in, e.g., modern bio-technologies.
 Let us mention some recently published articles, where the authors mainly consider the global-in-time well-posedness of bulk-surface systems (sometimes also called volume-surface systems).
 O.~Sou\v{c}ek et al.\ \cite{SoOrMaBo19} considered bulk/surface systems within a thermodynamically consistent framework, with particular emphasis on diffusion and sorption phenomena.
 For example, the second author, M.~Köhne, S.~Maier and J.~Saal \cite{BKMS17} consider a bulk-surface reaction diffusion system with Fickian type diffusion in a  cylindrical bulk phase and on parts of its surface, amended with surface chemical reactions.
  Additionally, they imposed inflow and outflow boundary conditions on the molar fluxes $\vec j_i$ corresponding to the molar concentrations $c_i$ on a \emph{transmissive part} $\partial \Omega \setminus \Sigma$ and no flux boundary conditions on the boundary $\partial \Sigma$ (in the sense of manifolds) between the active and transmissive part of the boundary.
 In \cite{NRV_2020}, B.~Niethammer, M.~Röger and J.J.L.~Vel\'{a}zquez¸ start from a bulk-surface reaction-diffusion-model with a Henry type sorption law and a Michaelis--Menten type surface reaction model plus an external stimulus signal.
 A main part of their research is not to stop at the question of global-in-time well-posedness of that particular system under consideration, but to derive polarisation models for suitable scaling of the parameter.
 In fact, they can show that under certain conditions these models have \emph{critical masses} for polarisation.
 Similar models have also been considered, e.g.\ in \cite{HauRoe18} and \cite{BaeRoe22}.
 On the frontier of global-in-time existence, and for a much more general class of bulk-surface reaction diffusion systems, quite recently J.~Morgan and B.Q.~Tang \cite{MorTan22+} addressed reaction-diffusion systems where the bulk and the surface dynamics are coupled via a non-linear Robin-type boundary condition. 
 Their main ingredients to obtain global-in-time existence of solutions are an assumption on \emph{mass control} and an \emph{intermediate sum condition}, similar to those conditions which we impose in this work.
 Let us in detail comment on similarities and differences in comparing their results to ours.
 Regarding the general structure of the system under consideration, we also consider a bulk-surface system of reaction-diffusion equations which are coupled by some kind of sorption mechanism. Whereas in our case, we include a (given) macroscopic flux $\vec u$ in our system, leading to an advective term $\vec u \cdot \nabla \vec c_i$ in the balance equation for each component $c_i$, Morgan and Tang consider the basic case where the macroscopic flux vanishes, i.e.\ $\vec u = 0$.
  An inspection of their technique of proof, however, indicates that they could have also included such a convective term in their analysis, essentially without changing the flow of their arguments.
  Also the general conditions on the sorption and reaction rates, which can be paraphrased as local Lipschitz continuity and quasi-positivity, correspond to each other in both works.
  The most essential differences can be observed concerning generality of the functions modelling chemical reaction rates and sorption rates, and on the regularity of the initial data.
  On the one hand, Morgan and Tang demand less structure on the sorption rates, without any assumption on a diagonal structure of the sorption rates, but they do demand that the initial data for the system lie in some $\WW_p^{2-2/p}$-Sobolev--Slobodetskii space, and -- in contrast to us -- they demand that $p > d$ which immediately implies some Hölder-regularity on the initial data, and, hence, the solution.
  In our case, on the other hand, we can also handle initial data in the class $\WW_p^{2-2/p}$ for some $p < d$.
  Therefore, already establishing local-in-time existence of strong solutions in the class $\WW_p^{(1,2)}$ requires considerable effort in our setting, whereas in \cite{MorTan22+} the initial data are only assumed to be attained in the sense that solutions lie in the class $\CC([0,\tau);\LL_p)$.
  As a by-product, we also have to impose more regularity in the formulation of the polynomial growth condition, and our polynomial growth conditions are, at first, growth conditions on the derivatives of the functions modelling chemical reaction and sorption rates, and only by the fundamental theorem of calculus these imply growth conditions on the respective functions itself, as they have been directly imposed in \cite{MorTan22+}. 
 In this sense, \cite{MorTan22+} and the present article focus on different, but closely related aspects of bulk-surface reaction-diffusion-sorption systems, and whereas the setting in \cite{MorTan22+} is more general from a structural point of view, we can handle a more notion of solutions, which allows for´ less regular initial data.
 Note that in \cite{MorTan22+}, the regularity of initial data might possibly be relaxed to initial data of class $\LL_\infty$, but even this class will in general not include all functions of class $\WW_p^{2-2/p}$ if $p < d$ is allowed for.
 This different setting is also reflected in the different proof technique for the global-in-time existence result.
 Similarly to the strategy in \cite{MorTan22+}, but also to works focusing on pure bulk reaction-diffusion systems such as \cite{Pierre_2010}, the key strategy for establishing global-in-time existence of solutions is the derivation of $\LL_q$--growth bounds, and, first for finite $q \in (1,\infty)$ and finally for the case $q = \infty$.
 As it is quite typical for such problems, we will use \emph{duality methods}, but -- in contrast to the main trend -- we will also include dual norms in the duality method based estimates that will be applied, cf.\ especially Lemma \ref{lem:dual-estimate-variant} and Lemma \ref{lem:surface-dual-estimate-variant}.
 \newline
 The manuscript is organised as follows:
 In Section~\ref{sec:introduction}, we fix the general framework and comment on some conditions on the structure and smoothness which we impose on the models for the chemical reactor, its active surface and the physical-chemical processes such as advection, diffusion, chemical reactions in the bulk and on the surface and adsorption and desorption at the boundary.
 Thereafter, we will give an overlook over the linear theory of $\LL_p$-maximal regularity for parabolic equations, on which the application to semi-linear problems is based on.
 Then, in Sections~\ref{sec:Positivity} and~\ref{sec:LIT-well-posedness}, we demonstrate that the solution to positive initial data is (just as one would expect) positive and only thereafter we show local-in-time well-posedness, i.e.\ existence and uniqueness of solutions (at least on some small time interval) and continuous dependence of the solution on the initial datum.
 To prepare a boot-strap argument, we derive several auxiliary results, including dual estimates and $\LL_p$--$\LL_q$-estimates in Section~\ref{sec:auxiliary-results}, which will become crucial when actually proving global-in-time existence of solutions in Section~\ref{sec:global-existence}.
 
\section{Modelling assumptions}
\label{sec:modelling_assmpts}

 Before introducing the concrete class of models under consideration, let us indicate how to derive this model, starting from general continuum-thermodynamic balance equations of reaction-diffusion-systems.
 Thereby, we will also introduce some notation and physical interpretation of the variables appearing in the equations.
 We consider reaction-diffusion systems both in the bulk phase, i.e.\  inside the domain $\Omega \subseteq \R^d$, and on its active surface, given by (a part of) its boundary $\Sigma \subseteq \partial \Omega$.
 For example, one may think of a chemical reactor, consisting of the interior bulk phase and a catalytic surface on the boundary.
 The general balance equations for the individual concentrations $c_i$ of a substance $A_i$ in a fluid mixture within the bulk phase read as
  \[
   \partial_t c_i + \dv(c_i \vec v_i)
    = r_i(\vec c)
    \quad
    \text{in } (0,\infty) \times \Omega,
  \]
 where, $c_i(t,\vec x)$ denotes the (molar) concentration of species $A_i$ at position $\vec x \in \Omega$ and time $t \in \R$ and $\vec v_i$ its individual continuum velocity.
Moreover, we write $\vec c = (c_1, \ldots, c_N)^\mathsf{T}$.
 The individual velocities $\vec v_i$ can, in the most general form of a model, obtained from individual momentum balance equations, .
 In the more case of a single momentum balance, also considered here, the diffusive part $\vec u_i = \vec v_i - \vec v$ will be modelled based on thermodynamic principles, where $\vec v$ the barycentric velocity of the fluid mixture.
 In the latter case, the individual velocity $\vec v_i = \vec v + \vec u_i$ of each fluid mixture component can be decomposed into an advective part $\vec v$, and a (frame-indifferent) diffusive part $\vec u_i$.
 When it comes to thermodynamic consistent models, it will be important that the diffusive part will be modelled observation frame indifferent, cf.\ \cite{BotDre15}, whereas the advective part $\vec v$ will be depend on the observation frame.
 Denoting by $\vec j_i^\mathrm{diff} := c_i \vec u_i$ the (molar-based) \emph{diffusive mass flux}, the general balance equations for reactive-fluid mixtures take the form
  \[
   \partial_t c_i + \dv (c_i \vec v + \vec j_i^\mathrm{diff})
    = r_i(\vec c)
    \quad
    \text{on } (0,\infty) \times \Omega,
  \]
 where $r_i$ models chemical reaction rates.
 We will employ Fick's model for the diffusive flux, i.e.\ we let
  \[
   \vec j_i^\mathrm{diff}
    = - d_i \nabla c_i
  \]
 with some species-dependent (constant) diffusivity $d_i > 0$.
 At least for dilute systems, i.e.\ systems where the concentration of the solvent is by far higher than the concentrations of the solutes which are dissolved in the solvent, this gives quite a good approximation to reality.
 \newline
 For modelling the advective part, i.e.\ the barycentric velocity $\vec v$, there are various approaches, such as the Navier--Stokes equations, the Stokes equations or Darcy's balance equation.
 Here, however, we will assume the macroscopic vector field $\vec v$ to be given.
 Let us note that under appropriate conditions on $\vec v_i$ the advective part can be treated as a perturbation in the analysis to come.
 \newline
 The reaction-diffusion model within the bulk phase reads
  \[
   \partial_t c_i + \dv ( c_i \vec v - d_i \nabla c_i)
    = r_i(\vec c),
    \quad
    i = 1, \ldots, N,
  \]
 which in the \emph{incompressible} case, i.e.\ $\dv v = 0$, and for constant diffusivities $d_i > 0$ can also be written as
  \[
   \partial_t c_i + \vec v \cdot \nabla c_i - d_i \Delta c_i
    = r_i(\vec c),
    \quad
    i = 1, \ldots, N.
  \]
 Note that we do \emph{not} include the solvent in the model.
 Let us remark that for higher concentrations of the solutes more delicate models also include cross-diffusion effects, such as the Fick--Onsager or the Maxwell--Stefan model, are available.
 The mathematical analysis of such cross-diffusion systems will be the topic of further research in the area of bulk-surface reaction-diffusion-sorption systems, as it is more involved due to its quasilinear nature. 
 Also note that results as in \cite{HeMePrWi17} indicate that for low concentrations, the Maxwell--Stefan diffusion model is very well approximated by the linear Fickian diffusion model.
 \newline
 Similarly, on the surface, the general balance equation for reaction-diffusion-systems reads
  \[
   \partial_t c_i^\Sigma
    + \dv_\Sigma(\vec j_i^\Sigma)
    = r_i^\Sigma + s_i^\Sigma,
  \]
 where the additional term $s_i^\Sigma$ corresponds to \emph{sorption processes}, i.e.\ the exchange of mass between the bulk phase and the boundary.
 Here, we will assume that the boundary of the chemical reactor is \emph{impermeable}, i.e.\ no particles may cross $\Sigma$, so that, identifying $A_i^\Sigma$ with an adsorbed version of the species $A_i$, we obtain the constraint
  \[
   s_i^\Sigma
    = c_i \vec v_i \cdot \vec \nu
    = c_i (\vec v \cdot \vec \nu) + \vec j_i^{\Sigma,\mathrm{diff}} \cdot \vec \nu,
  \]
 i.e.\ $s_i^\Sigma$ is the normal part of the individual mass flux near the boundary.
 We will further restrict our attention to the case where the given macroscopic velocity field $\vec v$ is subject to no-flux boundary conditions, i.e.\ we assume that
  \[
   \vec v \cdot \vec \nu
    = 0
    \quad
    \text{on } \Sigma,
  \]
 and, consequently, the constraint on the sorption rate reduces to
  \[
   s_i^\Sigma
    = \vec j_i^{\Sigma,\mathrm{diff}} \cdot \vec \nu.
  \]
 On the one hand, this may serve as a \emph{definition} of the sorption rate $s_i^\Sigma$, but, on the other hand, we may use $s_i^\Sigma$ to impose an adsorption-desorption \emph{model} and, thereby, to close the system; this will be done below.
 For the surface flux, we consider the special case of vanishing surface velocity $\vec v^\Sigma = \vec 0$ (the model will not include the solid material constituting the catalytic surface) and model the surface diffusive fluxes by $\vec j_i^\Sigma = - d_i^\Sigma \nabla_\Sigma c_i^\Sigma$ for the surface gradient $\nabla_\Sigma$, i.e.\ we employ a Fickian type diffusion model also on the surface.
 \newline
 To sum up, we consider bulk-surface reaction-sorption-diffusion models of the form
  \begin{alignat}{2}
   \partial_t c_i + \vec v \cdot \nabla c_i - \dv (d_i \nabla c_i)
    &= r_i^\Omega(\vec c)
	\qquad &
    &\text{in } (0, \infty) \times \Omega, \, i = 1, \ldots, N,
    \nonumber \\
   \partial_t c_i^\Sigma - \dv_\Sigma (d_i^\Sigma \nabla_\Sigma c_i^\Sigma)
    &= r_i^\Sigma(\vec c^\Sigma) + s_i^\Sigma(\vec c, \vec c^\Sigma)
	\qquad &
    &\text{on } (0, \infty) \times \Sigma, \, i = 1, \ldots, N,
    \label{eqn:RDSS}
    \\
   - d_i \partial_\nu c_i
    &= s_i^\Sigma(\vec c, \vec c^\Sigma)
	\qquad &
    &\text{on } (0, \infty) \times \Sigma, \, i = 1, \ldots, N,
    \nonumber
  \end{alignat}
 where $\Omega \subseteq \R^d$ is a (sufficiently regular) bounded domain, serving as a model for a (prototypical) chemical reactor.
 Moreover, $\Sigma \subseteq \partial \Omega$ denotes a (sufficiently regular) part of the boundary of $\Omega$, and which later on will be interpreted as an \emph{active part} of the reactor surface.
  In this paper, for simplicity, we will consider the special case that the active surface constitutes the whole boundary of the chemical reactor, i.e.\  $\Sigma = \partial \Omega$.
 \begin{remark}
  In the research paper \cite{BKMS17} the authors restrict themselves to the special, but in practical applications relevant case of \emph{cylindrical} domains $\Omega$ and (in chemical engineering most relevant) spatial dimension $d = 3$, whereas $\Omega$ may be any bounded $\CC^2$-domain in $\R^d$ here.
  The additional geometrical constraint helps in \cite{BKMS17} to include inflow and outflow conditions at the bottom resp.\ top of the cylindrical domain in the local well-posedness analysis, i.e.\ for $\LL_p$-maximal regularity theory.
  As for such domains a fully developed $\LL_p$-maximal regularity theory is available, the observations made in the present paper, however, will carry over to the case of inflow and outflow boundary conditions for a cylindrical domain, too.
 \end{remark}
 To get a complete and consistent model, we write the sorption rates $s_i^\Sigma$ between bulk and surface as $s_i^\Sigma(\vec c|_\Sigma, \vec c^\Sigma)$, where $s_i^\Sigma$ is a function which in general may depend on two $\R^N$-valued variables.
 These functions reflect the model we impose for the adsorption and desorption mechanisms.
 \newline
 Below, we will state the most relevant conditions on the `externally'' given `velocity field $\vec v$, the diffusion coefficients $d_i$, $d_i^\Sigma$ and the reaction and sorption models $r_i^\Sigma$ and $s_i^\Sigma$.
 For the moment, however, let us start with some realistic modelling assumptions on the sorption model and, in particular, compare these with the special case of a Langmuir adsorption model which constitutes a highly relevant model for ad- and desorption at an active surface:
  \begin{enumerate}
   \item
    As a function of $c_i$ and $c_i^\Sigma$, $s_i^\Sigma$ is continuously differentiable twice, and its first and second derivatives are uniformly bounded on $\R^2$.
    As we shall see, we have to weaken this condition to include Langmuir type adsorption.
   \item
   \label{cond:3}
    For fixed surface concentration, we expect that the higher the bulk concentration, the higher the adsorption rate of that respective component.
    That is, we demand that the function $s_i^\Sigma$ increases in the argument $c_i$.
   \item
   \label{cond:4}
    For fixed bulk concentration, the desorption rate increases by increasing the surface concentration of the respective component.
    Therefore, we assume that the functions $s_i^\Sigma$ decrease in arguments $c_k^\Sigma$.
   \item
    Moreover, we expect adsorption and desorption not to increase faster than linearly, if we increase the bulk resp.\ the surface concentration.
    Hence, we assume that the function $s_i^\Sigma$ is bounded linearly both from below and from above, in the sense that there are constants $k_i^{\mathrm{de}}, k_i^{\mathrm{ad}} > 0$ such that
     \[
      - k_i^{\mathrm{de}} z_i^\Sigma
       \leq s_i^\Sigma(\vec z, \vec z^\Sigma)
       \leq k_i^{\mathrm{ad}} z_i
       \quad \text{for} \quad
       \vec z, \vec z^\Sigma \in \R_+^N.
     \]
  \end{enumerate}
 Typically, we will write the effective sorption rate as the difference between the adsorption rate and the desorption rate, i.e.\
  \[
      s_i^\Sigma
    = s_i^{\mathrm{ad}} - s_i^{\mathrm{de}},
  \]
 where we expect $s_i^{\mathrm{ad}}$ and $s_i^{\mathrm{de}}$ to be functions which are non-negative on $(0, \infty)^2$, and conditions \ref{cond:3} and \ref{cond:4} will typically mean that $s_i^{\mathrm{ad}}(\vec z, \vec z^\Sigma)$ is bounded from above by $k_i^{\mathrm{ad}} z_i$, whereas $s_i^{\mathrm{de}}(\vec z, \vec z^\Sigma)$ is bounded from above by $k_i^\mathrm{de} z_i^\Sigma$.
 This, in particular, means that the adsorption rate is at most linearly growing w.r.t.\ the boundary trace of the bulk concentration $c_i|_\Sigma$ and, on the other side, the desorption rate is at most growing linearly w.r.t.\ the surface concentration.
 All these assumptions are obviously met by the so-called Henry model for ad- and desorption, where ad- and desorption simply depend linearly on the trace of the bulk concentration and the surface concentration, resp., i.e.
  \[
    s_i^{\mathrm{ad}}(c_i|_\Sigma, c_i^\Sigma) = k_i^{\mathrm{ad}} c_i|_\Sigma,
     \quad
    s_i^{\mathrm{de}}(c_i|_\Sigma, c_i^\Sigma) = k_i^{\mathrm{de}} c_i^\Sigma.
  \]
 The disadvantage of this model is that it does not take into account that (in the relevant case of relatively high surface concentrations and a finite capacity on the surface) for adsorption of a molecule to the surface, enough empty space on the surface is necessary.
 In this sense, it is only a realistic model for low concentrations on the active surface, whereas for higher surface concentrations it will be more realistic to assign a maximal capacity $c_i^{\Sigma,\infty} > 0$ to the surface, and to consider the well-known Langmuir model \cite{Langmuir1918}
  \[
   s_i^{\mathrm{ad}}(c_i|_\Sigma,c_i^\Sigma)
    = k_i^{\mathrm{ad}} c_i|_\Sigma \left( 1 - \theta(\vec c^\Sigma) \right),
    \quad
    c_i|_\Sigma \geq 0, \, 0 \leq c_i^\Sigma \leq c_i^{\Sigma,\infty},
  \]
 where $\theta(\vec c^\Sigma) = \frac{1}{c_S^{\Sigma,\infty}} \sum_{i=1}^N \sigma_i c_i^\Sigma$ denotes the occupancy number on the surface w.r.t.\ to some final number $c_S^\Sigma$ of available sites per area element.
 This first ansatz, however, has a serious drawback: The mass-action-kinetics models for the reaction rates $r_i^\Sigma(\vec c^\Sigma)$ and the Fickian diffusive flux model $- d_i^\Sigma \nabla_\Sigma c_i^\Sigma$ do, in general, not prevent the weighted sum of surface concentrations $\sum_{i=1}^N \sigma_i c_i^\Sigma$  to go above the threshold $c_S^{\Sigma,\infty}$, as for systems, in general, \emph{no maximum principle} is valid.
 To overcome this problem, there are at least two options available: On the one hand, we might refine the reaction-rate model and the diffusion model to prevent the case $\sum_{i=1}^N \sigma_i c_i(t,y) > c_S^{\Sigma,\infty}$.
 This can be achieved, for example, by modelling the chemical reaction rates based on the Langmuir model for the free energy and diffusive mass fluxes by a Maxwell--Stefan or Fick--Onsager model.
 Alternatively, we may extend the domain of $r_i^\mathrm{ad}$ to $\R^2$ by setting
  \[
   s_i^{\mathrm{ad}}(\vec z, \vec z^\Sigma)
    = k_i^{\mathrm{ad}} z_i \left(1 - \theta(\vec z^\Sigma) \right)^+,
    \quad
    \vec z, \vec z^\Sigma \in \R_+^N,
  \]
 where $s^+ := \max \{0, s\}$ is the positive part of a scalar $s \in \R$.
 Let us compare the properties of this particular sorption model with the ad-hoc assumptions from above:
  First of all, the function $s_i^\Sigma(\vec c, \vec c^\Sigma)$ depends on (the boundary trace of) the bulk concentration $c_i$ and the surface concentration $c_i^\Sigma$ of the $i$-th species and it is increasing in the first and decreasing in the second component.
  The upper and lower linear bounds $- k_i^{\mathrm{de}} c_i^\Sigma \leq s_i^\Sigma(c_i, c_i^\Sigma) \leq k_i^{\mathrm{ad}} c_i$ hold true for all non-negative concentrations $c_i, c_i^\Sigma \geq 0$ as well. However, concerning regularity, in the Langmuir sorption model, the function $s_i^\Sigma$ is neither a $\CC^2$-function, nor are its first and second derivative (which exist almost everywhere) uniformly bounded on $\R^2$.
 We therefore aim to prove local-in-time existence and, thereafter, global existence results under less restrictive growth and regularity conditions.
 This will allow us to handle the Langmuir adsorption without, say, replacing the Langmuir adsorption by more regular functions as it has been done, e.g., in \cite{BKMS17}.
 Furthermore, we allow for chemical reactions in the bulk phase, constituting an additional source or sink term $r^\Omega_i(\vec c)$ on the right-hand side of the reaction-diffusion-advection equations in the bulk phase, viz.\
  \[
   \partial_t c_i + \vec v \cdot \nabla c_i - \dv (d_i \nabla c_i)
    = r^\Omega_i(\vec c)
    \quad
    \text{in } (0, \infty) \times \Omega, \, i = 1, \ldots, N.
  \]
 Similar to reaction-diffusion-systems in the bulk phase with, say, no-flux boundary conditions, the semilinear term can be handled using maximal regularity results to establish local-in-time existence of solutions.
 However, it may very well influence the long-time behaviour of the reaction-diffusion-system, so that we have to impose additional structural conditions to prevent a blow-up of the solution in finite time.  
    
\section{On the $\LL_p$-theory for the linearised problem}
\label{sec:Lp-theory}

The reaction-sorption-diffusion systems under consideration are of semilinear nature: The diffusive fluxes in the bulk phase ($\vec j_i = - d_i \nabla c_i$) and on the surface ($\vec j_i^\Sigma = - d_i^\Sigma \nabla_\Sigma c_i^\Sigma$) depend linearly on the (gradient of the) bulk resp.\ surface concentrations $c_i$ resp.\ $c_i^\Sigma$.
The effective sorption rate $s_i^\Sigma$ at the boundary constitute a, in general nonlinear, coupling between the two dynamics.
The models for the chemical reaction rates $r_i^\Omega$ and $r_i^\Sigma$ in the bulk phase and on the surface, respectively, depend (nonlinearly) on the bulk resp.\ surface concentrations.
Therefore, as for bulk reaction-diffusion-advection systems with, say, no-flux boundary conditions, we first revisit the $\LL_p$-maximal regularity theory of a linearised version of the system.
In fact, by the diagonal nature of the diffusion matrices $\bb D = \diag(d_i) \in \R^{N \times N}$ and $\bb D^\Sigma = \diag(d_i^\Sigma) \in \R^{N \times N}$, this problem can easily be reduced to the question of $\LL_p$-maximal regularity for the respective scalar equations, for which a complete theory is available.

 We use the following notation: For $p \in (1, \infty)$, we denote
  \begin{align*}
   \fs E_p^\Omega(T)
    &= \WW^1_p(0,T; \LL_p(\Omega;\R^N)) \cap \LL_p(0,T; \WW^2_p(\Omega;\R^N))
    \\
    &= \WW^{(1,2)}_p(\Omega_T;\R^N)
	\qquad &
    &
    \hspace{-2cm}
    \text{(max.\ reg.\ space in the bulk phase)}
    \\
   \fs E_p^\Sigma(T)
    &= \WW^1_p(0,T; \LL_p(\Sigma;\R^N)) \cap \LL_p(0,T; \WW^2_p(\Sigma;\R^N))
    \\
    &= \WW^{(1,2)}_p(\Sigma_T;\R^N)
	\qquad &
    &
    \hspace{-2cm}
    \text{(max.\ reg.\ space on the surface)}
    \\
   \fs F_p^\Omega(T)
    &= \LL_p(0,T;\LL_p(\Omega;\R^N))
    \\
    &\cong \LL_p((0,T) \times \Omega)
    = \LL_p(\Omega_T;\R^N)    
	\qquad &
    &
    \hspace{-2cm}
    \text{(base space in the bulk phase)}
    \\
   \fs F_p^\Sigma(T)
    &= \LL_p(0,T;\LL_p(\Sigma;\R^N))
    \\
    &\cong \LL_p((0,T) \times \Sigma;\R^N)
    = \LL_p(\Sigma_T;\R^N)
	\qquad &
    &
    \hspace{-2cm}
    \text{(base space on the surface)}
    \\
   \fs G_p^\Sigma(T)
    &= \WW^{\nicefrac{1}{2} - \nicefrac{1}{2p}}_p(0,T; \LL_p(\Sigma;\R^N))
     \cap \LL_p(0,T; \WW^{1-\nicefrac{1}{p}}_p (\Sigma;\R^N))
     &
    \\
    &= \WW^{(1,2)\cdot(\frac{1}{2} - \frac{1}{2p})}_p(\Sigma_T;\R^N)
	\qquad &
	&
    \hspace{-2cm}
	\text{(Neumann trace space)}
    \\
   \fs H_p^\Sigma(T)
    &= \WW^{1-\nicefrac{1}{2p}}_p(0,T;\LL_p(\Sigma;\R^N))
     \cap \LL_p(0,T; \WW^{2-\nicefrac{1}{p}}_p(\Sigma;\R^N))
    \\
    &= \WW^{(1,2) \cdot (1 - \frac{1}{2p})}_p(\Sigma_T;\R^N)
	\qquad &
    &
    \hspace{-2cm}
    \text{(Dirichlet trace space)}
    \\
   \fs I_p^\Omega
    &=\big( \LL_p(\Omega;\R^N), \WW^2_p(\Omega;\R^N) \big)_{1-\nicefrac{1}{p},p}
    \\
    &= \WW^{2-\nicefrac{2}{p}}_p(\Omega;\R^N),
	\qquad &
    &
    \hspace{-2cm}
    \text{(phase space in the bulk phase)}
    \\
   \fs I_p^\Sigma
    &= \big( \LL_p(\Sigma;\R^N), \WW^2_p(\Sigma;\R^N) \big)_{1-\nicefrac{1}{p},p}
    \\
    &= \WW^{2-\nicefrac{2}{p}}_p(\Sigma;\R^N)
	\qquad &
    &
    \hspace{-2cm}
    \text{(phase space on the surface)}
  \end{align*}
 with the short hand notations $\Omega_T := (0,T) \times \Omega$ and $\Sigma_T := (0,T) \times \Sigma$ and $\cong$ denotes equality up to isometric isomorphisms; here, by identifying $f(t)(\vec x)$ with $f(t,\vec x)$.
 We use standard notation, i.e.\ we denote by $\LL_p(\Omega)$, $\WW_p^k(\Omega)$, $\WW_p^s(\Omega)$ Lebesgue, Sobolev and Sobolev--Slobodetskii spaces, resp., and by $\LL_p(\Omega;E)$, $\WW_p^k(\Omega;E)$ and $\WW_p^s(\Omega;E)$ ($E$ any real or complex Banach space) their respective vector-valued versions.
 Also $\LL_p(\Sigma)$, $\WW_p^k(\Sigma)$, $\WW_p^s(\Sigma)$ denote the corresponding spaces w.r.t.\ the surface measure $\sigma$ on the boundary $\Sigma$, which is a (sufficiently regular) submanifold of dimension $d-1$ without boundary (in the sense of manifolds).
 We refer, e.g., to \cite{HebRob08} for an overview on (the more general case of) Sobolev spaces over Riemannian manifolds and density, embedding and compactness theorems.
\newline
Note that in the linearised version of the bulk-surface reaction-diffusion-advection systems, the problems in the bulk phase (inhomogeneous Neumann problem) and on the surface are decoupled.
The Neumann problem is classical and can be handled even in the vector-valued case; see \cite[Theorem 2.1]{DeHiPr07} for a very general result on parabolic initial-boundary value problems in domains and \cite[Theorem 6.4.3]{PruSim16} for a version for compact hypersurfaces without boundary (also see \cite[Theorem 1.23]{Ama17} for a much more general result on uniformly regular Riemannian manifolds).
The following $\LL_p$-maximal regularity result can, therefore, easily be derived from the well-known scalar theory.

 \begin{lemma}[$\LL_p$-maximal regularity]
 \label{lem:maximal_regularity_scalar_case}
  Let $\Omega \subseteq \R^d$ be a bounded $\CC^2$-domain, $p \in (1, \infty)$, $T_0 > 0$, $T \in (0, T_0]$, $\vec v \in \LL_s((0,T_0);\WW_r^1(\Omega))$ with $r,s \in (1, \infty)$ such that $1 -\frac{2}{s} - \frac{d}{r} > 0$ and $d_i \in \CC^1(\overline{\Omega})$, $d_i^\Sigma \in \CC^1(\Sigma)$\additionalremarks{If only $\LL_\infty$-regularity, the regularity spaces have to be adjusted, e.g.\ $\fs E_p^\Omega(T) = \{ \vec u \in \WW_p^1(\Omega_T;\R^N): \, d_i \nabla u_i \in \LL_p(\WW_p^1) \}$. Question: If $d_i$, $d_i^\Sigma \in \LL^\infty$, does one have $\LL_p$-maximal regularity with the adjusted maximal regularity space above? Is such a result available from the literature? Dieter: I doubt this" Some Hoelder rgularity may be needed! Possibly relevant reference: W.~Arendt, A.F.M.~ter~Elst: Gaussian Estimates for second order elliptic operators with boundary conditions, J.~ Operator Theory 38 (1997), 87--130.} be such that $d_i > 0$ and $d_i^\Sigma > 0$.
  The quasi-autonomous, linear problem
   \begin{alignat*}{2}
    \partial_t c_i + \vec v \cdot \nabla c_i - \dv( d_i \nabla c_i)
     &= f_i
	\qquad &
     &\text{on } (0,T) \times \Omega, \, i = 1, \ldots, N,
     \\
    \partial_t c_i^\Sigma - \dv_\Sigma( d_i^\Sigma c_i^\Sigma)
     &= f_i^\Sigma
	\qquad &
     &\text{on } (0,T) \times \Sigma, \, i = 1, \ldots, N,
     \\
    - d_i \nabla c_i \cdot \vec \nu
     &= g_i^\Sigma
	\qquad &
     &\text{on } (0,T) \times \Sigma, \, i = 1, \ldots, N,
     \\
    (c_i(0,\cdot), c_i^\Sigma(0,\cdot))
     &= (c^0_i, c^{\Sigma,0}_i)
   \end{alignat*}
  has a unique solution $(\vec c, \vec c^\Sigma) \in \fs E_p^\Omega(T) \times \fs E_p^\Sigma(T)$ if and only if $(\vec f, \vec f^\Sigma, \vec g^\Sigma) \in \fs F_p^\Omega(T) \times \F_p^\Sigma(T) \times \fs G_p^\Sigma(T)$, $(\vec c^0, \vec c^{\Sigma,0}) \in \fs I_p^\Omega \times \fs I_p^\Sigma$, and, additionally in case $p > 3$, $- d_i \nabla c_{0,i} \cdot \vec \nu = g_i^\Sigma|_{t=0}$ (in the trace sense) on $\Sigma$ for $i = 1, \ldots, N$.
  Moreover, there is a constant $C_{p,T_0} > 0$, which is independent of the data $(\vec f, \vec f^\Sigma, \vec g^\Sigma)$, $\vec v^0$ and $\vec v^{\Sigma,0}$ and of $T \in (0,T_0]$, such that
   \[
    \norm{(\vec c, \vec c^\Sigma)}_{\fs E_p^\Omega(T) \times \fs E_p^\Sigma(T)}
     \leq C_{p,T_0} \big( \norm{\vec f}_{\fs F_p^\Omega(T)} + \norm{\vec f^\Sigma}_{\fs F_p^\Sigma(T)} + \norm{\vec g^\Sigma}_{\fs G_p^\Sigma(T)} + T^{-1/p} \norm{(\vec c^0, \vec c^{\Sigma,0})}_{\fs I_p^\Omega \times \fs I_p^\Sigma} \big).
   \]
  In particular, the operators $A_i: \dom(A_i) \subseteq \LL_p(\Omega) \rightarrow \LL_p(\Omega)$ and $A_i^\Sigma: \dom(A_i^\Sigma) \subseteq \LL_p(\Sigma) \rightarrow \LL_p(\Sigma)$ defined by
   \begin{alignat*}{2}
    A_i
     &= - \dv(d_i \nabla \cdot),
	\qquad &
     &\dom(A_i) = \{ c_i \in \WW^2_p(\Omega): \, - d_i \nabla c_i \cdot \vec \nu = 0 \, \text{ on } \partial \Omega \}
     \subseteq \LL_p(\Omega),
     \\
    A_i^\Sigma
     &= - \dv_\Sigma(d_i^\Sigma \nabla_\Sigma \cdot),
	\qquad &
     &\dom(A_i^\Sigma) = \WW^2_p(\Sigma)
     \subseteq \LL_p(\Sigma)
   \end{alignat*}
  generate bounded analytic and strongly continuous semigroups on $\LL_p(\Omega)$ and $\LL_p(\Sigma)$.
 \end{lemma}

As a consequence of Lemma~\ref{lem:maximal_regularity_scalar_case}, the block-diagonal operators
 \begin{align*}
  A
   = \diag(A_i)_{i=1}^N:
   &\dom(A)
   = \prod_{i=1}^N \dom(A_i)
   \subseteq \LL_p(\Omega;\R^N)
   \rightarrow \LL_p(\Omega;\R^N),
   \\
  A^\Sigma
   = \diag(A_i^\Sigma)_{i=1}^N:
   &\dom(A_i^\Sigma)
   = \prod_{i=1}^N \dom(A_i^\Sigma)
   \subseteq \LL_p(\Sigma;\R^N)
   \rightarrow \LL_p(\Sigma;\R^N)
 \end{align*}
generate bounded analytic $\CC_0$-semigroups on their respective underlying Banach spaces, for every $p \in (1, \infty)$.
 
\section{Assumptions on the Semilinear Problem}
\label{sec:Assmpt-on-Semilinear-Problem}

We list assumptions to be imposed on the reaction and sorption models and draw first conclusions, e.g.\ concerning properties of the associated Nemytskii-operators on $\LL_q$-spaces.
 
 \begin{assumption}[Regularity and Growth Bounds]
 \label{assmpt:general}
  Let $p \in (1, \infty)$ be fixed.
  We group the assumptions in those on the (macroscopic) velocity $\vec v$, on the surface reaction rates $r_i^\Sigma$, on the bulk reaction rates $r^\Omega_i$ and on the sorption model $s_i^\Sigma$.
   \begin{enumerate}
    \item
     Assumptions on the velocity field $\vec v$ of the fluid mixture\additionalremarks{It seems like the incompressibility and the no-flux boundary condition on the macroscopic velocity vector field $\vec v$ are crucial to ensure that the system is dissipative.
     Dieter: Ja, der Sinn einer Stoffmengenerhaltung. Aber nicht zwingend.} (regularity, incompressibility, no-flux boundary condition):
      \begin{equation}
       \tag{$A^{\mathrm{advec}}$}
       \label{A^vel}
       \vec v \in \fs U_{\tilde p}^\Omega(T)
        := \big\{ \vec u \in \WW^{(1,2)}_{\tilde p}(\Omega_T;\R^d): \, \dv 
        \vec u = 0 \text{ in } \Omega_T, \, \vec u \cdot \vec \nu = 0 \text{ on } \Sigma_T \big\},
      \end{equation}
     for some $\tilde p \in (d+2, \infty)$.
     Then, in particular, $\vec v \in \CC([0,T];\CC^1(\overline{\Omega};\R^d))$.
    \item
     Assumptions on the chemical surface reaction rates $\vec r^\Sigma$ (regularity, quasi-positivity and polynomial bounds):
       \begin{align}
        &r_i^\Sigma
         = r_i^\Sigma(\vec c^\Sigma)
         \quad \text{for some function} \quad
         r_i^\Sigma \in \Lip_\mathrm{loc}([0, \infty)^N; \R)
         \quad
         \text{for each }
         i = 1, \ldots, N.
         \tag{$A_F^{\mathrm{ch}}$}
         \label{A_F^ch}
         \\
        &r_i^\Sigma(\vec z) \geq 0
         \text{ for each } \vec z \in [0, \infty)^N \text{ with } z_i = 0,
          \quad
          i = 1, \ldots, N.
          \tag{$A_N^{\mathrm{ch}}$}
          \label{A_N^ch}
          \\
         \intertext{There are constants $M^\Sigma > 0$ and $\gamma^\Sigma \geq 1$ with $\begin{cases} \gamma^\Sigma \geq 1 \text{ arbitrary}, & \text{if } p \geq \tfrac{d+1}{2}, \\ \gamma^\Sigma \leq (1 - \tfrac{2p}{d+1})^{-1}, &\text{if } p \in (1, \tfrac{d+1}{2}) \end{cases}$ s.t.}
         &\abs{(\vec r^\Sigma)'(\vec y)}
          \leq M^\Sigma (1 + \abs{\vec y}^{\gamma^\Sigma -1}) \quad \text{for all } \vec y \in [0, \infty)^N.
          \tag{$A_P^{\mathrm{ch}}$}
          \label{A_P^ch}
       \end{align}
    \item
     Assumptions on the chemical bulk reaction rates $\vec r^\Omega$ (regularity, quasi-positivity and polynomial growth bounds):
      \begin{align}
       &r^\Omega_i
        = r^\Omega_i(\vec c)
        \quad \text{for some function} \quad
        r^\Omega_i \in \Lip_\mathrm{loc}([0,\infty)^N;\R)
        \text{ for each }
         i = 1, \ldots, N.
         \tag{${A_F^{\mathrm{ch}}}'$}
         \label{A_F^bulk}
         \\
       &r^\Omega_i(\vec z)
        \geq 0
        \quad
        \text{for each } \vec z \in [0, \infty)^N \text{ with } z_i = 0, \quad i = 1, \ldots, N.
        \tag{${A_N^{\mathrm{ch}}}'$}
        \label{A_N^bulk}
      \end{align}
      Moreover, there are constants $M^\Omega > 0$ and $\gamma^\Omega \geq 1$ with \[ \begin{cases} \gamma^\Omega \geq 1 \text{ arbitrary}, &\text{if } p \geq \frac{d+2}{2}, \\ \gamma^\Omega \leq (1 - \tfrac{2p}{d+2})^{-1}, &\text{if } p \in (1, \tfrac{d+2}{2}) \end{cases} \] such that
      \begin{equation}
       \abs{(\vec r^\Omega)'(\vec y)}
        \leq M^\Omega (1 + \abs{\vec y}^{\gamma^\Omega - 1})
        \quad
        \text{for all } \vec y \in [0, \infty)^N.
        \tag{${A_P^{\mathrm{ch}}}'$}
        \label{A_P^bulk}
      \end{equation}
    \item
     Assumptions on the \emph{ad- and desorption} rates $\vec r^{\mathrm{sorp}}$ (regularity, monotonicity and partial linear upper and lower bounds):
     for all $i = 1, \ldots, N$ it holds that
      \begin{alignat*}{2}
       s_i^\Sigma
        &= s_i^\Sigma(\vec c, \vec c^\Sigma)
        \quad \text{for some function }
        s_i^\Sigma \in \Lip_\mathrm{loc}(\R_+^{2N}; \R)
        \quad
        \text{ such that}
       \tag{${A_F^{\mathrm{sorp}}}$}
        \label{A_F^sorp}
        \\
         &\begin{cases}
          \abs{s_i^\Sigma(\vec{\check{a}}_i, \vec{\check{b}}_i)}
           \leq C \big( 1 + \abs{\vec{\check{a}}_i}^{K^\Omega} \abs{\vec{\check{b}}_i}^{K^\Sigma} \big),
           \\
          \abs{\partial_{a_i} s_i^\Sigma(\vec a, \vec b)}
           \leq C \big( 1 + \abs{\vec a}^{K^\Omega-1} \abs{\vec b}^{K^\Sigma} \big)
           &\text{and}
           \\
          \abs{\partial_{b_i} s_i^\Sigma(\vec a, \vec b)}
           \leq C \big( 1 + \abs{\vec a}^{K^\Omega} \abs{\vec b}^{K^\Sigma - 1} \big)
           &\text{for all } \vec a, \vec b \in \R_+^N,
         \end{cases}
        \tag{$A_P^\mathrm{sorp}$}
        \label{A_P^sorp}
        \\
        \intertext{for some $K^\Omega, K^\Sigma \in \N$ such that}
       d+1
        &\geq K^\Omega (d+1-p)^+  + K^\Sigma (d-p)^+
        \quad \text{and}
        \\
       p
        &\geq \min \left\{ \max \big\{ \frac{d+1}{2}, \frac{(K^\Omega - 1)^+}{(2K^\Omega - 1)^+} (d+2) \big\}, \frac{(K^\Omega + K^\Sigma - 1)(d+1) - 1}{2 (K^\Omega + K^\Sigma) - 1} \right\},
        \\
        \intertext{where we write}
         \vec{\check{v}}_i
          &:= (v_1, \ldots, v_{i-1}, 0, v_{i+1}, \ldots, v_N)
          \quad
         \text{for } \vec v \in \R^N \text{ and } i = 1, \ldots, N.
       \\
       s_i^\Sigma(\vec a, \vec b)
        &\text{ increases in the variable } a_i \text{ and decreases in the variable } b_i.
        \tag{$A_M^{\mathrm{sorp}}$}
        \label{A_M^sorp}
        \\
       &\intertext{There are constants $k_i^{\mathrm{de}}, k_i^{\mathrm{ad}} > 0$ such that}
        - k_i^{\mathrm{de}} (1 + \abs{\vec c^\Sigma})
         &\leq s_i^\Sigma(\vec c, \vec c^\Sigma)
         \leq k_i^{\mathrm{ad}} (1 + \abs{\vec c})
         \\
         s_i^\Sigma(\vec{\check{c}}_i, \vec c^\Sigma)
          &\leq 0,
          \quad
         s_i^\Sigma(\vec c, \vec{\check{c}}_i^\Sigma)
          \geq 0
         \quad
         \text{for all } \vec c, \vec c^\Sigma \in \R_+^N.
         \tag{$A_B^{\mathrm{sorp}}$}
         \label{A_B^sorp}
      \end{alignat*}
   \end{enumerate}
 \end{assumption}
 
For the particular case of a Langmuir sorption model, we may then identify integrability parameters $p \in (1, \infty)$ which fulfill the conditions stated in Assumption \ref{assmpt:general} on the sorption terms.
 
 \begin{example}
 \label{exa:Langmuir}
  The standard multi-component Langmuir sorption model
   \[
    s_i^\Sigma(\vec c, \vec c^\Sigma)
     := k_i^\mathrm{ad} c_i (1 - \theta)^+ - k_i^\mathrm{de} c_i^\Sigma,
     \quad
     \text{where }
     \theta := \frac{1}{c_S^\Sigma} \sum_{j=1}^N \sigma_j c_j^\Sigma
   \]
  for some $\sigma_j > 0$, $j = 1, \ldots, N$, and $c_S^\Sigma > 0$ fits into the class considered in assumption~\ref{assmpt:general} by choosing functions $f_1$, $f_2$ as follows:
   \begin{align*}
    \frac{\partial s_i^\Sigma(\vec c, \vec c^\Sigma)}{\partial c_i}
     &= k_i^\mathrm{ad} (1-\theta)^+
     \geq 0,
     \\
    \frac{\partial s_i^\Sigma(\vec c, \vec c^\Sigma)}{\partial c_i^\Sigma}
     &= - k_i^\mathrm{ad} c_i \chi_{[\theta \leq 1]} \frac{\sigma_i}{c_S^\Sigma} - k_i^\mathrm{de}
     \leq 0
     \quad
     \text{for all } (\vec c, \vec c^\Sigma) \in \R_+^{2N}.
   \end{align*}
  From the first term (where $K^\Omega = K^\Sigma = 1$), we obtain that the conditions on the sorption rate model in Assumption \ref{assmpt:general} are met, provided
   \[
    p
     \geq \max \{\frac{d}{2}, \frac{d}{3} \}
     = \frac{d}{2}.
   \]
 \end{example}
 
 \begin{example}[Further Sorption models]
  Several sorption rates resulting from different models for the surface chemical potential, cf.\ \cite[Table 7.4]{KrDaDe08}, fall in the class considered in the present paper, i.e.\ satisfy assumption~\ref{assmpt:general}:
   \begin{enumerate}
    \item
     Henry model: $s_i^\Sigma = k_i^\mathrm{ad} c_i - k_i^\mathrm{de} c_i^\Sigma$;
    \item
     Volmer model: $s_i^\Sigma = k_i^\mathrm{ad} c_i \big( 1 - \theta \big)^+ \exp \big( - \beta \frac{\theta}{1 - \theta} \big) - k_i^\mathrm{de} c_i^\Sigma$ for some $\beta > 0$;
    \item
     Frumkin model: $s_i^\Sigma = k_i^\mathrm{ad} c_i \theta^{\sigma_i} \exp \big( - \sigma_i \beta \theta \big) - k_i^\mathrm{de} c_i^\Sigma$ for some $\sigma_i \geq 1$ and $\beta > 0$;
    \item
     Van-der-Waals model: $s_i^\Sigma = k_i^\mathrm{ad} c_i \exp \big( - \beta \theta \big) \exp \big( - \sigma_i \frac{\theta}{1 - \theta} \big) - k_i^\mathrm{de} c_i^\Sigma$ for some $\beta, \sigma_i > 0$,
   \end{enumerate}
  where, whenever $\theta$ appears, it is defined as $\theta := \frac{1}{c_S^\Sigma} \sum_{j=1}^N \sigma_j c_j^\Sigma$.
  In all of these models we may choose $K^\Omega = K^\Sigma = 1$, so that from Assumption \ref{assmpt:general} we get the constraint $p \geq \frac{d}{2}$.
  Also note that all of these models have the general form $s_i^\Sigma = s_i^\Sigma(c_i, c_i^\Sigma, \theta)$, so that for checking the monotonicity condition it suffices to check that
   \[
    \frac{\partial s_i^\Sigma(c_i, c_i^\Sigma, \theta)}{\partial c_i}
     \geq 0,
      \quad
    \frac{\partial s_i^\Sigma(c_i, c_i^\Sigma, \theta)}{\partial c_i^\Sigma}
     \leq 0
     \quad \text{and} \quad
     \frac{\partial s_i^\Sigma(c_i, c_i^\Sigma, \theta)}{\partial \theta}
      \leq 0
   \]
  since $\frac{\partial \theta}{\partial c_i^\Sigma} \geq 0$.
 \end{example}
 
 \begin{remark}
  Note that, by the (weak form of the) fundamental theorem of calculus, the polynomial growth conditions on the derivatives of $\vec r^\Omega$, $\vec r^\Sigma$ and $\vec s^\Sigma$ in \eqref{A_P^ch}, \eqref{A_P^bulk} and \eqref{A_P^sorp} directly imply polynomial growth bounds on the functions themselves.
  For example, for the surface chemical reaction rates, we obtain
   \begin{align*}
    \abs{\vec r^\Sigma(\vec y)}
     &= \abs{\vec r^\Sigma(\vec 0) + \int_0^1 \vec y \cdot \nabla \vec r^\Sigma(s \vec y) \dd s}
     \leq \abs{\vec r^\Sigma(0)} + \abs{\vec y} \int_0^1 \abs{\nabla \vec r^\Sigma(s \vec y)} \dd s
     \\
     &\leq \abs{\vec r^\Sigma(\vec 0)} + \abs{\vec y} \int_0^1 M^\Sigma (1 + s \abs{\vec y})^{\gamma^\Sigma - 1} \dd s
     \leq \abs{\vec r^\Sigma(\vec 0)} + \frac{M^\Sigma}{\gamma^\Sigma} (1 + \abs{\vec y})^{\gamma^\Sigma}
   \end{align*}
  for all $\vec y \in [0,\infty)^N$, and, similarly, for the sorption rates
   \begin{align*}
    \abs{s_i^\Sigma(\vec c, \vec c^\Sigma)}
     &= \abs{s_i^\Sigma(\vec{\check{c}}_i, \vec{\check{c}}_i^\Sigma)
      + \int_0^{c_i} \partial_{a_i} s_i^\Sigma(\vec{\check{c}}_i + a_i \vec{e}_i, \vec{\check{c}}_i^\Sigma) \, \dd a_i
      + \int_0^{c_i^\Sigma} \partial_b s_i^\Sigma(\vec c, \vec{\check{c}}_i^\Sigma + b_i \vec{e}_i) \, \dd b_i}
     \\
    &\leq C \big( 1 + \abs{\vec c}^{K^\Omega} \abs{\vec c^\Sigma}^{K^\Sigma} \big).
   \end{align*}
 \end{remark}
 
Absolutely continuous functions with polynomial growth bounds on their derivatives can always be written as a product of Lipschitz continuous functions, as the following auxiliary result shows.

 \begin{lemma}
 \label{lem:polynomial_bds_imply_structure}
   Let $\psi \in \WW^1_{1,\mathrm{loc}}(\R^m;\R)$ be polynomially bounded with polynomially bounded derivatives, i.e.\ there are $\alpha_1, \ldots, \alpha_m \geq 0$ such that
   \begin{equation}
    \abs{\psi(\vec s)}, \abs{\partial_i \psi(\vec s)}
     \leq C \prod_{j=1}^m (1 + s_j^2)^{\alpha_j/2}
     \quad
     \text{for all } \vec s \in \R^m, \, i = 1, \ldots, m
     \label{eqn:lemma_polynomial-bound}
   \end{equation}
  for some $C > 0$.
  Then $\psi$ has the representation
   \begin{equation}
    \psi(\vec s)
     = \psi_0(\vec s) \prod_{j=1}^m (1 + s_j^2)^{\alpha_j/2}
     \quad
     \text{for all } \vec s \in \R^m,
    \label{eqn:representation_psi}
   \end{equation}
  where $\psi_0$ is bounded and globally Lipschitz-continuous.
 \end{lemma}
 
 \begin{remark}
  Note that, on the other hand, every function of the form \eqref{eqn:representation_psi} is polynomially bounded in the sense of equation \eqref{eqn:lemma_polynomial-bound}.
 \end{remark}

 \begin{proof}[Proof of Lemma \ref{lem:polynomial_bds_imply_structure}]
 We have to verify that
   \[
    \psi_0(\vec s)
     := \psi(\vec s) \cdot \prod_{j=1}^m (1 + s_j^2)^{-\alpha_j/2},
     \quad
     \vec s \in \R^m,
   \]
  is a Lipschitz-continuous function.
  Since $\psi$ is weakly differentiable and each of the functions $(1 + s_j^2)^{-\alpha_j/2}$ is classically continuously differentiable, $\psi_0$ is weakly differentiable, so that it remains to check that its derivatives are uniformly bounded.
  In fact,
   \[
    \abs{\partial_i \psi_0(\vec s)}
     = \abs{\partial_i \psi(\vec s) \cdot \prod_{j=1}^m (1 + s_j^2)^{\alpha_j/2}
      - \sum_{i=1}^m \frac{\alpha_i s_i}{(1 + s_i^2)^{1/2}} \psi(\vec s) \prod_{j=1}^m (1 + s_j^2)^{\alpha_j/2}}
     \leq C (m + 1)
     < \infty
   \]
  for all $\vec s \in \R^m$.
 \end{proof}

Similarly, we may conclude the following structural property of the sorption rates $s_i^\Sigma$.

 \begin{corollary}
  \label{cor:sorption_structure}
  Each of the sorption terms $s_i^\Sigma$, $i = 1, \ldots, N$, may be written as a sum of functions $f$ of the form
   \begin{equation}
    f(\vec u, \vec u^\Sigma)
     = \psi_0(\vec u, \vec u^\Sigma) \prod_{k=1}^{K^\Omega} \psi_k(u_i) \cdot \prod_{k=1}^{K^\Sigma} \psi^\Sigma_k(u_i^\Sigma),
    \label{A_S^sorp}
    \tag{$A_S^{\mathrm{sorp}}$}
   \end{equation}
  where each of the functions $\psi_k$, $\psi_k^\Sigma: \R \rightarrow \R$ is  Lipschitz continuous and $\psi_0: \R^2 \rightarrow \R$ is bounded and Lipschitz continuous.
 \end{corollary}

Thanks to the polynomial growth bounds on the weak derivatives of $\vec r^\Omega$ and $\vec r^\Sigma$, these functions act as continuous Nemytskii-operators between suitable $\LL_q$-spaces.
 
 \begin{remark}[Reaction terms acting as Nemytskii Operators]
  \label{rem:Nemytskii_Operators}
  By the polynomial growth bounds \eqref{A_P^ch} and \eqref{A_P^bulk} imposed on the bulk and surface chemistry reaction rate models, the Nemytskii operators corresponding to the functions $\vec r^\Sigma$ and $\vec r^\Omega$ and their derivatives induce nonlinear maps as follows:
  For any $p \in [1,\infty)$, we have
   \begin{align*}
    \vec r^\Sigma: \LL_{p \gamma^\Sigma}(\Sigma_T; \R^N) &\rightarrow \LL_p(\Sigma_T; \R^N),
     \\
    (\vec r^\Sigma)': \LL_{p \gamma^\Sigma}(\Sigma_T; \R^N) &\rightarrow \LL_{p \gamma^\Sigma / (\gamma^\Sigma - 1)}(\Sigma_T; \R^{N \times N}),
     \\
    \vec r^\Omega: \LL_{p \gamma^\Omega}(\Omega_T; \R^N) &\rightarrow \LL_p(\Omega; \R^N),
     \\
    (\vec r^\Omega)': \LL_{p \gamma^\Omega}(\Omega_T; \R^N) &\rightarrow \LL_{p \gamma^\Omega / (\gamma^\Omega - 1)}(\Omega_T; \R^{N \times N}).
   \end{align*}
  Moreover, for each $\delta, T > 0$, there are constants $k^\Sigma(\delta,T), k^\Omega(\delta,T) > 0$ such that
   \begin{align*}
     \norm{({\vec r^\Sigma)}'(\vec u^\Sigma)}_{\LL_p(\Sigma_T)}
      &\leq k^\Sigma(\delta,T)
	  \\
      \intertext{for all $\vec u^\Sigma \in \LL_{p \gamma^\Sigma/(\gamma^\Sigma - 1)}(\Sigma_T;\R^N)$ with $\norm{\vec u^\Sigma}_{\LL_{p\gamma^\Sigma /(\gamma^\Sigma - 1)}(\Sigma_T)} \leq \delta$,}
     \norm{({\vec r^\Omega})'(\vec u)}_{\LL_p(\Omega_T)}
      &\leq k^\Omega(\delta,T)
	  \\
      \intertext{for all $\vec u \in \LL_{p\gamma^\Omega/(\gamma^\Omega-1)}(\Omega_T;\R^N)$ with $\norm{\vec u}_{\LL_{p\gamma^\Omega/(\gamma^\Omega-1)}(\Omega_T)} \leq \delta$.}
   \end{align*}
  Hence,
   \begin{align*}
    &\norm{\vec r^\Sigma(\vec c^\Sigma) - \vec r^\Sigma(\vec z^\Sigma)}_{\LL_p(\Sigma_T;\R^N)}
     \\
     &\leq \sup_{\vec v^\Sigma \in B_\delta(0)} \norm{ (\vec r^\Sigma)'(\vec v^\Sigma)}_{\LL_{p \gamma^\Sigma / (\gamma^\Sigma - 1)}(\Sigma_T; \R^{N \times N})} \norm{\vec c^\Sigma - \vec z^\Sigma}_{\LL_{p \gamma^\Sigma}(\Sigma_T; \R^N)}
     \\
     &\leq k^\Sigma(\delta,T) \norm{\vec c^\Sigma - \vec z^\Sigma}_{\LL_{p \gamma^\Sigma}(\Sigma_T; \R^N)},
     \quad
     \vec c^\Sigma, \vec z^\Sigma \in B_\delta(0) \subseteq \LL_{p\gamma^\Sigma}(\Sigma_T; \R^N),
     \\
     \intertext{and}
    &\norm{\vec r^\Omega(\vec c) - \vec r^\Omega(\vec z)}_{\LL_p(\Omega_T;\R^N)}
     \\
     &\leq \sup_{\vec v \in B_\delta(0)} \norm{ (\vec r^\Omega)'(\vec v)}_{\LL_{p \gamma^\Omega / (\gamma^\Omega - 1)}(\Omega_T; \R^{N \times N})} \norm{\vec c - \vec z}_{\LL_{p \gamma^\Omega}(\Omega_T; \R^N)}
     \\
     &\leq k^\Omega(\delta,T) \norm{\vec c - \vec z}_{\LL_{p \gamma^\Omega}(\Omega_T; \R^N)},
     \quad
     \vec c, \vec z \in B_\delta(0) \subseteq \LL_{p\gamma^\Omega}(\Omega_T; \R^N),
   \end{align*}
  i.e.\ the Nemytskii operators associated to $\vec r^\Sigma$ and $\vec r^\Omega$ are Lipschitz continuous on each bounded subset of these spaces, respectively.
 \end{remark}
 \begin{proof}
 Let us consider the model for the surface reaction rates $\vec r^\Sigma$. Properties of the model for the bulk reaction rates $\vec r^\Omega$ can be derived analogously.
 Observe that for every $\vec c^\Sigma \in \LL_{p \gamma^\Sigma}(\Sigma_T;\R^N)$, by using the polynomial boundedness of $\vec r^\Sigma$, we may estimate
  \begin{align*}
   \norm{\vec r^\Sigma(\vec c^\Sigma)}_{\LL_p(\Sigma_T)}^p
    &= \int_{\Sigma_T} \abs{\vec r^\Sigma(\vec c^\Sigma)}^p \dd(t,\sigma(\vec x))
    \leq M^\Sigma \int_{\Sigma_T} \big( 1 + \abs{\vec c^\Sigma}^{\gamma^\Sigma} \big)^p \dd(t,\sigma(\vec x))
    \\
    &\leq 2^p M^\Sigma \int_{\Sigma_T} \big(1 + \abs{\vec c^\Sigma}^{p\gamma^\Sigma} \big) \dd(t,\sigma(\vec x))
    = 2^p M^\Sigma \big( T \sigma(\Sigma) + \norm{\vec c^\Sigma}_{\LL_{p\gamma^\Sigma}(\Sigma_T)}^{p\gamma^\Sigma} \big),
  \end{align*}
 where $\sigma(\Sigma)$ denotes the (finite) surface measure of $\Sigma$.
 This shows that the Nemytskii operator associated to $\vec r^\Sigma$ maps $\LL^{p\gamma^\Sigma}(\Sigma_T;\R^N)$ into $\LL_p(\Sigma_T;\R^N)$.
 Similarly,
  \begin{align*}
   \norm{(\vec r^\Sigma)'(\vec c^\Sigma)}_{\LL_{\nicefrac{p\gamma^\Sigma}{\gamma^\Sigma - 1}}}^{\nicefrac{p\gamma^\Sigma}{\gamma^\Sigma - 1}}
    &= \int_{\Sigma_T} \abs{(\vec r^\Sigma)'(\vec c^\Sigma)}^{\nicefrac{p\gamma^\Sigma}{\gamma^\Sigma-1}} \dd(t,\sigma(\vec x))
    \leq M^\Sigma \int_{\Sigma_T} \big( 1 + \abs{\vec c^\Sigma} \big)^{p\gamma^\Sigma} \dd(t,\sigma(\vec x))
    \\
    &\leq 2^p M^\Sigma \big( T \sigma(\Sigma) + \norm{\vec c}_{\LL_{p\gamma^\Sigma}}^{p\gamma^\Sigma} \big),
  \end{align*}
 so that $(\vec r^\Sigma)'$ has the claimed mapping property.
 From here, the remaining estimates follow easily by the fundamental theorem of calculus, Tonelli's theorem and Hölder's inequality:
  \begin{align*}
   &\norm{\vec r^\Sigma(\vec c^\Sigma) - \vec r^\Sigma(\vec {\tilde c}^\Sigma)}_{\LL_p(\Sigma_T;\R^N)}^p
    = \int_{\Sigma_T} \abs{\vec r^\Sigma(\vec c^\Sigma) - \vec r^\Sigma(\vec {\tilde c}^\Sigma)}^p \dd(t,\sigma(\vec x))
    \\
    &= \int_{\Sigma_T} \abs{ \int_0^1 (\vec r^\Sigma)'(s \vec c^\Sigma + (1-s) \vec {\tilde c}^\Sigma) \cdot (\vec c^\Sigma - \vec {\tilde c}^\Sigma) \dd s }^p \dd(t,\sigma(\vec x))
    \\
    &\leq \int_{\Sigma_T} \int_0^1 \abs{(\vec r^\Sigma)'(s \vec c^\Sigma + (1-s) \vec {\tilde c}^\Sigma) \cdot (\vec c^\Sigma - \vec {\tilde c}^\Sigma)}^p \dd s \dd (t, \sigma(\vec x))
    \\
    &= \int_0^1 \int_{\Sigma_T} \abs{(\vec r^\Sigma)'(s \vec c^\Sigma + (1-s) \vec {\tilde c}^\Sigma) \cdot (\vec c^\Sigma - \vec {\tilde c}^\Sigma)}^p \dd (t, \sigma(\vec x)) \dd s
    \\
    &\leq \int_0^1 \norm{\vec c - \vec {\tilde c}}_{\LL_{p\gamma^\Sigma}(\Sigma_T;\R^N)}^p \norm{(\vec r^\Sigma)'(s \vec c + (1-s) \vec {\tilde c})}_{\LL_{\nicefrac{p\gamma^\Sigma}{\gamma^\Sigma-1}}(\Sigma_T;\R^{N \times N})}^p \dd s
    \\
    &\leq \big( k^\Sigma(\delta,T) \norm{\vec c - \vec {\tilde c}}_{\LL_{p\gamma^\Sigma}(\Sigma_T;\R^N)} \big)^p.
  \end{align*}
 \end{proof}
 
\section{Positivity of solutions}
\label{sec:Positivity}

In this section, we are concerned with an aspect which is quite important from the modelling perspective:
We show that strong solutions to the reaction-diffusion-advection-sorption system
   \begin{alignat*}{2}
    \partial_t c_i + \vec v \cdot \nabla c_i - \dv (d_i \nabla c_i)
     &= r^\Omega_i(\vec c)
	 \qquad &
     &\text{on } (0,T) \times \Omega,
     \\
    \partial_t c_i^\Sigma - \dv_\Sigma (d_i^\Sigma \nabla_\Sigma c_i^\Sigma)
     &= s_i^\Sigma(c_i, c_i^\Sigma) + r_i^\Sigma(\vec c^\Sigma)
     \qquad
     &&\text{on } (0,T) \times \Sigma,
     \tag{RDASS}
     \label{eqn:RDASS}
     \\
    - d_i \nabla c_i \cdot \nu
     &= s_i^\Sigma(c_i, c_i^\Sigma)
     \qquad &
     &\text{on } (0,T) \times \Sigma,
     \\
    (c_i(0,\cdot), c_i^\Sigma(0,\cdot))
     &= (c^0_i, c^{\Sigma,0}_i)
     \qquad &
     &\text{for } i = 1, \ldots, N
   \end{alignat*}
  will be positive, provided the initial data are positive.
  Let us start with fixing the solution concept used in this manuscript.
 
 \begin{definition}
  Let $p \in [1, \infty)$.
  A pair $(\vec c, \vec c^\Sigma)$ of functions $\vec c \in \WW_p^{(1,2)}(\Omega_T;\R^N)$ and $\vec c^\Sigma \in \WW_p^{(1,2)}(\Sigma_T;\R^N)$ is called \emph{strong $\WW_p^{(1,2)}$--solution} of \eqref{eqn:RDASS}, if $(\vec c, \vec c^\Sigma)|_{t=0} = (\vec c^0, \vec c^{\Sigma,0})$ in the trace sense and the differential equations in \eqref{eqn:RDASS} are satisfied for a.e.\ $(t, \vec x) \in (0,T) \times \Omega$ and $(t, \vec x) \in (0,T) \times \Sigma$, respectively.
 \end{definition}
 
 \begin{remark}
  Note that in related papers different solution concepts have been used, e.g.\ \emph{classical solutions} in \cite{MorTan22+}, where the authors demand that
   \begin{align*}
    (\vec c, \vec c^\Sigma)
     &\in \CC([0,T]; \LL_p(\Omega;\R^N) \times \LL_p(\Sigma; \R^N))
      \cap \LL_\infty( (0,T); \LL_\infty(\Omega;\R^N) \times \LL_\infty(\Sigma;\R^N))
      \\
      &\quad
      \cap \big( \CC^{(1,2)}( (0,T) \times \Omega; \R^N) \times \CC^{(1,2)}( (0,T) \times \Sigma; \R^N) \big).
   \end{align*}
  Away from zero, they thus demand more regularity (classical instead of weak derivatives), whereas close the initial time $t = 0$ they only demand continuity w.r.t.\ values in $\LL_p$, but the strong solutions will at least be continuous w.r.t.\ values in the trace space $\WW_p^{2-2/p}$.
 \end{remark}
 
 We have not yet shown that such strong solutions exist, but we nevertheless can give a first result on positivity of solutions.

 \begin{lemma}[Positive solutions]
  Let $p \in (1, \infty)$ and $T > 0$ be given.
  Let assumption~\ref{assmpt:general} hold true.
  If all components of the initial data $(\vec c^0, \vec c^{\Sigma,0}) \in \fs I_p^{\Omega,+} \times \fs I_p^{\Sigma,+}$ are (a.e.) non-negative and $(\vec c, \vec c^\Sigma) \in \fs E_p^\Omega(T) \times \fs E_p^\Sigma(T)$ form a strong $\WW_p^{(1,2)}$-solution of the reaction-diffusion-advection-sorption system \eqref{eqn:RDASS}, then $c_i \geq 0$ and $c_i^\Sigma \geq 0$ a.e.\ on $(0,T) \times \Omega$ and $(0, T) \times \Sigma$, resp.
 \end{lemma}
 
 \begin{proof}
  A similar result can be found in \cite[Lemma 5.2]{BKMS17}, where, however, a more restrictive regularity condition than \eqref{A_F^sorp} is demanded of the sorption model.
  Nevertheless, the proof of \cite[Lemma 5.2]{BKMS17} essentially carries over, except for the following steps, where we have to employ the adjusted condition \eqref{A_F^sorp}.\newline
  The structural and regularity assumptions on $s_i^\Sigma$ can be used to prove that
   \[
    \limsup_{\varepsilon \rightarrow 0+} \int_\Sigma (\phi_\varepsilon'(c_i^\Sigma) - \phi_\varepsilon'(c_i)) s_i^\Sigma(\vec c, \vec c^\Sigma) \dd \sigma
     \leq 0,
   \]
  where $\phi_\varepsilon(r) := - r \ee^{\varepsilon / r} \chi_{[r \leq 0]}$ is an $\varepsilon$-regularised ($\varepsilon > 0$) version of the function $- r \chi_{[r  \leq 0]}$ and has the derivative $\phi_\varepsilon'(r) = (\tfrac{\varepsilon}{r} - 1) \ee^{\varepsilon / r} \chi_{[r < 0]}$.
  In particular, its derivative $\phi_\varepsilon'$ is monotone and $\phi_\varepsilon'(r) \rightarrow - \chi_{[r < 0]}$ converges pointwise as $\varepsilon \rightarrow 0+$.
  For any fixed time $t \in (0,T)$ and representative functions $\vec c$ and $\vec c^\Sigma$, we decompose the surface $\Sigma$ into the three pairwise disjoint sets $[\sign c_i(t,\cdot) = \sign c_i^\Sigma(t,\cdot)]$, $[c_i(t,\cdot) < 0 \leq c_i^\Sigma(t,\cdot)]$ and $[c_i^\Sigma(t,\cdot) < 0 \leq c_i(t,\cdot)]$, where we set $\sign(s) = 1$ if $s \geq 0$ and $\sign(s) = -1$ otherwise.
  Recall that, by assumption~\ref{assmpt:general}, $s_i^\Sigma(\vec a, \vec b)$ is increasing in the variable $a_i$ and decreasing in the variable $b_i$, respectively.
  We also recall the notation $\vec{\check{v}}_i := (v_1, \ldots, v_{i-1}, 0, v_{i+1}, \ldots, v_N)$ for vectors $\vec v \in \R^N$ and $i = 1, \ldots, N$.
  Employing Lebesgue's theorem on dominated convergences (and using that $\norm{\phi_\varepsilon'}_\infty = \lim_{s \rightarrow -\infty} \abs{\phi_\varepsilon'(s)} = 1$ for all $\varepsilon > 0$), we obtain by letting $\varepsilon > 0$ tend to zero that
   \begin{align*}
    &\int_{[\sign c_i(t,\cdot) = \sign c_i^\Sigma(t,\cdot)]} (\phi_\varepsilon'(\vec c_i^\Sigma) - \phi_\varepsilon'(c_i)) s_i^\Sigma(\vec c, \vec c^\Sigma)(t,\vec x) \dd \sigma(\vec x)
     \xrightarrow{\varepsilon \rightarrow 0+} 0,
     \\
    &\int_{[c_i(t,\cdot) < 0 \leq c_i^\Sigma(t,\cdot)]} (\phi_\varepsilon'(c_i^\Sigma(t,\vec x)) - \phi_\varepsilon'(c_i(t,\vec x))) s_i^\Sigma(\vec c(t,\vec x), \vec c^\Sigma(t,\vec x)) \dd \sigma(\vec x)
     \\
     &\quad
     \xrightarrow{\varepsilon \rightarrow 0+} \int_{[c_i(t,\cdot) < 0 \leq c_i^\Sigma(t,\cdot)]} s_i^\Sigma(\vec c(t,\vec x), \vec c^\Sigma(t,\vec x)) \dd \sigma(\vec x)
     \\
     &\quad
     \leq \int_{[c_i(t,\cdot) < 0 \leq c_i^\Sigma(t,\cdot)]} s_i^\Sigma(\vec{\check{c}}_i(t,\vec x), \vec c^\Sigma(t,\vec x)) \dd \sigma(\vec x)
     \leq 0,
     \\
    &\int_{[c_i^\Sigma(t,\cdot) < 0 \leq c_i(t,\cdot)]} (\phi_\varepsilon'(c_i^\Sigma(t,\vec x)) - \phi_\varepsilon'(c_i(t,\vec x))) s_i^\Sigma(c_i(t,\vec x), c_i^\Sigma(t,\vec x)) \dd \sigma(\vec x)
     \\
     &\quad
     \xrightarrow{\varepsilon \rightarrow 0+} - \int_{[c_i^\Sigma(t,\cdot) < 0 \leq c_i(t,\cdot)]} s_i^\Sigma(\vec c(t,\vec x), \vec c^\Sigma(t,\vec x)) \dd \sigma(\vec x)
     \\
     &\quad
     \leq - \int_{[c_i(t,\cdot) < 0 \leq c_i^\Sigma(t,\cdot)]} s_i^\Sigma(\vec c(t,\vec x), \vec{\check{c}}_i^\Sigma(t,\vec x)) \dd \sigma(\vec x)
     \leq 0.
   \end{align*}
  Therefore, the proof of \cite[Lemma 5.2]{BKMS17} almost literally extends to the situation considered here, provided we neglect bulk chemistry, i.e.\ $\vec r^\Omega(\vec c) = \vec 0$.
  In the case $\vec r^\Omega \not= \vec 0$, the same techniques as have been used for the $\vec r^\Sigma$-term can be employed to show that for the additional terms we have the estimate
   \[
    \limsup_{\varepsilon \rightarrow 0+} \int_\Omega \phi_\varepsilon'(c_i) r^\Omega_i(\vec c) \dd \vec x
     \leq 0.
   \]
Therefore, strong solutions to positive initial data are positive.
 \end{proof}

\section{Local-in-time well-posedness}
\label{sec:LIT-well-posedness}

Let $p \in (1,\infty)$, initial data $(\vec c^0, \vec c^{\Sigma,0}) \in \fs I_p^\Omega \times \fs I_p^\Sigma$ and a finite time horizon $T > 0$ be given.
We now consider local-in-time well-posedness for the initial-boundary value problem \eqref{eqn:RDASS}.
To do so, we formally rewrite \eqref{eqn:RDASS} as an abstract semilinear evolution equation of the form
 \begin{equation}
  \begin{cases}
  \mathcal{L}_{T,i}(c_i,c_i^\Sigma)
   &= \mathcal{N}_{T,i}(\vec c, \vec c^\Sigma)
    \quad\text{subject to the constraint}
   \\
  (c_i(0,\cdot), c_i^\Sigma(0, \cdot))
   &= (c^0_i, c^{\Sigma,0}_i)
   \quad
   \text{for } i = 1, \ldots, N,
   \end{cases}
   \label{eqn:ASLP}
 \end{equation}
where the maps $\mathcal{L}_{T,i}$ (linear part) and $\mathcal{N}_{T,i}$ (nonlinear part) are defined as follows:
 \begin{align*}
  \mathcal{L}_T:
  \fs E_p^\Omega(T) \times \fs E_p^\Sigma(T)
   &\rightarrow \fs F_p^\Omega(T) \times \fs F_p^\Sigma(T) \times \fs G_p^\Sigma(T),
   \\
  (\vec c, \vec c^\Sigma)
   &\mapsto (\partial_t c_i + \vec v \cdot \nabla c_i - \dv (d_i \nabla c_i), \partial_t c_i^\Sigma - \dv_\Sigma (d_i^\Sigma \nabla_\Sigma c_i^\Sigma), - d_i \nabla c_i|_\Sigma \cdot \vec \nu)_{i=1,\ldots,N},
   \\
  \mathcal{N}_T:
  \fs E_p^\Omega(T) \times \fs E_p^\Sigma(T)
   &\rightarrow \fs F_p^\Omega(T) \times \fs F_p^\Sigma(T) \times \fs G_p^\Sigma(T),
    \\
  (\vec c, \vec c^\Sigma)
   &\mapsto (\vec r^\Omega(\vec c), \vec s^\Sigma(\vec c|_\Sigma, \vec c^\Sigma) + \vec r^\Sigma(\vec c^\Sigma), \vec s^\Sigma(\vec c|_\Sigma, \vec c^\Sigma)).
 \end{align*}
 Observe that $\mathcal{L}_T$ combines the action of the linear diffusion-advection operators in the bulk phase and on the surface, as well as the Neumann trace map, corresponding to the non-reactive case, whereas $\mathcal{N}_T$ contains the non-linear contributions from bulk and surface chemical reactions and the sorption near the boundary.

As a first step, we reduce this semilinear problem with inhomogeneous initial data to a zero time-trace problem, i.e.\ to the special case where $(c^0_i|_\Sigma, c^{\Sigma,0}_i) = 0$ and $s_i^\Sigma(c^0_i|_\Sigma, c^{\Sigma,0}_i) = 0$.
To this end, we solve the quasi-autonomous problems
 \[
  \begin{cases}
  \mathcal{L}_{T,i}(c_i^\ast, c_i^{\Sigma,\ast})
   &= (0, 0, s_i^\ast)
   \\
  (c_i(0,\cdot), c_i^\Sigma(0,\cdot))
   &= (c^0_i, c^{\Sigma,0}_i),
  \end{cases}
  \quad
  i = 1, \ldots, N,
 \]
where the function $s_i^\ast$ on the right-hand side,
 \begin{equation}
  s_i^\ast(t,\cdot)
   = \mathcal{R}_{\R^d \rightarrow \Sigma} \ee^{t A_i^\Sigma} \mathcal{E}_{\Sigma \rightarrow \R^d} s_i^\Sigma(\vec c^0, \vec c^{\Sigma,0}),
   \quad
   t > 0,
   \label{eqn:riast}
 \end{equation}
is defined via the analytic semigroup $(\ee^{tA_i^\Sigma})_{t \geq 0}$ generated by the scalar differential operator $A_i^\Sigma$ on the surface, cf.\ Lemma~\ref{lem:maximal_regularity_scalar_case}, and where $\Ecal_{\Sigma \rightarrow \R^d} \in \B(\WW^{1-\nicefrac{3}{p}}_p(\Sigma); \WW^{1-\nicefrac{2}{p}}_p(\R^d))$ is some fixed continuous extension operator from $\Sigma$ to $\R^d$ and $\mathcal{R}_{\R^d \rightarrow \Sigma}$ is the surface trace operator which can be considered as an operator in the class $\B(\WW^{(1,2) \cdot \frac{1}{2}}_p(\R^d_T); \WW^{(1,2)\cdot(\frac{1}{2}-\frac{1}{2p})}_p(\Sigma_T))$ of bounded linear operators between the spaces $\WW^{(1,2) \cdot \frac{1}{2}}_p(\R^d_T)$ and $\WW^{(1,2)\cdot(\frac{1}{2}-\frac{1}{2p})}_p(\Sigma_T))$, cf.\ \cite{Rychkov_1999}.
For each component $i = 1, \ldots, N$, the elliptic surface differential operator in divergence form, $A_i^\Sigma = \dv_\Sigma (d_i^\Sigma \nabla_\Sigma \cdot)$, defined on the domain $\dom(A_i^\Sigma) = \WW^2_p(\Sigma)$ generates an analytic semigroup on $\LL_p(\Sigma)$.
Moreover, by $\LL_p$-maximal regularity of the corresponding parabolic initial-boundary value problem, it follows that $\vec s^\ast \in \fs G_p^\Sigma(T) = \WW^{(1,2)\cdot(\frac{1}{2}-\frac{1}{2p})}_p(\Sigma_T;\R^N)$ if and only if
 \[
  s_i^\Sigma(\vec c^0|_\Sigma, \vec c^{\Sigma,0})
   \in \WW^{1-\nicefrac{3}{p}}_p(\Sigma)
   \quad
   \text{for } i = 1, \ldots, N.
 \]
Indeed, if $s_i^\Sigma(\vec c^0|_\Sigma, \vec c^{\Sigma,0}) \in \WW_p^{1-3/p}(\Sigma)$, then $(I - A_i^\Sigma)^{-(2p+1)/2p} s_i^\Sigma(\vec c^0|_\Sigma, \vec c^{\Sigma,0}) \in \WW_p^{2-2/p}(\Sigma)$.
Hence,
 \begin{align*}
  \big[ \big( (t, \vec x)
  &\mapsto \ee^{t A_i^\Sigma} (I - A_i^\Sigma)^{-(2p+1)/2p} s_i^\Sigma(\vec c^0|_\Sigma, \vec c^{\Sigma,0})
   \\
   &= (I - A_i^\Sigma)^{-(2p+1)/2p} \ee^{t A_i^\Sigma} s_i^\Sigma(\vec c^0|_\Sigma, \vec c^{\Sigma,0}) \big) \big] \in \WW_p^{(1,2)}(\Sigma_T)
 \end{align*}
 for every $T > 0$, and in this case $\big( t \mapsto \ee^{t A_i^\Sigma} s_i^\Sigma(\vec c^0|_\Sigma, \vec c^{\Sigma,0}) \big) \in \WW_p^{(1,2) \cdot (\frac{1}{2} - \frac{1}{2p})}(\Sigma_T)$.
 \newline
In our case, we are given initial data $(\vec c^0|_\Sigma, \vec c^{\Sigma,0}) \in \WW^{2-\nicefrac{3}{p}}_p(\Sigma;\R^N) \times \WW^{2-\nicefrac{2}{p}}_p(\Sigma;\R^N)$ and the general sorption model is of class $s_i^\Sigma \in \WW^1_{\infty,\mathrm{loc}}(\R^{2N}; \R^N)$.
Unfortunately this may, in general, not be enough to ensure that $s_i^\Sigma(\vec c^0|_\Sigma, \vec c^{\Sigma,0}) \in \WW^{1-\nicefrac{3}{p}}_p(\Sigma)$.
  \begin{remark}
   In general, the continuous embedding $\WW_p^{2-2/p}(\Omega) \hookrightarrow \LL_q(\Sigma)$ (by which we mean that the trace operator is a continuous linear operator between the first and the latter space) is only true for $q \in [1,q^\ast)$ below some critical value $q^\ast < \infty$, and, thus, the function $(u|_\Sigma)^k$ will in general not be in the class $\LL_p(\Sigma)$ (choose, e.g., $k \in \N$ such that $\frac{1}{p} < \frac{k}{q*}$).
   Therefore, without further structural conditions we cannot infer from $\vec u \in \WW^{2-2/p}_p(\Omega; \R^N)$, $\vec u^\Sigma \in \WW_p^{2-2/p}(\Sigma; \R^N)$ and $S \in \WW^1_{\infty,\mathrm{loc}}(\R^{2N})$ that $S(\vec u|_\Sigma, \vec u^\Sigma) \in \LL_p(\Sigma)$, or, even, $S(\vec u|_\Sigma, \vec u^\Sigma) \in \WW^{1-3/p}_p(\Sigma)$.
 \newline
Also note that in the critical case $p \leq \frac{d+2}{2}$ the space $\WW_p^{2-2/p}$ will not embed into $\LL_\infty$, in general, in contrast to the -- in this regard -- simpler situation, e.g.\ in \cite{MorTan22+}, where always $p > d$, so that in particular $\WW_p^{2-2/p}(\Omega)$ continuously embeds into some Hölder space $\CC^\alpha(\overline{\Omega})$ for certain $\alpha > 0$.
  \end{remark}
 However, as the general assumptions imply that $s_i^\Sigma$ is of some \emph{generalised polynomial structure}, such a property can be ensured by using the following (special case of a more general) result of Amann \cite[Theorems 2.1 \& 4.1]{Amann_1991}, and which has been extended by Köhne and Saal, cf.\ e.g.\ \cite[Theorem 1.2]{KoehneSaal_2017+}, to functions mapping into a Banach space of class $\HTcal$, cf.\ \cite{KalWei01}.

 \begin{theorem}
 \label{thm:algebra-full_space}
  Let a natural number $m \in \N$, scalars $s_1, \ldots, s_m, s \in (0, \infty)$, $p_1, \ldots, p_m \in (1, \infty)$ and Banach spaces $X_1, \ldots, X_m, X$ of class $\mathcal{HT}$ be given.
  Assume that there exists a continuous multiplication
   \[
    \cdot: X_1  \times \cdots \times X_m \rightarrow X.
   \]
  This induces the pointwise defined continuous multiplication operator, acting as
   \[
    \CC_0(\R \times \R^d;X_1) \times \cdots \times \CC_0(\R \times \R^d;X_m) \rightarrow \CC_0(\R \times \R^d;X).
   \]
  We denote by
   \begin{align*}
    \operatorname{ind}_j
     &:= \operatorname{ind} (\WW^{(1,2) \cdot s_j}_{p_j}(\R \times \R^d;X_j)) = 2 s_j - \frac{d+2}{p_j},
     \\
    \operatorname{ind}
     &:= \operatorname{ind} (\WW^s_p)(\R \times \R^d; X) = 2 s - \frac{d+2}{p}
   \end{align*}
  the \emph{anisotropic Sobolev indices} of the Sobolev--Slobodetskii spaces $\WW^{(1,2) \cdot s_j}_{p_j}(\R \times \R^d;X)$ and $\WW^s_p(\R \times \R^d;X)$, respectively.
  There is a unique continuous extension of the multiplication operator introduced above to a multiplication operator 
   \[
    \cdot: \WW^{(1,2) \cdot s_1}_{p_1}(\R \times \R^d;X_1) \times \cdots \times \WW^{(1,2) \cdot s_m}_{p_m}(\R \times \R^d;X_m)
     \rightarrow \WW^{(1,2) \cdot s}_p(\R \times \R^d;X),
   \]
  if the following conditions are satisfied:
   \begin{alignat*}{2}
    s
     &\leq \min_{j=1, \ldots, m} s_j
	 \qquad &
     &(\text{regularity constraint})
     \\
    \frac{1}{p}
     &\geq \sum_{j=1}^m \frac{1}{p_j}
     \qquad &
     &(\text{integrability constraint})
     \\
    \operatorname{ind} 
     &\leq \begin{cases} \min_{j=1, \ldots, m} \operatorname{ind}_j, &\text{if all } \operatorname{ind}_j \geq 0 \\ \sum_{j: \operatorname{ind}_j < 0} \operatorname{ind}_j, &\text{otherwise} \end{cases}
	 \qquad &
     &(\text{index constraint})
   \end{alignat*}
  with a strict inequality in the index constraint if at least one of the Sobolev indices $\operatorname{ind}_j$ equals zero.
 \end{theorem}

By means of localisation, flattening of the boundary and interpolation, we obtain the following result for sufficiently smooth $(d-1)$-dimensional manifolds, which, in particular, includes the case of boundaries $\partial \Omega$ of regular bounded domains $\Omega$.

 \begin{corollary}
 \label{cor:algebra-property-manifold}
  Let $T > 0$, $d \in \N$, $d^\mathcal{M} \leq d$, and $\mathcal{M} \subseteq \R^d$ be a compact $d^\mathcal{M}$-dimensional $\CC^{2\alpha}$-manifold for some $\alpha > 0$.
  Let $m \in \N$ and $s_1, \ldots, s_m, s \in (0, \alpha]$, $p_1, \ldots, p_m, p \in (1, \infty)$ and Banach spaces $X_1, \ldots, X_m, X$ of class $\HTcal$ such that there is a continuous multiplication
   \[
    \cdot: X_1  \times \cdots \times X_m \rightarrow X.
   \]
  Set $\mathcal{M}_T := [0,T] \times \mathcal{M}$ and let
   \begin{align*}
    \operatorname{ind}^\mathcal{M}_j
     &:= \operatorname{ind} (\WW^{(1,2) \cdot s_j}_{p_j}(\mathcal{M}_T;X_j)) = 2 s_j - \frac{d^\mathcal{M} + 2}{p_j},
     \\
    \operatorname{ind}^\mathcal{M}
     &:= \operatorname{ind} (\WW^{(1,2) \cdot s}_p(\mathcal{M}_T; X)) = 2s - \frac{d^\mathcal{M}+2}{p}.
   \end{align*}
  There is a (unique) continuous extension of the multiplication operator to a multiplication operator
   \[
    \cdot: \WW^{(1,2) \cdot s_1}_{p_1}(\mathcal{M}_T;X_1) \times \cdots \times \WW^{(1,2) \cdot s_m}_{p_m}(\mathcal{M}_T;X_m)
     \rightarrow \WW^{(1,2) \cdot s}_p(\mathcal{M}_T;X),
   \]
  if the following conditions are satisfied:
   \begin{alignat*}{2}
    s
     &\leq \min_{j=1, \ldots, m} s_j
     \qquad &
     &(\text{regularity constraint})
     \\
    \frac{1}{p}
     &\geq \sum_{j=1}^m \frac{1}{p_j}
	 \qquad &
     &(\text{integrability constraint})
     \\
    \operatorname{ind}^\mathcal{M}
     &\leq \begin{cases} \min_{j=1, \ldots, m} \operatorname{ind}^\mathcal{M}_j, &\text{if all } \operatorname{ind}^\mathcal{M}_j \geq 0 \\ \sum_{j: \operatorname{ind}^\mathcal{M}_j < 0} \operatorname{ind}^\mathcal{M}_j, &\text{otherwise} \end{cases}
	 \qquad &
     & (\text{index constraint})
   \end{alignat*}
  with a strict inequality in the index constraint if at least one of the Sobolev indices $\operatorname{ind}^\mathcal{M}_j$ equals zero.
 \end{corollary}
 
 \begin{proof}
  Since $\mathcal{M}$ is a compact, $d^\mathcal{M}$-dimensional $\CC^{2\alpha}$-manifold, we find a partition of unity $(\psi_\nu)_{\nu=1}^n$ of $\mathcal{M}$, where each function $\psi_\nu$ is of class $\CC_c^\infty(\R^d)$, subordinate to a finite open (in $\R^d$) covering $\bigcup_{\nu=1}^n U_\nu \supset \mathcal{M}$ of the manifold, and such that there exist $\CC^{2\alpha}$-isomorphisms $\Phi_\nu: U_\nu \rightarrow V_\nu$ such that $U_\nu \cap \mathcal{M} = \{ \vec x \in U_\nu: \, \Phi_\nu(\vec x) \in \R^{d^\mathcal{M}} \times \{ \vec 0 \} \subseteq \R^d \}$.
  Additionally, we may and will assume that not only $\psi_\nu \in \CC_c^\infty(\R^d)$, but also $\psi_\nu^{1/m} \in \CC_c^\infty(\R^d)$ (this will be needed below).
  By definition, (the equivalence class of) a function $u: \mathcal{M} \rightarrow E$ then lies in the anisotropic Sobolev space $\WW_q^{(1,2) \cdot \sigma}(\mathcal{M}_T;E)$ for some Banach space $E$ and parameters $q \in (1,\infty)$ and $\sigma \in (0, \alpha]$, if and only if for every $\nu = 1, \ldots, n$ the function $v_\nu(t,\vec x') := (\psi_\nu u)(t, \Phi_i^{-1}((\vec x', \vec 0)))$ with $\vec x' \in V_\nu' := \{ \vec x' \in \R^{d^\mathcal{M}}: \, (\vec x', \vec 0) \in V_\nu \}$ defines a function in $\WW_q^{(1,2) \cdot \sigma}((V_\nu')_T)$, and which, therefore, can be extended by zero to a function (still denoted by $v$) in $\WW_q^{(1,2) \cdot \sigma}((\R^{d^\mathcal{M}})_T$.
  \newline
  As we chose the functions $\psi_\nu$ such that, additionally, $\psi_\nu^{1/m}  \in \CC_c^\infty(\R^d)$, we infer that
   \[
    \psi_\nu^{1/m} u
     \in \WW_q^{(1,2) \cdot \sigma}(\mathcal{M}_T)
     \quad \text{with} \quad
    \norm{\psi_\nu^{1/m} u}_{\WW_q^{(1,2) \cdot \sigma}}
     \leq C_{\nu,m,q,\sigma} \norm{u}_{\WW_q^{(1,2) \cdot \sigma}}.
   \]
  Employing Theorem~\ref{thm:algebra-full_space}, we conclude that there exists a continuous extension of the multiplication operator to a multiplication on the full space $M^\mathrm{fs}: \, \WW_{p_1}^{(1,2) \cdot s_1}(\R^d_T) \times \cdots \times \WW_{p_m} ^{(1,2) \cdot s_m}(\R^d_T) \rightarrow \WW_p^{(1,2) \cdot s}(\R^d_T)$.
  We may, therefore, define a candidate for the desired extension of the multiplication operator by setting
   \[
    (M^\mathcal{M} (u_1,\ldots,u_m))(t, \vec x)
     := \sum_{ \{\nu: \vec x \in U_\nu \} } M^\mathrm{fs} ((\psi_\nu^{1/m} u_1)^\ast_\nu, \ldots, (\psi_\nu^{1/m} u_m)^\ast_\nu)(t, \Phi_\nu(\vec x))
   \]
  for $(t, \vec x) \in \mathcal{M}_T$,
  where we use the notation $(\psi_\nu^{1/m} u_i)^\ast_\nu(t, \Phi_\nu(\vec x)) := (\psi_\nu^{1/m} u_i)(t, \vec x)$ for $\nu = 1, \ldots, n$, $t \in [0,T]$ and $\vec x \in \mathcal{M}$.
  This defines a map $M^\mathcal{M}: \WW_{p_1}^{(1,2) \cdot s_1}(\mathcal{M}_T) \times \cdots \times \WW_{p_m} ^{(1,2) \cdot s_m}(\mathcal{M}_T) \rightarrow \WW_p^{(1,2) \cdot s}(\mathcal{M}_T)$ which is continuous thanks to $M^\mathrm{fs}$ being continuous and continuity of the push-forward and pull-back through the maps $\Phi_\nu$.
  It remains to check that this map actually extends the multiplication operator $\cdot: \, X_1 \times \ldots \times X_m \rightarrow X$.
  To this end, let us consider functions $u_i \in \WW_{p_i}^{(1,2) \cdot s_i}(\mathcal{M}_T)$ which, additionally, are continuous.
  Then each function $(\psi_\nu^{1/m} u_i)_\nu^\ast$ is continuous.
  Since $M^\mathrm{fs}$ is an extension of the pointwise multiplication for continuous functions, we infer that
   \begin{align*}
    &M^\mathrm{fs} ((\psi_\nu^{1/m} u_1)^\ast_\nu, \ldots, (\psi_\nu^{1/m} u_m)^\ast_\nu)(t, \Phi_\nu(\vec x))
     \\
     &= (\psi_\nu^{1/m} u_1)^\ast_\nu(t, \Phi_\nu(\vec x)) \cdot \ldots (\psi_\nu^{1/m} u_m)^\ast_\nu(t, \Phi_\nu(\vec x))
      \\
     &= \psi_\nu(\vec x) \big( u_1(t, \vec x) \cdot \ldots \cdot u_m(t, \vec x) \big).
   \end{align*}
  Summing this identity over those $\nu = 1, \ldots, n$ for which $\vec x \in U_\nu$, we conclude that
   \begin{align*}
    M^{\mathcal{M}}(u_1,\ldots,u_m)(t, \vec x)
     &= \sum_{ \{\nu: \vec x \in U_\nu \} } \psi_\nu(\vec x) \big( u_1(t, \vec x) \cdot \ldots \cdot u_m(t, \vec x) \big)
     \\
     &= u_1(t, \vec x) \cdot \ldots \cdot u_m(t, \vec x)
     \quad
     \text{for }
     (t, \vec x) \in \mathcal{M}_T,
   \end{align*}
  since $(\psi_\nu)_{\nu=1}^n$ is a partition of unity subject to the open covering $(U_\nu)_{\nu=1}^n \supset \mathcal{M}$.
  This shows that, for continuous functions, $M^{\mathcal{M}}$ does indeed act as a pointwise multiplication.
  \newline
  On the other hand, uniqueness of the continuous extension follows from density of the space $\WW^{(1,2)\cdot s_j}_p(\mathcal{M}_T) \cap \CC(\mathcal{M}_T)$ in $\WW^{(1,2)\cdot s_j}_p(\mathcal{M}_T)$.
 \end{proof}
 
 We will apply this result to the particular case where $\mathcal{M} = \Sigma = \partial \Omega$ is the $\CC^2$-boundary of a bounded domain $\Omega \subseteq \R^d$, which then has dimension $d^\mathcal{M} = d^\Sigma = d - 1$.
 
 \begin{corollary}
 \label{cor:algebra_boundary}
  Let $\Sigma = \partial \Omega$ be the boundary of a bounded $\CC^2$-domain $\Omega \subseteq \R^d$ and let $p \in (1, \infty)$, $K^\Omega, K^\Sigma \in \N$.
  Assume that one of the following cases is valid:
   \begin{itemize}
    \item
     $p > d$,
    \item
     $p \geq d - \frac{d+1-K^\Omega}{K^\Omega + K^\Sigma}$ and $p \geq \frac{d+2}{2}$ and $p \geq \frac{K^\Omega - 1}{2 K^\Omega - 1} (d + 2)$,
    \item
     $p \geq d - \frac{d+1-K^\Omega}{K^\Omega + K^\Sigma}$ and $p \geq \frac{(K^\Omega + K^\Sigma - 1)(d + 2) - K^\Sigma}{2 (K^\Omega + K^\Sigma) - 1}$.
   \end{itemize} 
  Then, for all Lipschitz continuous functions $\psi^\Omega_1, \ldots, \psi^\Omega_{K^\Omega}, \psi^\Sigma_1, \ldots, \psi^\Sigma_{K^\Sigma} \in \Lip(\R^N;\R)$, the function
   \[
    f: (\R^N)^{K^\Omega} \times (\R^N)^{K^\Sigma} \rightarrow \R,
     \quad
     f(\vec u, \vec u^\Sigma)
      = \prod_{k=1}^{K^\Omega} \psi^\Omega_k(\vec u_k) \cdot \prod_{k=1}^{K^\Sigma} \psi^\Sigma_k(\vec u_k^\Sigma)
   \]
  acts as a continuous Nemytskii operator
   \[
    F: \quad [\WW^{(1,2) \cdot (1 - \frac{1}{2p})}_p(\Sigma_T;\R^N)]^{K^\Omega} \times [\WW^{(1,2)}_p(\Sigma_T;\R^N)]^{K^\Sigma}
     \rightarrow \WW_p^{(1,2) \cdot (\frac{1}{2} - \frac{1}{2p})}(\Sigma_T),
   \]
  via
   \[
     [F(\vec u,\vec u^\Sigma)](t, \vec x)
     = f(\vec u(t, \vec x), \vec u^\Sigma(t, \vec x))
     \quad \text{for a.e.\ } (t,\vec x) \in \Sigma_T.
   \]
 \end{corollary}
 
\begin{proof}
 We start with the special case $N = 1$ and, for the moment, assume that each function $\psi^\Omega_k = \psi_k^\Sigma \equiv 1$ is the identity on $\R$, for all $k = 1, \ldots, K^\Omega$ resp.\ $k = 1, \ldots, K^\Sigma$.
 We may w.l.o.g.\ assume that $K^\Omega + K^\Sigma \geq 2$; otherwise, the assertion is trivial in view of the anisotropic version of the Sobolev embedding theorem.
 Employing Corollary~\ref{cor:algebra-property-manifold} for $\alpha = 1$, we check the validity of the regularity, the integrability and the index constraint:
  \newline
  \textit{Regularity constraint:} Obviously, $\min \{1, 1 - \frac{1}{2p}\} > \frac{1}{2} - \frac{1}{2p}$, thus there is nothing to check here.
  \newline
  \textit{Integrability constraint:} Naively, we would get the integrability constraint $\frac{K^\Omega + K^\Sigma}{p} \leq \frac{1}{p}$, which is not valid, in general, under the imposed conditions.
  However, we may use anisotropic Sobolev embeddings into spaces with lower regularity, but higher integrability parameters, provided by the continuous embeddings $\WW^{(1,2) \cdot (1 - \frac{1}{2p})}_p(\Sigma_T) \hookrightarrow \WW_q^{(1,2)\cdot(\frac{1}{2}-\frac{1}{2p})}(\Sigma_T)$ and $\WW_p^{(1,2)}(\Sigma_T) \hookrightarrow \WW_r^{(1,2)\cdot(\frac{1}{2}-\frac{1}{2p})}(\Sigma_T)$, where we look for integrability parameters $q, r \in (1, \infty)$ such that (in the best case) the anisotropic (parabolic) Sobolev index of the spaces involved in each of these embeddings is the same.
  That is, we look for $q,r \in (p,\infty)$ such that
   \[
    2 - \frac{d+2}{p}
     = 1 - \frac{1}{p} - \frac{d+1}{q}
     \quad \text{and} \quad
    2 - \frac{d+1}{p}
     = 1 - \frac{1}{p} - \frac{d+1}{r}
   \]
  which is equivalent to the condition that
   \begin{equation}
    1 - \frac{d+1}{p}
     = - \frac{d+1}{q}
     \quad \text{and} \quad
    1 - \frac{d}{p}
     = - \frac{d+1}{r}.
     \label{eqn:index_condition}
   \end{equation}
  The latter two equations have solutions $q,r \in (p, \infty)$ if (and only if) $p < d$. However, one may see that choosing $q,r \in (p, \infty)$ sufficiently large in case $p \geq d$, the index constraint will still hold true for the spaces into which we embed, so that Corollary~\ref{cor:algebra-property-manifold} is still applicable.
  We, therefore, focus on the critical case $p \in (1, d)$.
  In this case, the parameters $q,r \in (p, \infty)$ can be defined by the equalities in equation \eqref{eqn:index_condition}, i.e.\
   \[
    \frac{1}{q}
     = \frac{(d + 1 - p)^+}{p (d + 1)}
     \quad \text{and} \quad
    \frac{1}{r}
     = \frac{(d-p)^+}{p(d+1)}.
   \]
  For these $q, r$, the spaces $\WW_q^{(1,2) \cdot (\frac{1}{2} - \frac{1}{2p})}(\Sigma_T)$ (which has the anisotropic Sobolev index $2 - \frac{d+2}{p}$) and $\WW_r^{(1,2)\cdot(\frac{1}{2}-\frac{1}{2p})}(\Sigma_T)$ (which has the anisotropic Sobolev index $2 - \frac{d+1}{p}$) then satisfy the integrability constraint if and only if
   \[
    \frac{1}{p}
     \geq \frac{K^\Omega}{q} + \frac{K^\Sigma}{r}
     = \frac{K^\Omega (d+1-p)^+ + K^\Sigma (d-p)^+}{p(d+1)},
   \]
 i.e.\ $K^\Omega (d+1-p)^+ + K^\Sigma (d-p)^+ \leq d+1$.
  We may, therefore, conclude that this is a sufficient condition to satisfy the integrability constraint for the existence of a continuous extension of the multiplication to a map
   \[
    [\WW_q^{(1,2)\cdot(\frac{1}{2}-\frac{1}{2p})}(\Sigma_T)]^{K^\Omega}
     \times [\WW_r^{(1,2)\cdot(\frac{1}{2}-\frac{1}{2p})}(\Sigma_T)]^{K^\Sigma}
     \rightarrow \WW_p^{(1,2)\cdot(\frac{1}{2}-\frac{1}{2p})}(\Sigma_T).
  \]
  \textit{Index constraint:}
  We employ the same Sobolev embeddings which we have used for the integrability constraint, i.e.\ the embeddings $\WW_p^{(1,2)}(\Omega_T) \hookrightarrow \WW_q^{(1,2) \cdot (\frac{1}{2} - \frac{1}{2p})}(\Sigma_T)$ and $\WW_p^{(1,2)}(\Sigma_T) \hookrightarrow \WW_r^{(1,2)\cdot(\frac{1}{2}-\frac{1}{2p})}(\Sigma_T)$.
  \newline
  If $p \geq d$, then $p,r \in (1,\infty)$ can be chosen arbitrarily large and the Sobolev indices of $\WW_q^{(1,2) \cdot (\frac{1}{2} - \frac{1}{2p})}(\Sigma_T)$ and $\WW_r^{(1,2)\cdot(\frac{1}{2}-\frac{1}{2p})}(\Sigma_T)$, thus, can be chosen to be arbitrarily close to $1 - \frac{1}{p} > 0$, i.e., in particular, positive.
  In this case, both these indices are strictly greater than $1 - \frac{d+2}{p}$, which is the Sobolev index of $\WW_p^{(1,2) \cdot (\frac{1}{2} - \frac{1}{2p})}(\Sigma_T)$, and the index constraint can be satisfied.
  \newline
  Thus, let us consider the case $p \in (1, d)$ next. Then, the Sobolev indices are not affected by the Sobolev embedding, thanks to the choice \eqref{eqn:index_condition}.
  The index constraint is trivial in case $p \geq \frac{d+2}{2}$, since then $\WW^{(1,2)}_p(\Omega_T)$ and $\WW^{(1,2)}_p(\Sigma_T)$ have non-negative indices ($2 - \frac{d+2}{p} \geq 0$ and $2 - \frac{d+1}{p} > 0$) strictly greater than the index ($1 - \frac{d+2}{p}$) of the space $\WW^{(1,2) \cdot (\frac{1}{2} - \frac{1}{2p})}(\Sigma_T)$.
 \newline
  We may, thus, focus on the case $p < \frac{d+2}{2}$.
  In the case $p \in [\frac{d+1}{2}, \frac{d+2}{2})$, it holds that $2 - \frac{d+2}{p} < 0$ whereas $2 - \frac{d+1}{p} \geq 0$, thus, for the index constraint to be satisfied, we have to ensure that
   \[
    K^\Omega (2 - \frac{d+2}{p}) \geq 1 - \frac{d+2}{p}
     \quad \Leftrightarrow \quad
     p \geq \frac{(K^\Omega - 1)^+}{(2K^\Omega - 1)^+} (d+2).
   \]
  As the right-hand side tends to $\frac{d+2}{2} > \frac{d+1}{2}$ as $K^\Omega \rightarrow \infty$, this, in general, constitutes an additional constraint on $p$.
  \newline
  If $p < \frac{d+1}{2}$, both $2 - \frac{d+2}{2}$ and $2 - \frac{d+1}{2}$ are strictly negative, thus, we need to ensure that
   \begin{align*}
    K^\Omega (2 - \frac{d+2}{p}) + K^\Sigma (2 - \frac{d+1}{p})
     \geq 1 - \frac{d+2}{p}
     \\
     \quad \Leftrightarrow \quad
     p
      \geq \frac{(K^\Omega + K^\Sigma - 1)(d+2) - K^\Sigma}{2(K^\Omega + K^\Sigma) - 1}.
   \end{align*}
  To finish the proof, the case of general Lipschitz continuous functions $\psi_i$ and $\psi_i^\Sigma$ follows from the observation that for all $\sigma \in [0,\frac{1}{2}]$ any Lipschitz continuous function $\psi$ continuously maps functions $u$ in $\WW^{(1,2) \cdot \sigma}$ to functions in $\psi \circ u \in \WW^{(1,2) \cdot \sigma}$, thus we may interpret functions in $\WW^{(1,2)}_p(\Omega_T)$ and $\WW^{(1,2)}(\Sigma_T)$ by the continuous embedding from above as functions in $\WW^{(1,2) \cdot (\frac{1}{2} - \frac{1}{2p})}_q(\Sigma_T)$ and $\WW^{(1,2)\cdot (\frac{1}{2} - \frac{1}{2p})}_r(\Sigma_T)$, respectively, then apply the Nemytskii operator corresponding to the functions $\psi_i$ and $\psi_i^\Sigma$, respectively, and finally use the previous considerations  to deduce the general assertion from the special case $\psi_i = \psi_i^\Sigma \equiv 1$.
\end{proof}

\begin{remark}
 Note that the particular case $p > d$ covered in \cite{MorTan22+} is also admissible in our situation, but we may actually allow for lower integrability index $p \in (1, \infty)$, provided the exponents $K^\Omega$ and $K^\Sigma$ for the polynomial growth bounds are not too large.
 \newline
 In fact, for the standard multicomponent Langmuir sorption model, we can choose $K^\Omega = K^\Sigma = 1$, hence, all $p \geq \frac{d+2}{2}$ become admissible for Corollary \ref{cor:algebra_boundary}.
\end{remark}

To include an additional bounded and Lipschitz continuous prefactor, let us note the following special case of an extension lemma for the multiplication: 
 
\begin{lemma}
\label{lem:special_case_multiplication}
 Let $\mathcal{M} \subseteq \R^d$ be a compact, regular Riemannian manifold of dimension $d^\mathcal{M}$ and let $T > 0$ and $p \in ( \frac{d^\mathcal{M} + 3}{2}, \infty)$ be given.
 Moreover, let functions $\varphi \in \LL_\infty(\mathcal{M}_T) \cap \WW_p^{(1,2) \cdot (1 - \frac{1}{2p})}(\mathcal{M}_T)$ and $f \in \WW_p^{(1,2) \cdot (\frac{1}{2} - \frac{1}{2p})}(\mathcal{M}_T)$ be given.
 Then the function $f \varphi$ belongs to the class $\WW_p^{(1,2) \cdot (\frac{1}{2} - \frac{1}{2p})}(\mathcal{M}_T)$ as well and there is a constant $C > 0$ independent of $\varphi$ and $f$ such that
  \[
   \norm{f \varphi}_{\WW_p^{(1,2) \cdot (\frac{1}{2} - \frac{1}{2p})}}
    \leq C \norm{\varphi}_{\LL_\infty \cap \WW_p^{(1,2) \cdot (1 - \frac{1}{2p})}} \norm{f}_{\WW_p^{(1,2) \cdot (\frac{1}{2} - \frac{1}{2p})}}
  \]
 for all $\varphi \in \LL_\infty \cap \WW_p^{(1,2) \cdot (1 - \frac{1}{2p})}(\mathcal{M}_T)$ and $f \in \WW_p^{(1,2) \cdot (\frac{1}{2} - \frac{1}{2p})}(\mathcal{M}_T)$.
 In other words, the bilinear form
  \[
   (\LL_\infty \cap \WW_p^{(1,2) \cdot (1 - \frac{1}{2p})})(\mathcal{M}_T) \times \WW_p^{(1,2) \cdot (\frac{1}{2} - \frac{1}{2p})}(\mathcal{M}_T)
    \ni (\varphi, f)
    \mapsto \varphi f
    \in \WW_p^{(1,2) \cdot (\frac{1}{2} - \frac{1}{2p})}(\mathcal{M}_T)
  \]
 is continuous.
\end{lemma}

\begin{remark}
 Note that this result is not just a direct consequence of Amann's multiplication lemma, since the integrability condition is not satisfied.
 However, using anisotropic versions of the Sobolev embedding theorems, we are able to demonstrate the result.
 \newline
 Also note that we obtain a function $\varphi \in \LL_\infty(\mathcal{M}_T) \cap \WW_p^{(1,2) \cdot (1 - \frac{1}{2p})}(\mathcal{M}_T)$, if we plug in a function $g \in \WW_p^{(1,2) \cdot (1 - \frac{1}{2p})}(\mathcal{M}_T; \R^m)$ into a bounded and Lipschitz continuous function $\Phi \in \WW_\infty^1(\R^m;\R)$, i.e.\ consider $\varphi := \Phi \circ g$.
\end{remark}

\begin{proof}[Proof of Lemma~\ref{lem:special_case_multiplication}]
 Obviously, $\varphi f \in \LL_p(\mathcal{M}_T)$ since
  \[
   \norm{\varphi f}_{\LL_p(\mathcal{M}_T)}
    \leq \norm{\varphi}_{\LL_\infty(\mathcal{M}_T)} \norm{f}_{\LL_p(\mathcal{M}_T)}.
  \]
 Let us set $\sigma := \frac{1}{2} - \frac{1}{2p} \in (0, \frac{1}{2})$.
 To estimate the seminorms $\seminorm{\varphi f}_{\WW_p^{(\sigma,0)}(\mathcal{M}_T)}$ and $\seminorm{\varphi f}_{\WW_p^{(0,2 \sigma)}(\mathcal{M}_T)}$, we use the anisotropic Sobolev embeddings
  \begin{align*}
   \WW_p^{(1,2) \cdot (1 - \frac{1}{2p})}(\mathcal{M}_T)
    &\hookrightarrow
    \WW_{q,p}^{(1,2) \cdot (\frac{1}{2} - \frac{1}{2p})}(\mathcal{M}_T)
    &&\quad \text{and} \quad&
   \WW_p^{(1,2) \cdot (\frac{1}{2} - \frac{1}{2p})}(\mathcal{M}_T)
    &\hookrightarrow
    \LL_{q^\ast,\infty}(\mathcal{M}_T),
    \\
   \WW_p^{(1,2) \cdot (1 - \frac{1}{2p})}(\mathcal{M}_T)
    &\hookrightarrow
    \WW_{p,r}^{(1,2) \cdot (\frac{1}{2} - \frac{1}{2p})}(\mathcal{M}_T)
    &&\quad \text{and} \quad&
   \WW_p^{(1,2) \cdot (\frac{1}{2} - \frac{1}{2p})}(\mathcal{M}_T)
    &\hookrightarrow
    \LL_{\infty,r^\ast}(\mathcal{M}_T),
  \end{align*}   
 which are valid if we can simultaneously satisfy the conditions
  \begin{align*}
   2 - \frac{d^\mathcal{M}+3}{p}
    &\geq 1 - \frac{d^\mathcal{M}+1}{p} - \frac{2}{q}
    &&\quad \text{and} \quad&
   1 - \frac{d^\mathcal{M}+3}{p}
    &> - \frac{2}{q^\ast},
    \\
   2 - \frac{d^\mathcal{M}+3}{p}
    &\geq 1 - \frac{3}{p} - \frac{d^\mathcal{M}}{r}
    &&\quad \text{and} \quad&
   1 - \frac{d^\mathcal{M}+3}{p}
    &> - \frac{d^\mathcal{M}}{r^\ast},
  \end{align*}
 and where $q, q^\ast, r, r^\ast \in (p, \infty)$ are chosen such that $\frac{1}{p} = \frac{1}{q} + \frac{1}{q^\ast} = \frac{1}{r} + \frac{1}{r^\ast}$.
 By summing up the conditions on $q$ and $q^\ast$, and on $r$ and $r^\ast$, respectively, we observe that such $q$ and $q^\ast$ can be chosen if and only if
  \[
   2 - \frac{d^\mathcal{M} + 3}{p}
    > 0
    \quad \Leftrightarrow \quad
   p
    > \frac{d^\mathcal{M} + 3}{2}.
  \]
 \newline
 For the seminorm $\seminorm{\varphi f}_{\WW_p^{(\sigma,0)}(\mathcal{M}_T)}$ we obtain the estimate
  \begin{align*}
   &\seminorm{\varphi f}_{\WW_p^{(\sigma,0)}}^p
    = \int_0^T \int_0^T \frac{ \norm{(\varphi f)(s,\cdot) - (\varphi f)(t,\cdot)}_{\LL_p(\mathcal{M})}^p }{ \abs{ s - t}^{\sigma p + 1} } \dd s \dd t
    \\
    &\leq 2^p \int_0^T \int_0^T \frac{ \norm{\varphi(s,\cdot)(f(s,\cdot) - f(t,\cdot))}_{\LL_p(\mathcal{M})}^p + \norm{(\varphi(s,\cdot) - \varphi(t,\cdot)) f(t,\cdot)}_{\LL_p(\mathcal{M})}^p  }{ \abs{ s - t}^{\sigma p + 1} } \dd s \dd t
    \\
    &\leq 2^p \int_0^T \int_0^T \frac{ \norm{\varphi(s,\cdot)}_{\LL_\infty}^p \norm{f(s,\cdot) - f(t,\cdot)}_{\LL_p(\mathcal{M})}^p + \norm{\varphi(s,\cdot) - \varphi(t,\cdot)}_{\LL_r(\mathcal{M})}^p \norm{f(t,\cdot)}_{\LL_{r^\ast}(\mathcal{M})}^p  }{ \abs{ s - t}^{\sigma p + 1} } \dd s \dd t
    \\
   &\leq 2^p \left( \norm{\varphi}_{\LL_\infty(\mathcal{M}_T)}^p \seminorm{f}_{\WW_p^{(\sigma,0)}(\mathcal{M}_T)}^p + \norm{\varphi}_{\WW_{p,r}^{(\sigma,0)}(\mathcal{M}_T)}^p \norm{f}_{\LL_{\infty,r^\ast}(\mathcal{M}_T)}^p \right)
   \\
   &\leq C \left( \norm{\varphi}_{\WW_p^{(1,2) \cdot (1 - \frac{1}{2p})}(\mathcal{M}_T) \cap \LL_\infty(\mathcal{M}_T)} \norm{f}_{\WW_p^{(1,2) \cdot (\frac{1}{2} - \frac{1}{2p})}(\mathcal{M}_T)} \right)^p.
  \end{align*}
 Similarly, we may estimate
  \begin{align*}
   &\seminorm{\varphi f}_{\WW_p^{(0,2\sigma)}}^p
    = \int_{\mathcal{M}} \int_{\mathcal{M}} \frac{ \norm{(\varphi f)(\cdot, \vec x) - (\varphi f)(\cdot, \vec y)}_{\LL_p(0,T)}^p}{\abs{\vec x - \vec y}^{2 \sigma p + d}} \, \dd \mu(\vec x) \, \dd \mu(\vec y)
    \\
    &\leq 2^p \int_{\mathcal{M}} \int_{\mathcal{M}} \frac{ \norm{\varphi(\cdot, \vec x)(f(\cdot, \vec x) - f(\cdot, \vec y))}_{\LL_p(0,T)}^p + \norm{(\varphi(\cdot, \vec x) - \varphi(\cdot, \vec y))f(\cdot, \vec y)}_{\LL_p(0,T)}^p}{\abs{\vec x - \vec y}^{2 \sigma p + d}} \, \dd \mu(\vec x) \, \dd \mu(\vec y)
    \\
    &\leq 2^p \int_{\mathcal{M}} \int_{\mathcal{M}} \frac{1}{\abs{\vec x - \vec y}^{2 \sigma p + d}} \cdot \big( \norm{\varphi(\cdot, \vec x)}_{\LL_\infty(0,T)}^p \norm{f(\cdot, \vec x) - f(\cdot, \vec y)}_{\LL_p(0,T)}^p
     \\ &\qquad
     + \norm{\varphi(\cdot, \vec x) - \varphi(\cdot, \vec y)}_{\LL_q(0,T)}^p \norm{f(\cdot, \vec y)}_{\LL_{q^\ast}(0,T)}^p \big) \, \dd \mu(\vec x) \, \dd \mu(\vec y)
     \\
   &\leq C \left( \norm{\varphi}_{\WW_p^{(1,2) \cdot (1 - \frac{1}{2p})} \cap \LL_\infty(\mathcal{M}_T)} \norm{f}_{\WW_p^{(1,2) \cdot (\frac{1}{2} - \frac{1}{2p})}(\mathcal{M}_T)} \right)^p.
  \end{align*}
\end{proof}

Therefore, we may conclude that in this case, the reference function $\vec r^\ast$ for the sorption rates defined in equation \eqref{eqn:riast} lies in the Neumann trace space $\fs G_p^\Sigma(T)$.
 Let us denote by $(\vec c^\ast,  \vec c^{\Sigma,\ast}) \in \fs E_p^\Omega(T) \times \fs E_p^\Sigma(T)$ the unique solution of the quasi-autonomous linear initial-boundary value problem
  \begin{align*}
   \mathcal{L}_{T,i}(c_i^\ast, c_i^{\Sigma,\ast})
    &= (0, 0, r_i^\ast)
    \\
   (c_i(0,\cdot), c_i^\Sigma(0,\cdot))
    &= (c^0_i, c^{\Sigma,0}_i),
  \end{align*}
 which will take the role of reference functions.
 We write $c_i = \bar c_i + c_i^\ast$ and $c_i^\Sigma = \bar c_i^\Sigma + c_i^{\Sigma,\ast}$ and observe that the functions $(\vec c, \vec c^\Sigma) \in \fs E_p^\Omega(T) \times \fs E_p^\Sigma(T)$ are solutions of the reaction-diffusion-advection-sorption model under consideration if and only if the remainder functions $(\vec {\bar c}, \vec {\bar c^\Sigma}) \in \fs E_p^\Omega(T) \times \fs E_p^\Sigma(T)$ solve the reduced semilinear \emph{zero time trace problem}
  \[
   \mathcal{L}^0_{T,i} (\bar c_i, \bar c_i^\Sigma)
    = \mathcal{N}_{T,i} (\bar c_i + c_i^\ast, \bar c_i^\Sigma + c_i^{\Sigma,\ast}) - (0,0,r_i^{\ast}),
  \]
 where $\mathcal{L}^0_T := \mathcal{L}_T|_{\mathring {\fs E}_p^\Omega(T) \times \mathring {\fs E}_p^\Sigma(T)}$ is the restriction of the linear parabolic operator $\mathcal{L}_T$ to zero time trace functions from the class
  \[
   \mathring {\fs E}_p^\Omega(T) \times \mathring{\fs E}_p^\Sigma(T)
    := \{(\vec c, \vec c^\Sigma) \in \fs E_p^\Omega(T) \times \fs E_p^\Sigma(T): \, (\vec c(0,\cdot), \vec c^\Sigma(0,\cdot)) = (\vec 0, \vec 0) \}.
  \]
 From here on, we may use $\LL_p$-maximal regularity and Banach's Contraction Mapping Theorem as powerful tools to establish local-in-time existence and uniqueness of strong $\WW_p^{(1,2)}$-solutions to \eqref{eqn:RDASS} and their continuous dependence on the initial data.

 \begin{theorem}[Local-in time well-posedness]
 \label{thm:local_existence}
  Let $J = (0,T') \subseteq \R$ for some $T' > 0$ and $p \in (1, \infty)$.
  Assume that assumptions \eqref{A^vel}, \eqref{A_F^sorp}, \eqref{A_M^sorp}, \eqref{A_B^sorp}, \eqref{A_F^ch}, \eqref{A_N^ch}, \eqref{A_P^ch} and \eqref{A_S^sorp} are satisfied, and $p \in [\frac{d+2}{2},\infty)$ satisfies one of the following conditions:
   \begin{itemize}
    \item
     $p > d$,
    \item
     $p \geq d - \frac{d+1-K^\Omega}{K^\Omega + K^\Sigma}$ and $p \geq \frac{d+2}{2}$ and $p \geq \frac{K^\Omega - 1}{2 K^\Omega - 1} (d + 2)$,
    \item
     $p \geq d - \frac{d+1-K^\Omega}{K^\Omega + K^\Sigma}$ and $p \geq \frac{d+2}{2} - \frac{d + 3}{4 (K^\Omega + K^\Sigma) - 2}$.
   \end{itemize}
  Then, for each initial data in the class
   \[
    c^0_i \in \WW^{2-\nicefrac{2}{p}}_p(\Omega),
     \quad
    c^{\Sigma,0}_i \in \WW^{2-\nicefrac{2}{p}}_p(\Sigma)
    \quad \text{for }
    i = 1, \ldots, N,
   \]
  which in case $p > 3$ additionally satisfy the compatibility conditions
   \[
    - d_i \partial_{\nu} c^0_i
     = s_i^\Sigma(c^0_i, c^{\Sigma,0}_i)
     \quad
     \text{on } \Sigma
     \quad \text{for } i = 1, \ldots, N,
   \]
  there are $T \in (0, T')$ and a unique local-in time strong solution $(\vec c, \vec c^\Sigma) \in \fs E_p^\Omega(T) \times \fs E_p^\Sigma(T)$ of system \eqref{eqn:RDASS}.
  The solution depends continuously w.r.t.\ $\norm{\cdot}_{\WW_p^{2-2/p}}$ on the initial data and can be extended in a unique way to a maximal-in-time solution on some time interval $(0, T_\mathrm{max})$ with $T_\mathrm{max} \in (0, \infty]$.
  If, additionally, $c^0_i \geq 0$ (a.e.), $c^{\Sigma,0}_i \geq 0$ ($\sigma$-a.e.) for all components $i = 1, \ldots, N$, then $c_i \geq 0$  a.e.\ on $\Omega_T$ and $c_i^\Sigma \geq 0$ $\sigma$-a.e.\ on $\Sigma_T$, for all $i = 1, \ldots, N$.
 \end{theorem}
 
 \begin{proof}
  We employ the theory of $\LL_p$-maximal regularity and Banach's Contraction Mapping Principle.
  Let us fix some initial data $(\vec c^0, \vec c^{\Sigma,0}) \in \fs I_p^\Omega \times \fs I_p^\Sigma$ and choose positive numbers $\rho_0 > 0$ and $\tau_0 \in (0, T')$.
  For each $\tau \in (0, \tau_0]$ and $\rho \in (0, \rho_0]$, we define a nonlinear map $\Phi = \Phi_{\rho,\tau}: D_{\rho,\tau} \subseteq \fs E_p^\Omega(\tau) \times \fs E_p^\Sigma(\tau) \rightarrow \fs E_p^\Omega(\tau) \times \fs E_p^\Sigma(\tau)$ on the complete metric space (with the metric inherited from $\fs E_p^\Sigma(\tau)$)
   \begin{align*}
    D_{\rho,\tau}
     &= \{ (\vec c, \vec c^\Sigma) \in \fs E_p^\Omega(\tau) \times \fs E_p^\Sigma(\tau): \, (\vec c(0,\cdot), \vec c^\Sigma(0,\cdot)) = (\vec c^0, \vec c^{\Sigma,0}),
     \\ &\qquad
      \norm{ (\vec c - \ee^{t A} \vec c^0, \vec c^\Sigma - \ee^{t A^\Sigma} \vec c^{\Sigma,0}) }_{\fs E_p^\Omega(\tau) \times \fs E_p^\Sigma(\tau)} \leq \rho \}.
   \end{align*}
  Given $(\vec c, \vec c^\Sigma) \in D_{\rho,\tau}$, let $\Phi(\vec c, \vec c^\Sigma) := (\vec v, \vec v^\Sigma)$ be the unique solution to the quasi-autonomous, linear initial-boundary value problem
   \begin{alignat*}{2}
    (\partial_t + \vec v \cdot \nabla - \bb D \Delta) \vec v
     &= \vec r^\Omega(\vec c)
	 \qquad &
     &\text{in } (0,\tau) \times \Omega,
     \\
    (\partial_t - \bb D \Delta_\Sigma) \vec v^\Sigma
     &= \vec r^\Sigma(\vec c^\Sigma) + \vec s^\Sigma(\vec c, \vec c^\Sigma)
	 \qquad &
     &\text{on } (0, \tau) \times \Sigma,
     \\
    - \bb D \partial_{\vec \nu} \vec v|_{\Sigma}
     &= \vec s^\Sigma(\vec c, \vec c^\Sigma)
	 \qquad &
     &\text{on } (0, \tau) \times \Sigma,
     \\
    \vec v(0,\cdot)
     &= \vec c^0
	 \qquad &
     &\text{in } \Omega,
     \\
    \vec v^\Sigma(0,\cdot)
     &= \vec c^{\Sigma,0}
	 \qquad &
     &\text{on } \Sigma.
   \end{alignat*}
  By $\LL_p$-optimal regularity in Lemma~\ref{lem:maximal_regularity_scalar_case}, this problem has a unique solution in the class $(\vec v, \vec v^\Sigma) \in \fs E_p^\Omega(\tau) \times \fs E_p^\Sigma(\tau)$ if and only if
   \begin{align*}
    (\vec c^0, \vec c^{\Sigma,0}) 
     &\in \WW_p^{2-2/p}(\Omega)^N \times \WW_p^{2-2/p}(\Sigma)^N,
     \\
    \vec r^\Omega(\vec c)
     &\in \LL_p(\Omega_\tau;\R^N),
     \\
    \vec r^\Sigma(\vec c^\Sigma) + \vec s^\Sigma(\vec c, \vec c^\Sigma)
     &\in \LL_p(\Sigma_\tau;\R^N),
     \\
    \vec s^\Sigma(\vec c, \vec c^\Sigma)
     &\in \WW_p^{(1,2) \cdot (1/2 - 1/2p)}(\Sigma_\tau)
   \end{align*}
  and, additionally, if $p > 3$, the compatibility condition
   \[
    \vec s^0
     := \vec s^\Sigma(\vec c^0|_\Sigma, \vec c^{\Sigma,0})
     = - \bb D \partial_{\vec \nu} \vec c^0|_{\partial \Omega}
   \]
  in the trace sense is satisfied.
  We recall that $A = - \bb D \Delta$ is defined on the Neumann domain $\dom(A) = \{ \vec u \in \WW_p^2(\Omega;\R^N): \, - \bb D \partial_{\vec \nu} \vec u = \vec 0 \text{ on } \partial \Omega \}$ and $A^\Sigma = - \bb D^\Sigma \Delta_\Sigma$ on $\dom(A^\Sigma) = \WW_p^2(\Sigma;\R^N)$.
  To estimate the norm of $\vec v - \ee^{tA} \vec c^0$, we set 
   \[
    \vec s^\ast(t,\cdot)
     := \ee^{t A^\Sigma} \vec s^0
     \quad
     \text{for } t \geq 0,
   \]
  which defines a function $\vec s^\ast \in \WW_p^{(1,2) \cdot (1-1/2p)}(\Sigma_\tau)$, and let $\vec c^\ast \in \WW_p^{(1,2)}(\Omega_\tau)$ be the unique strong $\WW_p^{(1,2)}$-solution to the inhomogeneous initial-boundary value problem
   \begin{alignat*}{2}
    (\partial_t + \vec v \cdot \nabla - \bb D \Delta) \vec c^\ast
     &= \vec r^\Omega(\vec c^0)
     \quad &
     &\text{in } (0,\tau) \times \Omega,
     \\
    - \bb D \partial_{\vec \nu} \vec c^\ast|_{\Sigma}
     &= \vec s^\ast
     \quad &
     &\text{on } (0,\tau) \times \Sigma
     \\
     \vec c^\ast(0,\cdot)
      &= \vec c^0
      \quad &
      &\text{in } \Omega.
   \end{alignat*}
  Then the remainder $\vec w := \vec v - \vec c^\ast$ solves the initial-boundary value problem with zero initial data
   \begin{alignat*}{2}
    (\partial_t + \vec v \cdot \nabla - \bb D \Delta) \vec w
     &= \vec r^\Omega(\vec c) - \vec r^\Omega(\vec c^0)
     \quad &
     &\text{in } (0,\tau) \times \Omega,
     \\
    - \bb D \partial_{\vec \nu} \vec w|_\Sigma
     &= \vec s^\Sigma(\vec c, \vec c^\Sigma) - \vec s^\ast
     \quad &
     &\text{on } (0,\tau) \times \Sigma,
     \\
    \vec w(0,\cdot)
     &= \vec 0
     \quad &
     &\text{in } \Omega,
   \end{alignat*}
  where (in case $p > 3$) the boundary and initial data are compatible thanks to
   \[
    \vec s^\Sigma(\vec c|_\Sigma, \vec c^\Sigma)|_{t=0}
      = \vec s^\Sigma(\vec c^0|_\Sigma, \vec c^{\Sigma,0})
      = \vec s^0
      = \vec s^\ast|_{t=0}.
   \]
  Thus, given $(\vec c, \vec c^\Sigma)$ and $(\vec {\tilde c}, \vec {\tilde c}^\Sigma) \in D_{\rho,\tau}$ and corresponding solutions $(\vec v, \vec v^\Sigma) = \Phi(\vec c, \vec c^\Sigma)$ and $(\vec {\tilde v}, \vec {\tilde v}^\Sigma) = \Phi(\vec {\tilde c}, \vec {\tilde c}^\Sigma)$, the difference $\vec v - \vec {\tilde v}$ may be written as $\vec w - \tilde {\vec w}$, where
   \begin{alignat*}{2}
    (\partial_t + \vec v \cdot \nabla - \bb D \Delta) (\vec w - \vec {\tilde w})
     &= \vec r^\Omega(\vec c) - \vec r^\Omega(\vec {\tilde c})
     \quad &
     &\text{on } (0,\tau) \times \Omega,
     \\
    - \bb D \partial_{\vec \nu} (\vec w - \vec {\tilde w})|_\Sigma
     &= \vec s^\Sigma(\vec c, \vec c^\Sigma) - \vec s^\Sigma(\vec {\tilde c}, \vec {\tilde c}^\Sigma)
     \quad &
     &\text{on } (0,\tau) \times \Sigma,
     \\
    (\vec w - \vec {\tilde w})(0,\cdot)
     &= \vec 0
     \quad &
     &\text{in } \Omega.
   \end{alignat*}
  Since the latter two systems have \emph{zero initial data}, we may again employ Lemma~\ref{lem:maximal_regularity_scalar_case}, which gives us a maximal regularity constant $C = C(\tau_0) > 0$, independent of $\tau \in (0,\tau_0]$ such that for all $\tau \in (0, \tau_0]$ the estimates
   \[
    \norm{\vec w}_{\WW_p^{(1,2)}(\Omega_\tau)}
     \leq C \big( \norm{\vec r^\Omega(\vec c) - \vec r^\Omega(\vec c^0)}_{\LL_p(\Omega_\tau)} + \norm{\vec s(\vec c, \vec c^\Sigma) - \vec s^\ast}_{\WW_p^{(1,2) \cdot (1/2-1/2p)}(\Sigma_\tau)} \big)
   \]
  and
   \[
    \norm{\vec w - \vec {\tilde w}}_{\WW_p^{(1,2)}(\Omega_\tau)}
     \leq C \big( \norm{\vec r^\Omega(\vec c) - \vec r^\Omega(\vec {\tilde c})}_{\LL_p(\Omega_\tau)} + \norm{\vec s(\vec c, \vec c^\Sigma) - \vec s(\vec {\tilde c}, \vec {\tilde c}^\Sigma)}_{\WW_p^{(1,2) \cdot (1/2-1/2p)}(\Sigma_\tau)} \big)
   \]
  are valid.
  We observe that the maps $D_{\rho,\tau} \ni (\vec c, \vec c^\Sigma) \mapsto \vec r^\Omega(\vec c) \in \LL_p(\Omega_\tau)$ and $D_{\rho,\tau} \ni (\vec c, \vec c^\Sigma) \mapsto \vec s(\vec c, \vec c^\Sigma) \in \WW_p^{(1,2) \cdot (1/2 - 1/2p)}(\Sigma_\tau)$ are both Lipschitz continuous with Lipschitz constants which are uniformly bounded for $\rho \in (0, \rho_0]$ and $\tau \in (0, \tau_0]$, and which tend to zero as $\tau \rightarrow 0$ uniformly for all $\rho \in (0, \rho_0]$, e.g.\
   \begin{align*}
    \norm{\vec r^\Omega(\vec c) - \vec r^\Omega(\vec {\tilde c})}_{\LL_p(\Omega_\tau)}
     &= \norm{ \int_0^1 (\vec c - \vec {\tilde c}) \cdot (\vec r^\Omega)'(s \vec c + (1-s) \vec {\tilde c}) \dd s }_{\LL_p(\Omega_\tau)}
     \\
     &\leq C \norm{\vec c - \vec {\tilde c}}_{\LL_{p\gamma^\Omega}(\Omega_\tau)} \int_0^1 \big( \tau^{1 - \frac{1}{p \gamma^\Omega}} + \norm{s \vec c + (1-s) \vec {\tilde c}}_{\LL_{p\gamma^\Omega}(\Omega_\tau)}^{\gamma^\Omega - 1} \big) \dd s
     \\
     &\leq C \norm{\vec c - \vec {\tilde c}}_{\LL_{p\gamma^\Omega}(\Omega_\tau)} \big( \tau^{1 - \frac{1}{p \gamma^\Omega}} + \norm{\vec c}_{\LL_{p\gamma^\Omega - 1}(\Omega_\tau)}^{\gamma^\Omega} + \norm{\vec {\tilde c}}_{\LL_{p\gamma^\Omega}(\Omega_\tau)}^{\gamma^\Omega - 1} \big).
   \end{align*}
 Let us note that
  \[
   \mathring \WW_p^{(1,2)}(\Omega_\tau)
    \hookrightarrow \LL_{p\gamma^\Omega}(\Omega_\tau)
  \]
 embeds continuously if $p \geq \frac{(\gamma^\Omega - 1) \cdot (d+2)}{2 \gamma^\Omega}$, and the embedding constant may be chosen uniform for all $\tau \in (0,T')$.
 Moreover, if $p > \frac{(\gamma^\Omega - 1)(d + 2)}{2 \gamma^\Omega}$, the embedding constant tends to zero as $\tau \rightarrow 0+$.
 (Recall that $(\vec c(0,\cdot), \vec c^\Sigma(0,\cdot)) = (\vec c^0, \vec c^{\Sigma,0})$ for all $(\vec c, \vec c^\Sigma) \in D_{\rho,T}$, thus $\vec c - \vec {\tilde c} \in \mathring \WW_p^{(1,2)}(\Omega_\tau)$.)
 Similarly,
  \begin{align*}
   \norm{\vec r^\Sigma(\vec c^\Sigma) - \vec r^\Sigma(\vec {\tilde c}^\Sigma)}_{\LL_p(\Sigma_\tau)}
    &\leq C \norm{\vec c^\Sigma - \vec {\tilde c}^\Sigma}_{\LL_{p\gamma^\Sigma}(\Sigma_\tau)} \big( \tau^{1 - \frac{1}{p \gamma^\Omega}} + \norm{\vec c^\Sigma}_{\LL_{p\gamma^\Sigma - 1}(\Sigma_\tau)}^{\gamma^\Sigma} + \norm{\vec {\tilde c}^\Sigma}_{\LL_{p\gamma^\Sigma}(\Sigma_\tau)}^{\gamma^\Sigma - 1} \big)
  \end{align*}
 provided $p \geq \frac{\gamma^\Sigma - 1}{\gamma^\Sigma} \cdot \frac{d+1}{2}$.
 Note that both $p > \frac{(\gamma^\Omega - 1)(d + 2)}{2 \gamma^\Omega}$ and $p > \frac{(\gamma^\Sigma - 1)(d + 1)}{2 \gamma^\Sigma}$ are always satisfied, if $p \geq \frac{d+2}{2}$.
 This means that there are constants $C_{\tau,\rho} > 0$, which can be chosen uniformly bounded for $\tau \in (0, \tau_0]$ and $\rho \in (0, \rho_0]$, and tend to zero as $\tau \rightarrow 0+$, uniformly for all $\rho \in (0, \rho_0]$, such that
  \begin{align*}
   \norm{\vec w}_{\WW_p^{(1,2)}(\Omega_\tau)}
    &\leq C_{\tau,\rho} \big( \norm{\vec c}_{\WW_p^{(1,2)}(\Omega_\tau)} + \norm{\vec c^\Sigma}_{\WW_p^{(1,2)}(\Sigma_\tau)} + 1  \big)
    \quad \text{and}
    \\
   \norm{\vec v - \vec {\tilde v}}_{\WW_p^{(1,2)}(\Omega_\tau)}
    = \norm{\vec w - \vec {\tilde w}}_{\WW_p^{(1,2)}(\Omega_\tau)}
    &\leq C_{\tau,\rho} \big( \norm{\vec c - \vec {\tilde c}}_{\WW_p^{(1,2)}(\Omega_\tau)} + \norm{\vec c^\Sigma - \vec {\tilde c}^\Sigma}_{\WW_p^{(1,2)}(\Sigma_\tau)} \big),
  \end{align*}
 see Lemma~\ref{lem:tau-dep-embedding-constant} below.
 On the other hand, concerning the terms on the boundary, we may similarly consider a reference function $\vec c^{\Sigma,\ast} \in \WW_p^{(1,2)}(\Sigma_\tau)^N$ defined as the solution to the quasi-autonomous surface diffusion system with source terms
  \begin{alignat*}{2}
   (\partial_t - \bb D^\Sigma \Delta_\Sigma) \vec c^{\Sigma,\ast}
    &= \vec r^\Sigma(\vec c^{\Sigma,0}) + \vec s^\Sigma(\vec c^0, \vec c^{\Sigma,0})
	\qquad &
    &\text{on } \Sigma_\tau,
    \\
   \vec c^\Sigma(0,\cdot)
    &= \vec c^{\Sigma,0}
	\qquad &
    &\text{on } \Sigma.
  \end{alignat*}
 Then $\vec w^\Sigma := \vec v^\Sigma - \vec c^{\Sigma,\ast}$ must solve the initial value problem on the surface
  \begin{alignat*}{2}
   (\partial_t - \bb D^\Sigma \Delta_\Sigma) \vec w^\Sigma
    &= \vec r^\Omega(\vec c) - \vec r^\Omega(\vec c^0) + \vec s^\Sigma(\vec c, \vec c^\Sigma) - \vec s^\Sigma(\vec c^0, \vec c^{\Sigma,0})
	\qquad &
    &\text{on } \Sigma_\tau,
    \\
   \vec w^\Sigma(0,\cdot)
    &= \vec 0
	\qquad &
    &\text{on } \Sigma.
  \end{alignat*}
 For this quasi-autonomous linear system, similar to the situation in the bulk phase, we may derive estimates of the form
  \begin{align*}
   \norm{\vec w^\Sigma}_{\WW_p^{(1,2)}(\Sigma_\tau)}
    &\leq C_{\tau,\rho} \big( \norm{\vec c}_{\WW_p^{(1,2)}(\Omega)} + \norm{\vec c^\Sigma}_{\WW_p^{(1,2)}(\Sigma_\tau)} + 1 \big),
    \\
   \norm{\vec w^\Sigma - \vec {\tilde w}^\Sigma}_{\WW_p^{(1,2)}(\Sigma_\tau)}
    &\leq C_{\tau,\rho} \big( \norm{\vec c - \vec {\tilde c}}_{\WW_p^{(1,2)}(\Omega_\tau)} + \norm{\vec c^\Sigma - \vec {\tilde c}^\Sigma}_{\WW_p^{(1,2)}(\Sigma_\tau)} \big),
  \end{align*}
 where again the constants $C_{\rho,\tau} > 0$ may be chosen uniformly bounded for $\rho \in (0, \rho_0]$ and $\tau \in (0,\tau_0]$, and such that they tend to zero as $\tau \rightarrow 0+$, uniformly for $\rho \in (0, \rho_0]$.
 Putting these estimates together, we infer that for sufficiently small $\tau \in (0, \tau_0]$, the map $\Phi$ maps $D_{\rho,\tau}$ into $D_{\rho,\tau}$ and is a strict contraction.
 Hence, by Banach's contraction mapping principle, $\Phi$ has a unique fixed point $(\vec c, \vec c^\Sigma) \in D_{\rho,\tau}$, which then is the unique solution to the semilinear reaction-diffusion-sorption problem in the subclass $D_{\rho,\tau} \subseteq \fs E_p^\Omega(\tau) \times \fs E_p^\Sigma(\tau)$ for all sufficiently small $\tau \in (0, \tau_1]$ and $\rho \in (0, \rho_1]$.
 We may conclude uniqueness of the solution in the class $\fs E_p^\Omega(\tau) \times \fs E_p^\Sigma(\tau)$, using a contradiction argument and the observation that by above reasoning any two solutions will coincide at least on some small, but positive time interval:
 In fact, given two solutions $(\vec c, \vec c^\Sigma)$ and $(\vec {\tilde c}, \vec {\tilde c}^\Sigma) \in \fs E_p^\Omega(T) \times \fs E_p^\Sigma(T) \hookrightarrow \BUC([0,T]; \WW_p^{2-2/p}(\Omega)^N \times \WW_p^{2-2/p}(\Sigma)^N)$, set
  \[
   t_0
    := \sup \{ \tau \in [0,T]: \, (\vec c, \vec c^\Sigma)|_{t \in [0,\tau]} = (\vec {\tilde c}, \vec {\tilde c}^\Sigma)|_{t \in [0, \tau]} \},
  \]
 where $t_0 > 0$ follows from the previous considerations (choose $\tau > 0$ small enough, such that $(\vec c, \vec c^\Sigma)|_{t \in [0,\tau]}, (\vec {\tilde c}, \vec {\tilde c}^\Sigma)|_{t \in [0, \tau]} \in D_{\rho,\tau}$).
 Now, assume $t_0 < T$.
 Then, by continuity w.r.t.\ the $\norm{\cdot}_{\WW_p^{2-2/p}}$-norm, we find that
  \[
   (\vec c(t_0,\cdot), \vec c^\Sigma(t_0,\cdot))
    = \lim_{\tau \rightarrow t_0} (\vec c(\tau,\cdot), \vec c^\Sigma(\tau,\cdot))
    = (\vec {\tilde c}(t_0,\cdot), \vec {\tilde c}^\Sigma(t_0,\cdot)).
  \]
 By time-invariance of the problem, we then find a \emph{unique} (local-in time) solution to the problem for given initial data $(\vec c(t_0,\cdot), \vec c^\Sigma(t_0,\cdot))$, hence $(\vec c(\tau,\cdot), \vec c^\Sigma(\tau,\cdot))|_{\tau \in [t_0, t_0+\varepsilon]} = (\vec {\tilde c}(\tau,\cdot), \vec {\tilde c}^\Sigma(\tau,\cdot))|_{\tau \in [t_0, t_0+\varepsilon]}$ for some $\varepsilon > 0$, contradicting the choice of $t_0$!
 \newline
 The extension to a maximal-in-time solution can be achieved based on the standard technique employing local-in-time existence and uniqueness of strong solutions and Zorn's lemma.
 \newline
 Continuous dependence on the initial data is, essentially, a consequence of the inverse function theorem:
 As the terms $\vec s^\ast$, $\vec r^\Omega(\vec c^0)$, $\vec r^\Sigma(\vec c^{\Sigma,0})$ and $\vec s^\Sigma(\vec c^0, \vec c^{\Sigma,0})$ depend continuously on the initial data, so do the reference functions $\vec c^\ast$ and $\vec c^{\Sigma,\ast}$ as well as the remainder parts $(\vec w, \vec w^\Sigma)$, hence, ultimatively the solution $(\vec c, \vec c^\Sigma)$ itself.
 \end{proof}
 
In the proof of the local-in-time well-posedness result, we employed the following auxiliary result which tells us how the constant in the anisotropic embedding theorem depends on the time horizon $\tau > 0$.

\begin{lemma}[$\tau$-dependence of the Sobolev embedding constant]
\label{lem:tau-dep-embedding-constant}
 Let $\Omega \subseteq \R^d$ be a domain with the extension property.
 Moreover, let $p, q \in (1,\infty)$ such that $2 - \frac{d+2}{p} \geq - \frac{d+2}{q}$.
 Then, for all $\tau \in (0, \infty)$, the space $\WW_p^{(1,2)}(\Omega_\tau)$ continuously embeds into $\LL_q(\Omega_\tau)$.
 Moreover, for functions in the subclass
  \[
   \mathring \WW_p^{(1,2)}(\Omega_\tau)
    := \{ u \in \WW_p^{(1,2)}(\Omega_\tau): \, u|_{t = 0} = 0 \text{ in the trace sense} \}
  \]
 we find that
  \[
   \norm{u}_{\mathring \WW_p^{(1,2)}}
    := \norm{\partial_t u}_{\LL_p} + \norm{\nabla^2 u}_{\LL_p}
  \]
 is an equivalent norm, for which we find a $\tau$-independent constant $C = C(p, q, \Omega)$ such that
  \[
   \norm{u}_{\LL_q}
    \leq C \tau^{\frac{1}{q} - \frac{1}{p} + \frac{2}{d+2}} \norm{u}_{\mathring \WW_p^{(1,2)}}
    \quad
    \text{for all }
    \tau > 0, \, u \in \mathring \WW_p^{(1,2)}(\Omega_\tau).
  \]
\end{lemma}

\begin{proof}
 Let us first consider the case
  \[
   2 - \frac{d+2}{p}
    = - \frac{d+2}{q}.
  \]
 Continuity of the embedding then follows from the anisotropic version of the Sobolev embedding theorem.
 Let $\tau > 0$ be given and set $\lambda := \sqrt{\tau} > 0$.
 Denote by $\mathcal{E}$ an extension operator, which is both continuous from $\LL_q(\Omega)$ to $\LL_q(\R^d)$ and from $\WW_p^2(\Omega)$ to $\WW_p^2(\R^d)$.
 For any function $u \in \mathring \WW_p^{(1,2)}(\Omega_\tau)$, we denote by $U \in \WW_p^{(1,2)}(\R \times \R^d)$ the extension of $U$ defined as
  \[
   U(t, \vec x)
    := \begin{cases}
        0,
        &\text{if } t < 0 \text{ or } t > 2T,
        \\
        \mathcal{E}u(t,\vec x),
        &\text{if } t \in [0,T],
        \\
        \mathcal{E}u(2T-t,\vec x),
        &\text{if } t \in (T, 2T].
       \end{cases}
  \]
 Further, let us introduce
  \[
   \tilde U(s, \vec y)
    := U(\lambda^2 s, \lambda \vec y),
    \quad
    s \geq 0, \, \vec y \in \R^d.
  \]
 This defines functions $U, \tilde U \in \WW_p^{(1,2)}(\R \times \R^d)$ and by the scaling we have ensured that $\tilde U$ has support in $[0,2] \times \R^d$.
 We may estimate the $\LL_q$-norm of $u$ as follows
  \begin{align*}
   &\norm{u}_{\LL_q(\Omega_\tau)}
    \leq \norm{U}_{\LL_q(\R \times \R^d)}
    = \left( \int_{\R^{d+1}} \abs{U(t,\vec x)}^q \, \dd (t, \vec x) \right)^{1/q}
    \\
    &= \left( \int_{\R^{d+1}} \abs{\tilde U(\tfrac{t}{\lambda^2},\tfrac{\vec x}{\lambda})}^q \, \dd (t, \vec x) \right)^{1/q}
    = \lambda^{\frac{d+2}{q}} \left( \int_{\R^{d+1}} \abs{\tilde U(s,\vec y)}^q \, \dd (s, \vec y) \right)^{1/q}
    = \lambda^{\frac{d+2}{q}} \norm{\tilde U}_{\LL_q(\R^{d+1})}
    \\
    &\leq C_\mathrm{Sob} \lambda^{\frac{d+2}{q}} \norm{\tilde U}_{\mathring \WW_p^{(1,2)}(\R \times \R^d)}
    \\
    &= C_\mathrm{Sob} \lambda^{\frac{d+2}{q}} \big\{ \big( \int_{\R^{d+1}} \abs{\partial_s \tilde U(s,\vec y)}^p \, \dd (s, \vec y) \big)^{1/p} 
     + \big( \int_{\R^{d+1}} \abs{\nabla_{\vec y}^2 \tilde U(s,\vec y)}^p \, \dd (s, \vec y) \big)^{1/p} \big\}
     \\
    &= C_\mathrm{Sob} \lambda^{\frac{d+2}{q}} \cdot \lambda^2 \big\{ \big( \int_{\R^{d+1}} \abs{\partial_t U(\lambda^2 s, \lambda \vec y)}^p \, \dd (s, \vec y) \big)^{1/p} 
     + \big( \int_{\R^{d+1}} \abs{(\nabla_{\vec x}^2 U)(\lambda^2 s,\lambda \vec y)}^p \, \dd (s, \vec y) \big)^{1/p} \big\}
     \\
    &= C_\mathrm{Sob} \lambda^{\frac{d+2}{q}} \cdot \lambda^2 \cdot \lambda^{- \frac{d+2}{p}} \big\{ \big( \int_{\R^{d+1}} \abs{\partial_t U(t, \vec x)}^p \, \dd (t, \vec x) \big)^{1/p} 
     + \big( \int_{\R^{d+1}} \abs{(\nabla_{\vec x}^2 U)(t, \vec x)}^p \, \dd (t, \vec x) \big)^{1/p} \big\}
     \\
    &= C_\mathrm{Sob} \tau^{1 - \frac{d+2}{2p} + \frac{d+2}{2q}} \norm{U}_{\mathring \WW_p^{(1,2)}}
    \leq C' \tau^{1 - \frac{d+2}{2p} + \frac{d+2}{2q}} \norm{u}_{\mathring \WW_p^{(1,2)}}
    = C' \norm{u}_{\mathring \WW_p^{(1,2)}}.
  \end{align*}
 For the general case $2 - \frac{d+2}{p} > - \frac{d+2}{q}$, we choose $q^\ast \in (q,\infty)$ and $p^\ast \in (1, p)$ such that
  \[
   2 - \frac{d+2}{p^\ast}
    = - \frac{d+2}{q^\ast}.
  \]
 As we have just seen, then
  \[
   \norm{u}_{\LL_{q^\ast}(\Omega_\tau)}
    \leq C' \norm{u}_{\mathring \WW_{p^\ast}^{(1,2)}}.
  \]
 Thus, choosing $r, s \in (1, \infty]$ such that $\frac{1}{r} := \frac{1}{q} - \frac{1}{q^\ast}$ and $\frac{1}{s} = \frac{1}{p^\ast} - \frac{1}{p}$, we find by Hölder's inequality that
  \begin{align*}
   \norm{u}_{\LL_q(\Omega_\tau)}
    &\leq \norm{u}_{\LL_{q^\ast}(\Omega_\tau)} \norm{1}_{\LL_r(\Omega_\tau)}
    \\
    &\leq C' (\abs{\Omega} \cdot \tau)^{1/r} \norm{u}_{\mathring \WW_{p^\ast}^{(1,2)}}
    \leq C' (\abs{\Omega} \cdot \tau)^{\frac{1}{q} - \frac{1}{p} + \frac{2}{d+2}} \norm{u}_{\mathring \WW_p^{(1,2)}} \norm{1}_{\LL_s(\Omega_\tau)}
    \\
    &\leq C' (\abs{\Omega} \cdot \tau)^{1/r + 1/r} \norm{u}_{\mathring \WW_p^{(1,2)}}
    = C' (\abs{\Omega} \cdot \tau)^{\frac{1}{q} - \frac{1}{p} - \frac{2}{d+2}} \norm{u}_{\mathring \WW_p^{(1,2)}}.
  \end{align*}
\end{proof}
 
 \section{Auxiliary estimates for parabolic bulk-surface systems based on the duality method}
 \label{sec:auxiliary-results}
 
 Before we prove our global-in-time existence result, we provide some preliminary results, which we later on will employ to establish global-in time existence of strong solutions.
 Some of these results are well-known and have already proved very useful for reaction-diffusion-systems with, say, homogeneous Neumann boundary conditions.
 Others can be seen as a generalisation of auxiliary results in \cite{BKMS17} by which we remove some of the restrictions imposed there, especially when passing from duality-method based estimates in the bulk phase to estimates for the corresponding boundary trace terms.
 \newline
 We start with the comparison principle, which for the pure bulk case, is a well-established tool for proving global-in-time existence for parabolic equations, in particular for reaction-diffusion-systems with a triangular structure, cf.\ M.~Pierre's survey \cite{Pierre_2010}.

 \begin{lemma}[Positivity principle]
  \label{lem:comparison_principle}
  Let a time horizon $T > 0$, parameters $p \in (1, \infty)$, $\alpha, \beta > 0$, and strictly positive $d \in \CC^1(\overline{\Omega})$, $d^\Sigma \in \CC^1(\Sigma)$ be given.
   \begin{enumerate}
    \item
     Let $f \in \LL_p(\Omega_T)$, $g^\Sigma \in \fs \WW_p^{(1,2) \cdot (\frac{1}{2} - \frac{1}{2p})}(\Sigma_T)$, $v_0 \in \WW_p^{2-2/p}(\Omega)$ with $- \beta \partial_\nu v_0 = g^\Sigma(0,\cdot)$ (on $\Sigma$ and in the trace sense) if $p > 3$.
     If $v \in \WW_p^{(1,2)}(\Omega_T)$ is a strong $\WW_p^{(1,2)}$--solution of the initial boundary value problem
      \begin{alignat*}{2}
       \alpha \partial_t v + \vec v \cdot \nabla v - \beta \dv (d \nabla v)
        &= f
	    \qquad &
        &\text{on } \Omega_T,
        \\
       - \beta \partial_\nu v
        &= g^\Sigma
	    \qquad &
        &\text{on } \Sigma_T,
        \\
       v(0)
        &= v_0
	    \qquad &
        &\text{in } \Omega,
      \end{alignat*}
     where $\vec v \in \fs U_{\tilde p}^\Omega(T)$ with $\tilde p > d + 2$, then $f, v_0 \geq 0$ and $g^\Sigma \leq 0$ imply $v \geq 0$.
    \item
     Let $f^\Sigma \in \LL_p(\Sigma_T)$, $v_0 \in \WW_p^{2-2/p}(\Sigma)$.
     If $v^\Sigma \in \WW_p^{(1,2)}(\Sigma_T)$ is a strong solution of the initial value problem
      \begin{alignat*}{2}
       \alpha \partial_t v^\Sigma - \beta \dv_\Sigma (d^\Sigma \nabla_\Sigma v^\Sigma)
        &= f^\Sigma
	    \qquad &
        &\text{on } \Sigma_T,
        \\
       v^\Sigma(0,\cdot)
        &= v_0^\Sigma
	    \qquad &
        &\text{on } \Sigma,
      \end{alignat*}
     where $f^\Sigma, v_0^\Sigma \geq 0$, then $v^\Sigma \geq 0$.
   \end{enumerate}
 \end{lemma}
 \begin{proof}
  This result can be established analogously to \cite[Lemma 5.2]{BKMS17}.
 \end{proof}

 The next auxiliary result may be seen as a slightly generalised version of \cite[Lemma 6.2]{BKMS17}.
 The main difference is that we do not have to restrict to the cases $p, q \geq 2$, but all $p, q > 1$ are allowed.
 To establish the result, we combine classical duality techniques, cf.\ M.~Pierre's survey \cite{Pierre_2010}, with interpolation methods.
 
 \begin{lemma}[Dual estimate]
  \label{lem:dual-estimate}
  Let $p, q \in (1, \infty]$, $T > 0$ and $\alpha, \beta > 0$ be fixed, and let $\tau \in (0,T]$.
  Assume that $\vec v \in \fs U_{\tilde p}^\Omega(T)$ with $\tilde p > d + 2$.
  Let $f \in \LL_p(\Omega_\tau)$, $g^\Sigma \in \WW_p^{(1,2) \cdot (\frac{1}{2} - \frac{1}{2p})}(\Sigma_\tau)$ and $v_0 \in \WW_p^{2-2/p}(\Omega)$ be given, and assume that the compatibility condition $- \beta \partial_{\vec \nu} v_0 = g^\Sigma(0)$ is satisfied if $p > 3$.
  Then there is a constant $C = C(q,T) > 0$ such that for every strong solution $v \in \WW_p^{(1,2)}(\Omega_\tau)$ of the initial-boundary value problem
    \begin{alignat}{2}
     \alpha \partial_t v + \vec v \cdot \nabla v - \beta \dv (d \nabla v)
      &= f
      \qquad &
      &\text{on } \Omega_\tau,
      \nonumber \\
     - \beta \partial_\nu v
      &= g^\Sigma
      \qquad &
      &\text{on } \Sigma_\tau,
      \label{eqn:dual-estimate}
      \\
     v(0)
      &= v_0
      \qquad &
      &\text{in } \Omega,
      \nonumber
    \end{alignat}
   the following estimate is valid:
    \[
     \norm{v}_{\LL_q(\Omega_\tau)} + \norm{v}_{\LL_q(\Sigma_\tau)} + \norm{v(\tau,\cdot)}_{\LL_q(\Omega)}
      \leq C \left( \norm{f}_{\LL_q(\Omega_\tau)} + \norm{g^\Sigma}_{\LL_q(\Sigma_\tau)} + \norm{v_0}_{\LL_q(\Omega)} \right).
    \]
 \end{lemma}
 
 \begin{remark}
  Note that, in particular, the constant $C > 0$ does neither depend on the data $f$, $v_0$ and $g^\Sigma$, nor on the time horizon $\tau \in (0,T]$.
  \newline
  For the case $p, q \in [2, \infty]$, a quite similar result has been derived in \cite[Lemma 6.2]{BKMS17} (originally, for cylindrical domains, but the proof easily extends to general $\CC^{2-}$-domains). We employ the following strategy: For $p > \frac{d+2}{2}$, i.e.\ the space $\WW_p^{(1,2)}((0,T) \times \Omega)$ continuously embeds into the space of continuous functions on $[0,T] \times \overline{\Omega}$, we first consider the case $q \geq 2$ and distinguish the three subcases $q = 2$, $q = \infty$ and $q \in (2, \infty)$:
  Here, the Hilbert space case $q = 2$ can be handled by multiplying the first line of \eqref{eqn:dual-estimate} by $v$, integrating over $\Omega$ and using integration by parts. This gives an estimate on the bulk term $\norm{v}_{\LL_2(\Omega_\tau)}$. The boundary term $\norm{v}_{\LL_2(\Sigma_\tau)}$ can be estimated by employing trace results on Bessel potential spaces $\HH^s(\Omega) \hookrightarrow \HH^{s-\nicefrac{1}{2}}(\Sigma)$ for $s \in (\frac{1}{2}, \frac{3}{2})$, combined with an interpolation inequality for $s \in (\frac{1}{2}, 1)$ and the estimate established for the bulk term.
  The second case $q = \infty$ is easier and can be handled using the estimate $\norm{u}_{\LL_\infty(\Sigma_\tau)} \leq \norm{u}_{\LL_\infty(\Omega_\tau)}$, which is trivial for continuous functions $u$.
  By the Riesz--Thorin Interpolation Theorem, we refer to e.g.\ \cite[Theorem 2.6]{Lunardi_2018} for a version which is general enough to be applicable here, the estimates for intermediate interpolation parameters $q \in (2, \infty)$ then follow as well.
  \newline
  In the case $p \leq \frac{d+2}{2}$ we may approximate the initial and boundary data and the source terms by regularised data, so that the results for the case $p > \frac{d+2}{2}$ become applicable.
  \newline
  On the other hand, in several previous works on bulk reaction-diffusion-advection systems, dual estimates have been heavily used as a tool to establish global-in-time existence for particular classes of reaction-diffusion-advection systems, cf., e.g., \cite[Lemma 3.4]{Pierre_2010}.
 To estimate the bulk term $\norm{u}_{\LL_q(\Omega_\tau)}$ for the range $q \in (1, \infty)$ via the \emph{duality method}, we proceed as follows: We multiply the evolutionary PDE by a function $\theta$, which is chosen in such a way that it solves an evolutionary problem which is \emph{dual} to the original initial-boundary value problem.
 (In particular, it evolves \emph{backwards in time} and is subject to \emph{homogeneous boundary conditions}.)
  This method relies on the fact that the dual space $(\LL_q(\Omega_\tau))^\ast$ of the Lebesgue space $\LL_q(\Omega_\tau)$ is (up to an isometric isomorphism) the Lebesgue space $\LL_{q'}(\Omega_\tau)$ for the Hölder dual exponent $q' \in (1,\infty)$ which is given by $\frac{1}{q'} := 1 - \frac{1}{q}$, whenever $q \in (1, \infty)$ belongs to the reflexive range.
  \newline
  Combining these methods, we may allow for $p, q$ from the full reflexive range $(1, \infty)$: First, we employ the duality method to get the estimate on the bulk term. Thereafter, we adjust the method used in \cite[Lemma 6.2]{BKMS17} to cover the general case $q \in (1, \infty)$.
 \end{remark}
 \begin{proof}
  The case $q = \infty$ can be handled exactly as in the proof of \cite[Lemma 6.2.]{BKMS17}.
  Moreover, with the techniques used in the proof of \cite[Lemma 3.4]{Pierre_2010} and for $q \in (2, \infty)$, one may easily establish the (weaker) estimate $\norm{v}_{\LL_q(\Omega_\tau)} \leq C_{q,T} \big( \norm{v^0}_{\LL_q(\Omega)} + \norm{f}_{\LL_q(\Omega_\tau)} + \norm{g^\Sigma}_{\LL_q(\Sigma_\tau)} \big)$ for every $q \in (1, \infty)$ and $\tau \in (0,T]$.
  We, therefore, focus on the boundary term $\norm{v|_\Sigma}_{\LL_q(\Sigma_T)}$ and show that it can be estimated by the right-hand side as well.
  To this end, for the moment, we will additionally assume that $v \in \WW_q^{(1,2)}(\Omega_\tau)$, e.g.\ by approximating $f$, $g^\Sigma$ and $v_0$ by smooth data, and multiply the first line in \eqref{eqn:dual-estimate} by $v \abs{v}^{q-2} := \sign(v) \abs{v}^{q-1} \in \LL_{q'}(\Omega_\tau)$ (which, up to a scalar factor and identification of the dual space $\LL_q(\Omega_\tau)^\ast$ with $\LL_{q'}(\Omega_\tau)$, is the unique element in the duality set of $v \in \LL_q(\Omega_\tau)$), then integrate over $\Omega_\tau$ and use equation \eqref{eqn:RDASS} and integration by parts, to calculate that
   \begin{align*}
    &\int_{\Omega_\tau} f v \abs{v}^{q-2} \dd(t,\vec x)
     = \int_{\Omega_\tau} \big( \partial_t v + \vec v \cdot \nabla v - \beta \dv (d \nabla v) \big) v \abs{v}^{q-2} \dd(t,\vec x)
     \\
     &= \int_{\Omega_\tau} \big( \frac{1}{q} \partial_t \abs{v}^q 
      + \frac{1}{q} \vec v \cdot \nabla \abs{v}^q \big) \dd(t,\vec x)
      - \int_{\Omega_\tau} \beta \dv (d \nabla v) \cdot v \abs{v}^{q-2} \dd(t,\vec x)
      \\
     &= \left[ \frac{1}{q} \int_\Omega \abs{v}^q \dd \vec x \right]_{t=0}^{t=\tau}
      - \frac{1}{q} \int_{\Omega_\tau} \dv \vec v \cdot \abs{v}^q \dd(t,\vec x)
      + \frac{1}{q} \int_{\Sigma_\tau} (\vec v \cdot \vec \nu) \abs{v}^q \dd(t,\sigma(\vec x))
      \\
      & \quad
      + \int_{\Omega_\tau} \beta d \nabla v \cdot \nabla (v \cdot \abs{v}^{q-2}) \dd(t,\vec x)
      - \int_{\Sigma_\tau} \beta \partial_{\vec \nu} v \cdot v \abs{v}^{q-2} \dd(t,\sigma(\vec x))
      \\
     &= \frac{1}{q} \big[ \norm{v(\tau,\cdot)}_{\LL_q(\Omega)}^q - \norm{v(0,\cdot)}_{\LL_q(\Omega)}^q \big]
      + \int_{\Omega_\tau} \beta d (q-1) \abs{\nabla v}^2 \abs{v}^{q-2} \dd(t,\vec x)
      \\ &\qquad
      + \int_{\Sigma_\tau} g^\Sigma \, v \abs{v}^{q-2} \dd(t,\sigma(\vec x)),
   \end{align*}
  provided these integrals exist.
  The latter will follow for general $q \in (1, \infty)$ from Lemma~\ref{lem:surface-term-estimate} below.
  Hence, after rearrangement of the terms, we obtain the estimate
   \begin{align*}
    &\frac{1}{q} \norm{v(\tau,\cdot)}_{\LL_q(\Omega)}^q
     + \int_{\Omega_\tau} \beta d (q-1) \abs{\nabla v}^2 \abs{v}^{q-2} \dd(t,\vec x)
     \\
     &= \frac{1}{q} \norm{v(0,\cdot)}_{\LL_q(\Omega)}^q
     + \int_{\Omega_\tau} f v \abs{v}^{q-2} \dd(t,\vec x)
     - \int_{\Sigma_\tau} g^\Sigma \, v \abs{v}^{q-2} \dd(t,\sigma(\vec x))
     \\
     &\leq \frac{1}{q} \norm{v(0,\cdot)}_{\LL_q(\Omega)}^q
     + \int_{\Omega_\tau} \abs{f} \abs{v}^{q-1} \dd(t,\vec x)
     + \int_{\Sigma_\tau} \abs{g^\Sigma} \abs{v}^{q-1} \dd(t,\sigma(\vec x)).
   \end{align*}
  The desired estimate follows from this, if we can ensure that for every $\varepsilon > 0$ there is $C = C(q,\varepsilon) > 0$ (independent of $v$) such that
   \begin{equation}
    \varepsilon \norm{v|_{\Sigma_\tau}}_{\LL_q(\Sigma)}^q
     \leq \int_{\Omega_\tau} \abs{\nabla v}^2 \abs{v}^{q-2} \dd (t, \vec x)
      + C(q,\varepsilon) \norm{v}_{\LL_q(\Omega_\tau)}^q,
     \label{eqn:dual-estimate-2}
   \end{equation}
  as we may then employ the generalized Hölder inequality to establish
   \begin{align*}
    \int_{\Omega_\tau} \abs{f} \abs{v}^{q-1} \dd(t,x)
     &\leq \norm{f}_{\LL_q(\Omega_\tau)} \norm{v}_{\LL_q(\Omega_\tau)}^{q-1},
     \\
    \int_{\Sigma_\tau} \abs{g^\Sigma} \abs{v}^{q-1} \dd(t,\sigma(x))
     &\leq \norm{g^\Sigma}_{\LL_q(\Sigma_\tau)} \norm{v}_{\LL_q(\Sigma_\tau)}^{q-1}
   \end{align*}
  and the statement follows by Young's inequality for products.
  To get estimate \eqref{eqn:dual-estimate-2}, we may employ the subsequent Lemma \ref{lem:surface-term-estimate}.
  \newline
  Finally, since the constant $C > 0$ is independent of the particular choice of the data, by approximation these estimates are also valid for less smooth data $(f, g^\Sigma, v_0)$, i.e.\ the estimate is valid for all such strong $\WW_p^{(1,2)}$-solutions.
 \end{proof}
 
 \begin{lemma}
 \label{lem:surface-term-estimate}
  Let $q \in (1, \infty)$, $\tau > 0$ and $\Omega \subsetneq \R^d$ be a bounded $\CC^{1,1}$-domain with boundary $\Sigma = \partial \Omega$.
  Denote by $\cdot|_\Sigma: \WW^{(1,2)}_q(\Omega_\tau) \rightarrow \LL_q(\Sigma_\tau)$ the (continuous) spatial boundary trace operator.
  Then, for every $s \in (\frac{1}{2}, 1)$, there is a constant $C = C(q,s,\Omega) > 0$ such that
   \begin{align*}
    \norm{v|_\Sigma}_{\LL_q(\Sigma_\tau)}^q
     &\leq C \norm{ \abs{v}^{\nicefrac{q}{2}} }_{\LL_2((0,\tau);\HH^1(\Omega))}^{2s} \norm{v}_{\LL_q(\Omega_\tau)}^{(1-s)q}
     \\
     &\leq C \big( \big( \int_{\Omega_\tau} \abs{v}^{q-2} \abs{\nabla v}^2 \dd (t,\vec x) \big)^{1/2} + \norm{v}_{\LL_q(\Omega_\tau)}^{q/2} \big)^{2s} \norm{v}_{\LL_q(\Omega_\tau)}^{(1-s)q},
     \quad
     v \in \WW^{(1,2)}_q(\Omega_\tau).
   \end{align*}
 \end{lemma}
 \begin{proof}
 As a preliminary step, we will derive the following identity, which is valid for every $u \in \CC^2(\overline{\Omega})$:
     \begin{equation}
   \int_{\Omega_\tau} \Delta u \cdot u \abs{u}^{q-2} \dd(t,\vec x)
    = - \int_{\Omega_\tau} \nabla u \cdot \nabla \big( u \abs{u}^{q-2} \big) \dd(t, \vec x)
     + \int_{\Sigma_\tau} \partial_{\vec \nu} u \cdot u \abs{u}^{q-2} \dd(t,\sigma(\vec x)).
   \label{eqn:test_function_CC^2}
  \end{equation}
 This identity does not immediately follow from the Gauß divergence theorem, since for $q \in (1,2)$ the term $u \abs{u}^{q-2}$ will in general \emph{not} be classically differentiable, nor can we ensure that its weak derivative lies in the space $\LL_{q'}(\Omega_\tau)$.
 Therefore, we will approximate the integrands by a regularised version and, instead of $\int_{\Omega_\tau} \nabla u \cdot \nabla (u \abs{u}^{q-2}) \dd (t,\vec x)$, we consider the integral
  \begin{align*}
   \int_{\Omega_\tau} \nabla u \cdot \nabla \big( u (u^2 + \varepsilon^2)^{\frac{q-2}{2}} \big) \dd (t,\vec x)
    &= \int_{\Omega_\tau} \abs{\nabla u}^2 (u^2 + \varepsilon^2)^{\frac{q-2}{2}} \big[ 1 + (q-2) \abs{u}^2 (u^2 + \varepsilon^2)^{-1} \big] \dd (t,\vec x).
  \end{align*}
 The terms in the integrand on the right-hand side are all positive for each $\varepsilon > 0$ and decrease in $\varepsilon > 0$, thus, for the critical case $q \in (1,2)$ we may use the Beppo--Levi theorem on montone convergence to conclude that
  \begin{align*}
   &\int_{\Omega_\tau} \abs{\nabla u}^2 (u^2 + \varepsilon^2)^{\frac{q-2}{2}} \big[ 1 + (q-2) \abs{u}^2 (u^2 + \varepsilon^2)^{-1} \big] \dd (t,\vec x)
   \\ &\qquad
   \xrightarrow{\varepsilon \rightarrow 0+} \int_{\Omega_\tau} \nabla u \cdot \nabla (u \abs{u}^{q-2}) \dd (t,\vec x)
   \in [0, \infty].
  \end{align*}
 For the regularised integral, we may use the Gauß divergence theorem, and obtain that
  \begin{align*}
   &\int_{\Omega_\tau} \nabla u \cdot \nabla \big( u (u^2 + \varepsilon^2)^{\frac{q-2}{2}} \big) \dd (t,\vec x)
    \\
    &= - \int_{\Omega_\tau} \Delta u \cdot u (u^2 + \varepsilon^2)^{\frac{q-2}{2}} \dd(t,\vec x)
     + \int_{\Sigma_\tau} \partial_{\vec \nu} u \cdot u (u^2 + \varepsilon^2)^{\frac{q-2}{2}} \dd(t,\sigma(\vec x)).
  \end{align*}
 Since we may estimate
  \begin{align*}
   \norm{u (u^2 + \varepsilon^2)^{\frac{q-2}{2}}}_{\LL_{q'}(\Omega_\tau)}^{q'}
    = \int_{\Omega_\tau} \big( \frac{\abs{u}}{\sqrt{u^2 + \varepsilon^2}} \big)^{q'} (u^2 + \varepsilon^2)^{\frac{q}{2}} \dd (t, \vec x)
    \leq \int_{\Omega_\tau} (u^2 + \varepsilon^2)^{\frac{q}{2}} \dd (t, \vec x)
  \end{align*}
  and, using that $q' = \frac{q}{q-1}$, so that $\frac{q-1}{2} \cdot q' = \frac{q}{2}$,
  \begin{align*}
   \norm{u (u^2 + \varepsilon^2)^{\frac{q-2}{2}}}_{\LL_{q'}(\Sigma_\tau)}^{q'}
    = \int_{\Sigma_\tau} \big( \frac{\abs{u}}{\sqrt{u^2 + \varepsilon^2}} \big)^{q'} (u^2 + \varepsilon^2)^{\frac{q}{2}} \dd(t,\sigma(\vec x))
    \leq \int_{\Sigma_\tau} (u^2 + \varepsilon^2)^{\frac{q}{2}} \dd(t,\sigma(\vec x)),
  \end{align*}
 we may conclude by Lebesgue's theorem on dominated convergence that
  \begin{align*}
   \int_{\Omega_\tau} \Delta u \cdot u (u^2 + \varepsilon^2)^{\frac{q-2}{2}} \dd(t,\vec x)
    &\xrightarrow{\varepsilon \rightarrow 0+}
    \int_{\Omega_\tau} \Delta u \cdot u \abs{u}^{q-2} \dd(t, \vec x),
    \\
   \int_{\Sigma_\tau} \partial_{\vec \nu} u \cdot u (u^2 + \varepsilon^2)^{\frac{q-2}{2}} \dd(t,\sigma(\vec x))
    &\xrightarrow{\varepsilon \rightarrow 0+}
    \int_{\Sigma_\tau} \partial_{\vec \nu} u \abs{u}^{q-2} \dd(t,\sigma(\vec x))
  \end{align*}
 and identity \eqref{eqn:test_function_CC^2} follows. 
 \newline
 Having identity \eqref{eqn:test_function_CC^2} at hand, we may proceed as follows: Via the boundary trace map, the space $\WW^{(1,2)}_q(\Omega_\tau)$ embeds continuously into $\LL_q((0,\tau);\WW_q^1(\Sigma)) \subseteq \WW_q^{(1/2,1)}(\Sigma_\tau)$ by comparing the parabolic Sobolev indices $2 - \frac{d+2}{q} \geq 1 - \frac{d+1}{q}$ (which is true for all $q \geq 1$) and, therefore, as $q' = \frac{q}{q-1}$ and $\abs{u \abs{u}^{q-2}}^{q'} = \abs{u}^q$ it follows that $u \abs{u}^{q-2} \in \LL_{q'}(\Sigma_\tau)$ for every $u \in \WW_q^{(1,2)}(\Omega_\tau)$ with $\norm{u \abs{u}^{q-2}}_{\LL_{q'}(\Sigma_\tau)} \leq C \norm{u}_{\WW_q^{(1,2)}(\Omega_\tau)}^{q-1}$ and $u \abs{u}^{q-2} \in \LL_{q'}(\Omega_\tau)$ with $\norm{u \abs{u}^{q-2}}_{\LL_{q'}(\Omega_\tau)} \leq C \norm{u}_{\LL_q(\Omega_\tau)}^{q-1}$.
 Further, $\nabla u \cdot \nabla (u \abs{u}^{q-2}) = (q-1) \abs{\nabla u}^2 \abs{u}^{q-2} \geq 0$.
 Approximating $u \in \WW_q^{(1,2)}(\Omega_\tau)$ by functions in $\CC^2(\overline{\Omega}_\tau)$ and employing Hölder's inequality for the left-hand side in \eqref{eqn:test_function_CC^2} and the boundary term on the right hand side, we may thus infer that the identity \eqref{eqn:test_function_CC^2} is valid for all $u \in \WW_q^{(1,2)}(\Omega_\tau)$, the integral $\int_{\Omega_\tau} \nabla u \cdot \nabla (u \abs{u}^{q-2}) \dd(t,\vec x)$ exists for all $u \in \WW_q^{(1,2)}(\Omega_\tau)$ and can be bounded by $C \norm{u}_{\WW_q^{(1,2)}(\Omega_\tau)}^q$.
  \newline
  It follows that $\abs{v}^{q/2} \in \LL_2((0,T);\WW^1_2(\Omega))$ with $\norm{ \abs{v}^{q/2} }_{\LL_2(\Omega_\tau)} = \norm{v}_{\LL_q(\Omega_\tau)}^{q/2}$ and $\nabla \abs{v}^{q/2} = \frac{q}{2} v \abs{v}^{q/2-2} \nabla v$ with $\LL_2$-norm $\norm{ \nabla \abs{v}^{q/2} }_{\LL_2(\Omega_\tau)} = \frac{q}{2} \norm{ \abs{v}^{q/2-1} \nabla v }_{\LL_2(\Omega_\tau)}$.
  \newline
  Moreover, the $\Sigma$-trace of $\abs{v}^{q/2} \in \LL_2((0,T);\WW_2^1(\Omega))$ is $\big( \abs{v}^{q/2} \big)|_\Sigma = \abs{v|_\Sigma}^{q/2} \in \LL_2(\Sigma_\tau)$, so that
   \begin{align*}
    \norm{v|_\Sigma}_{\LL_q(\Sigma_\tau)}^q
     &= \norm{\abs{v}^{\nicefrac{q}{2}}|_\Sigma}_{\LL_2(\Sigma_\tau)}^2
     \leq C \norm{\abs{v}^{\nicefrac{q}{2}}|_\Sigma}_{\LL_2((0,\tau);\HH^{s-1/2}(\Sigma))}^2
     \\
     &\leq C \norm{\abs{v}^{\nicefrac{q}{2}}}_{\LL_2((0,\tau);\HH^s(\Omega))}^2
     \leq C \norm{\abs{v}^{\nicefrac{q}{2}}}_{\LL_2((0,\tau);\HH^1(\Omega))}^{2s} \norm{\abs{v}^{\nicefrac{q}{2}}}_{\LL_2((0,\tau);\LL_2(\Omega))}^{2(1-s)}
     \\
     &= C \norm{\abs{v}^{\nicefrac{q}{2}}}_{\LL_2((0,\tau);\HH^1(\Omega))}^{2s} \norm{v}_{\LL_q(\Omega_\tau)}^{(1-s)q}.
   \end{align*}
 \end{proof}

Adjusting the norm for which we want to estimate the surface trace $v|_\Sigma$ in the previous result, we obtain the following variant using the duality method, cf.\ \cite{BoFiPiRo_2017}.

 \begin{lemma}[Dual estimate (Variant)]
  \label{lem:dual-estimate-variant}
  Let $p, q \in (1, \infty]$, $T > 0$ and $\alpha, \beta > 0$ be fixed, and let $\tau \in (0,T]$.
  Assume that $\vec v \in \fs U_{\tilde p}^\Omega(T)$ with $\tilde p > d + 2$.
  Let $f \in \LL_p(\Omega_\tau)$, $g^\Sigma \in \WW_p^{(1,2) \cdot (\frac{1}{2} - \frac{1}{2p})}(\Sigma_\tau)$ and $v_0 \in \WW_p^{2-2/p}(\Omega)$ be given and assume that $- \beta \partial_{\vec \nu} v_0|_\Sigma = g^\Sigma|_{t=0}$ if $p > 3$.
  Then there is a constant $C = C(q,T) > 0$ such that for every strong solution $v \in \WW_p^{(1,2)}(\Omega_\tau)$ of the initial-boundary value problem
    \begin{alignat}{2}
     \alpha \partial_t v + \vec v \cdot \nabla v - \dv (d \nabla v)
      &= f
	  \qquad &
      &\text{on } \Omega_\tau,
      \nonumber \\
     - d \partial_\nu v
      &= g^\Sigma
	  \qquad &
      &\text{on } \Sigma_\tau,
      \label{eqn:dual-estimate-variant}
      \\
     v(0)
      &= v_0
	  \qquad &
      &\text{in } \Omega,
      \nonumber
    \end{alignat}
   the following estimate is valid:
    \begin{align*}
     &\norm{v}_{\LL_q(\Omega_\tau)} + \norm{v|_\Sigma}_{\mathring \WW_{q'}^{(1,2) \cdot (\frac{1}{2} - \frac{1}{2q'})}(\Sigma_\tau)^\ast}
      \\ &\qquad
      \leq C \left( \norm{f}_{\WW_{q'}^{(1,2)}(\Omega_\tau)^\ast} + \norm{g^\Sigma}_{\mathring \WW_{q'}^{(1,2) \cdot (1 - \frac{1}{2q'})}(\Sigma_\tau)^\ast} + \norm{v^0}_{\WW_{q'}^{2-2/q'}(\Omega)^\ast} \right),
    \end{align*}
  where we write
   \begin{align*}
    \mathring \WW_{q'}^{(1,2) \cdot (\frac{1}{2} - \frac{1}{2q'})}(\Sigma_\tau)
     &:= \begin{cases}
          \{ g \in \WW_{q'}^{(1,2) \cdot (\frac{1}{2} - \frac{1}{2q'})}(\Sigma_\tau): \, g|_{t = \tau} = 0 \},
          &q < \frac{3}{2},
          \vspace{2pt}
          \\
          \WW_{q'}^{(1,2) \cdot (\frac{1}{2} - \frac{1}{2q'})}(\Sigma_\tau),
          &q \geq \frac{3}{2},
         \end{cases}
     \\
    \mathring \WW_{q'}^{(1,2) \cdot (1 - \frac{1}{2q'})}(\Sigma_\tau)
     &:= \begin{cases}
          \{ h \in \WW_{q'}^{(1,2) \cdot (1 - \frac{1}{2q'})}(\Sigma_\tau): \, h|_{t = \tau} = 0 \},
          &q < 3,
          \vspace{2pt}
          \\
          \WW_{q'}^{(1,2) \cdot (1 - \frac{1}{2q'})}(\Sigma_\tau),
          &q \geq 3.
         \end{cases}
   \end{align*}
 \end{lemma}
 
\begin{proof}
 In the spirit of the duality technique, let arbitrary $\Theta \in \LL_{q'}(\Omega_\tau)$ and $\eta \in \WW_{q'}^{(1,2) \cdot (\frac{1}{2} - \frac{1}{2q'})}(\Sigma_\tau)$ such that (if $q < \frac{3}{2}$) $\eta|_{t = \tau} = 0$ be given and let $\theta \in \WW_{q'}^{(1,2)}(\Omega_\tau)$ be the unique solution to the (time-reversed) dual problem
  \[
   \begin{cases}
    - \alpha \partial_t \theta - \vec v \cdot \nabla \theta - \dv(d \nabla \theta)
     = \Theta \quad
     &\text{in } \Omega_\tau,
     \\
    - d \partial_{\vec \nu} \theta
     = \eta \quad
     &\text{on } \Sigma_\tau,
     \\
    \theta(\tau,\cdot)
     = 0 \quad
     &\text{in } \Omega.
   \end{cases}
  \]
 Thanks to $\LL_{q'}$-maximal regularity, this solution $\theta \in \WW_{q'}^{(1,2)}(\Omega_\tau)$ exists and is uniquely determined.
 Moreover, thanks to the condition $\theta(\tau,\cdot) = 0$ there is a constant $C = C(q,T)$, which is independent of $\tau \in (0,T]$ such that
  \[
   \norm{\theta}_{\WW_{q'}^{(1,2)}(\Omega_\tau)}
    \leq C \big( \norm{\Theta}_{\LL_{q'}(\Omega_\tau)} + \norm{\eta}_{\WW_{q'}^{(1,2) \cdot (\frac{1}{2} - \frac{1}{2q'})}(\Sigma_\tau)} \big).
  \]
 Using integration by parts and Green's formula, we may conclude that (recall that $\theta(\tau,\cdot) = 0$ in $\Omega$, $\dv \vec v = 0$ in $\Omega$ and $\partial_{\vec \nu} v \equiv 0$ on $\Sigma_\tau$)
  \begin{align*}
   &\int_{\Omega_\tau} v \Theta \dd (t,\vec x)
    + \int_{\Sigma_\tau} v \eta \dd (t, \sigma(\vec x))
    \\
    &= - \int_{\Omega_\tau} v (\alpha \partial_t \theta + \vec v \cdot \nabla \theta + \dv(d \nabla \theta)) \, \dd (t,\vec x)
     - \int_{\Sigma_\tau} v d \partial_{\vec \nu} \theta \, \dd (t, \sigma(\vec x))
     \\
    &= \int_{\Omega_\tau} (\alpha \partial_t v + \vec v \cdot \nabla v - \dv(d \nabla v)) \theta \dd (t, \vec x)
     - \int_{\Sigma_\tau} \theta d \partial_{\vec \nu} v \dd (t, \sigma(\vec x))
     - \int_{\Omega} v^0 \theta(0,\cdot) \dd \vec x
     \\
    &= \int_{\Omega_\tau} f \theta \dd (t, \vec x)
     + \int_{\Sigma_\tau} g^\Sigma \theta \dd (t, \sigma(\vec x))
     - \int_{\Omega} v^0 \theta(0,\cdot) \dd \vec x.
  \end{align*}
 The right-hand side can be estimated by
  \begin{align*}
   &\norm{f}_{\WW_{q'}^{(1,2)}(\Omega_\tau)^\ast} \norm{\theta}_{\WW_{q'}^{(1,2)}(\Omega_\tau)}
    + \norm{g^\Sigma}_{\mathring \WW_{q'}^{(1,2) \cdot (1 - \frac{1}{2q'})}(\Sigma_\tau)^\ast} \norm{\theta}_{\mathring \WW_{q'}^{(1,2) \cdot (1 - \frac{1}{2q'})}(\Sigma_\tau)}
    \\ &\qquad
    + \norm{v^0}_{\WW_{q'}^{2-2/q'}(\Omega)^\ast} \norm{\theta(0,\cdot)}_{\WW_{q'}^{2-2/q'}(\Omega)}
    \\
    &\leq
    C \big( \norm{f}_{\WW_{q'}^{(1,2)}(\Omega_\tau)^\ast} + \norm{g^\Sigma}_{\mathring \WW_{q'}^{(1,2) \cdot (1 - \frac{1}{2q'})}(\Sigma_\tau)^\ast} + \norm{v^0}_{\WW_{q'}^{2-2/q'}(\Omega)^\ast} \big)
     \\ &\qquad
     \cdot \big( \norm{\Theta}_{\LL_{q'}(\Omega_\tau)} + \norm{\eta}_{\mathring \WW_{q'}^{(1,2) \cdot (1 - \frac{1}{2q'})}(\Sigma_\tau)} \big).
  \end{align*}
 Taking the supremum over all such $\Theta$ and $\eta$ with norm less or equal $1$, and using that $\LL_q(\Omega_\tau)$ is isometrically isomorphic to $\LL_{q'}(\Omega_\tau)^\ast$, we find that
  \begin{align*}
   &\norm{v}_{\LL_q(\Omega_\tau)} + \norm{v|_\Sigma}_{\mathring \WW_{q'}^{(1,2) \cdot (1 - \frac{1}{2q'})}(\Sigma_\tau)^\ast}
    \\
    &\leq C \big( \norm{f}_{\WW_{q'}^{(1,2)}(\Omega_\tau)^\ast} + \norm{g^\Sigma}_{\mathring \WW_{q'}^{(1,2) \cdot (1 - \frac{1}{2q'})}(\Sigma_\tau)^\ast} + \norm{v^0}_{\WW_{q'}^{2-2/q'}(\Omega)^\ast} \big).
  \end{align*}
\end{proof}

Similar estimates are also valid for reaction-diffusion-advection systems on the surface $\Sigma = \partial \Omega$:

\begin{lemma}[Dual estimates on the surface]
  \label{lem:surface-dual-estimate-variant}
  Let $\Sigma$ be the boundary of a bounded $\CC^{3-}$-domain $\Omega \subseteq \R^d$, $p, q \in (1, \infty]$, $T > 0$ and $\alpha, \beta > 0$ be fixed, and let $\tau \in (0,T]$.
  Let $f^\Sigma \in \LL_p(\Sigma_\tau)$ and $v^{\Sigma,0} \in \WW_p^{2-2/p}(\Sigma)$.
  Then there is a constant $C = C(q,T) > 0$ such that for every strong solution $v^\Sigma \in \WW_p^{(1,2)}(\Sigma_\tau)$ of the initial value problem
    \begin{alignat}{2}
     \partial_t v^\Sigma - \dv_\Sigma (d^\Sigma \nabla_\Sigma v^\Sigma)
      &= f^\Sigma
	  \qquad &
      &\text{on } \Sigma_\tau,
      \label{eqn:surface-dual-estimate-variant}
      \\
     v^\Sigma(0,\cdot)
      &= v^{\Sigma,0}
	  \qquad &
      &\text{on } \Sigma
      \nonumber
    \end{alignat}
   the following estimate is valid:
    \[
     \norm{v^\Sigma}_{\LL_q(\Sigma_\tau)}
      \leq C \left( \norm{f^\Sigma}_{\WW_{q'}^{(1,2)}(\Sigma_\tau)^\ast} + \norm{v^{\Sigma,0}}_{\WW_{q'}^{2-2/q'}(\Sigma)^\ast} \right).
    \]
\end{lemma}

\begin{proof}
 For arbitrary $\Theta^\Sigma \in \LL_{q'}(\Sigma_\tau)$, let us consider the dual, time-reversed problem
  \[
   \begin{cases}
    - \partial_t \Psi^\Sigma - \dv_\Sigma(d^\Sigma \nabla_\Sigma \Psi^\Sigma)
     = \Theta^\Sigma \quad
     &\text{on } \Sigma_\tau,
     \\
    \Psi^\Sigma(\tau,\cdot)
     = 0 \quad
     &\text{on } \Sigma,
   \end{cases}
  \]
 which admits $\LL_{q'}$-maximal regularity, hence, has a unique solution $\Psi^\Sigma \in \WW_{q'}^{(1,2)}(\Sigma_\tau)$ such that
  \[
   \norm{\Psi^\Sigma}_{\WW_{q'}^{(1,2)}(\Sigma_\tau)}
    \leq C \norm{\Theta^\Sigma}_{\LL_{q'}(\Sigma_\tau)}
  \]
 for some constant $C > 0$, which is not only independent of $\Theta^\Sigma$, but of $\tau \in (0, T]$ as well since we only consider zero final data $\Psi^\Sigma(\tau,\cdot) = 0$.
 We then calculate that
  \begin{align*}
   \int_{\Sigma_\tau} v^\Sigma \cdot \Theta^\Sigma \dd (t, \sigma(\vec x))
    &= - \int_{\Sigma_\tau} v^\Sigma (\partial_t \Psi^\Sigma + \dv_\Sigma (d^\Sigma \nabla_\Sigma \Psi^\Sigma)) \dd(t, \sigma(\vec x))
    \\
    &= \int_{\Sigma_\tau} (\partial_t v^\Sigma - \dv_\Sigma(d^\Sigma \nabla_\Sigma v^\Sigma)) \Psi^\Sigma \dd (t, \sigma(\vec x))
    + \int_\Sigma v^{\Sigma,0} \Psi^{\Sigma}(0,\cdot) \dd \sigma(\vec x)
    \\
    &= \int_{\Sigma_\tau} f^\Sigma \Psi^\Sigma \dd (t, \sigma(\vec x))
     + \int_\Sigma v^{\Sigma,0} \Psi^{\Sigma}(0,\cdot) \dd \sigma(\vec x),
  \end{align*}
 which we can estimate by
  \begin{align*}
   &\norm{f^\Sigma}_{\WW_{q'}^{(1,2)}(\Sigma_\tau)^\ast} \norm{\Psi^\Sigma}_{\WW_{q'}^{(1,2)}(\Sigma_\tau)}
    + \norm{v^{\Sigma,0}}_{\WW_{q'}^{2-2/q'}(\Sigma)^\ast} \norm{\Psi^\Sigma}_{\WW_{q'}^{2-2/q'}(\Sigma)}
    \\
   &\leq \big(\norm{f^\Sigma}_{\WW_{q'}^{(1,2)}(\Sigma_\tau)^\ast} + \norm{v^{\Sigma,0}}_{\WW_{q'}^{2-2/q'}(\Sigma)^\ast} \big) \norm{\Psi^\Sigma}_{\WW_{q'}^{2-2/q'}(\Sigma)^\ast},
  \end{align*}
 and the assertion follows by taking the supremum over all $\Theta^\Sigma \in \LL_{q'}(\Sigma_\tau)$ with $\norm{\Theta^\Sigma}_{\LL_{q'}(\Sigma_\tau)} \leq 1$.
\end{proof}

We will also need a version of the duality estimate related to the positive part.
Such positive versions are well-established for the pure bulk phase case, where we once again refer to M.~Pierre's survey \cite{Pierre_2010}.

 \begin{lemma}[Positive duality estimate]
 \label{lem:pos-dual-estimate-bulk-variant}
  Let $T > 0$, $q \in (1, \infty]$ and $\Omega \subseteq \R^d$ be a bounded domain with $\CC^{3-}$-boundary $\Sigma = \partial \Omega$.
  Let $\vec v \in \fs U_{\tilde p}^\Omega(T)$ for some $\tilde p > d+2$.
  Given $f, f_i \in \LL_q(\Omega_\tau)$, $g^\Sigma, g_i^\Sigma \in \WW_q^{(1,2) \cdot (\frac{1}{2} - \frac{1}{2p})}(\Sigma_\tau)$ and $v^0, v^0_i \in \WW_q^{2-2/q}(\Omega) $ (with compatibility condition $\partial_{\vec \nu} v^0 = g^\Sigma|_{t=0}$ and $\partial_{\vec \nu} v^0_i = g_i^\Sigma|_{t=0}$ if $q > 3$), let $v, v_i \in \WW_q^{(1,2)}(\Omega_\tau)$ be the (unique) strong $\WW_q^{(1,2)}$-solutions to the parabolic initial-boundary value problems
   \begin{alignat*}{2}
    \partial_t v + \vec v \cdot \nabla v - \dv( d \nabla v)
     &= f
	 \qquad &
     &\text{in } (0, \tau) \times \Omega,
     \\
    d \partial_{\vec \nu} v
     &= g^\Sigma
	 \qquad &
     &\text{on } (0, \tau) \times \Sigma,
     \\
    v(0,\cdot)
     &= v^0
     \qquad &
     &\text{in } \Omega
     \\
     \intertext{and}
    \alpha_i \partial_t v_i + \beta_i \vec v \cdot \nabla v_i - \dv( d_i \nabla v_i)
     &= f_i
	 \qquad &
     &\text{in } (0, \tau) \times \Omega,
     \\
    d_i \partial_{\vec \nu} v_i
     &= g_i^\Sigma
	 \qquad &
     &\text{on } (0, \tau) \times \Sigma,
     \\
    v_i(0,\cdot)
     &= v^0_i
	 \qquad &
     &\text{in } \Omega,
   \end{alignat*}
  where $\tau \in (0, T)$, and positive constants $d, \alpha_i, d_i > 0$ and $\beta_i \in \R$ are given.
  \newline
  Then there is a constant $C = C(q,\Omega,T) > 0$ such that
   \begin{align*}
    &\norm{v^+}_{\LL_q(\Omega_\tau)} + \norm{v^+|_\Sigma}_{\mathring \WW_{q'}^{(1,2) \cdot (\frac{1}{2} - \frac{1}{2q'})}(\Sigma_\tau)^\ast}
     \\
     &\leq C \big(\sum_{j=1}^N \norm{v_j}_{\LL_q(\Omega_\tau)} + \norm{v_j^+|_\Sigma}_{\mathring \WW_{q'}^{(1,2) \cdot (\frac{1}{2} - \frac{1}{2q'})}(\Sigma_\tau)^\ast} + \norm{(v^0 - \sum_{j=1}^N \alpha_j v^0_j)^+}_{\WW_{q'}^{2-2/q'}(\Omega)^\ast}
      \\ &\qquad
      + \norm{(f - \sum_{j=1}^N f_j)^+}_{\LL_q(\Omega_\tau)} + \norm{(g^\Sigma - \sum_{j=1}^N g_j^\Sigma)^+}_{\WW_{q'}^{(1,2) \cdot (1 - \frac{1}{2q'})}(\Sigma_\tau)^\ast} \big).
   \end{align*}
 \end{lemma}
 
\begin{proof}
 We proceed similar to the proof of the dual estimates, however, this time restrict ourselves to functions $\Theta \in \LL_{q'}(\Omega_\tau)$ and $\eta \in \mathring \WW_{q'}^{(1,2) \cdot (\frac{1}{2} - \frac{1}{2q'})}(\Sigma_\tau)$ which are \emph{non-negative}.
 As in the proof of Lemma~\ref{lem:dual-estimate-variant}, we consider the dual problem
  \[
   \begin{cases}
    - \partial_t \Psi - \vec v \cdot \nabla \Psi - \dv(d \nabla \Psi)
     = \Theta \quad
     &\text{in } \Omega_\tau,
     \\
    d \partial_{\vec \nu} \Psi
     = \eta \quad
     &\text{on } \Sigma_\tau,
     \\
    \Psi(\tau,\cdot)
     = 0 \quad
     &\text{in } \Omega.
   \end{cases}
  \]
 By $\LL_{q'}$-maximal regularity of this problem and the comparison principle for parabolic equations with inhomogeneous Neumann data, we obtain a unique and non-negative solution $\Psi \in \WW_{q'}^{(1,2)}(\Omega_\tau)$ such that
  \[
   \norm{\Psi}_{\WW_{q'}^{(1,2)}(\Omega_\tau)}
    \leq C \big( \norm{\Theta}_{\LL_{q'}(\Omega_\tau)} + \norm{\eta}_{\mathring \WW_{q'}^{(1,2) \cdot (\frac{1}{2} - \frac{1}{2q'})}(\Sigma_\tau)} \big)
  \]
 for a constant $C > 0$ which is independent of the time-horizon $\tau \in (0,T]$ (thanks to the homogeneous final condition $\Psi(\tau,\cdot) = 0$) and the data $(\Theta, \eta)$.
 Then, using the PDE solved by $\theta$ and $v$, and via integration by parts, we find that
  \begin{align*}
   &\int_{\Omega_\tau} v \Theta \dd (t, \vec x)
    + \int_{\Sigma_\tau} v \eta \dd (t, \sigma(\vec x))
    \\
    &= - \int_{\Omega_\tau} v (\partial_t \Psi + \vec v \cdot \nabla \Psi + \dv(d \nabla \Psi)) \dd (t, \vec x)
     - \int_{\Sigma_\tau} v \partial_{\vec \nu} \Psi \dd (t, \sigma(\vec x))
     \\
    &= \int_{\Omega_\tau} (\partial_t v + \vec v \cdot \nabla v - \dv(d \nabla v)) \Psi \dd (t, \vec x)
     - \int_{\Sigma_\tau} \Psi d \partial_{\vec \nu} v \dd (t, \sigma(\vec x))
     + \int_\Omega v^0 \Psi(0,\cdot) \dd \vec x
     \\
    &= \int_{\Omega_\tau} f \Psi \dd (t, \vec x)
     + \int_{\Sigma_\tau} g^\Sigma \Psi \dd (t, \sigma(\vec x))
     + \int_\Omega v^0 \Psi(0,\cdot) \dd \vec x.
  \end{align*}
 The latter terms can be written as
  \begin{align*}
   &\int_{\Omega_\tau} f \Psi \dd (t, \vec x)
    + \int_{\Sigma_\tau} g^\Sigma \Psi \dd (t, \sigma(\vec x))
    + \int_\Omega v^0 \Psi(0,\cdot) \dd \vec x
    \\
   &= \int_{\Omega_\tau} (f - \sum_{j=1}^m f_j) \Psi \dd (t, \vec x)
    + \int_{\Sigma_\tau} (g^\Sigma - \sum_{j=1}^m g_j^\Sigma) \Psi \dd (t, \sigma(\vec x))
    + \int_\Omega (v^0 - \sum_{j=1}^m \alpha_j v_j^0) \Psi(0,\cdot) \dd \vec x
    \\
    &\quad
    + \sum_{j=1}^m \big(
     \int_{\Omega_\tau} f_j \Psi \dd (t, \vec x)
     + \int_{\Sigma_\tau} g_j^\Sigma \Psi \dd (t, \sigma(\vec x))
     + \int_\Omega \alpha_j v_j^0 \Psi(0,\cdot) \dd \vec x
     \big).
  \end{align*}
 Using the PDEs for the functions $v_i$, we further obtain
  \begin{align*}
   &\int_{\Omega_\tau} f_j \Psi \dd (t, \vec x)
    + \int_{\Sigma_\tau} g_j^\Sigma \Psi \dd (t, \sigma(\vec x))
    + \int_\Omega \alpha_j v_j^0 \Psi(0,\cdot) \dd \vec x
    \\
   &= \int_{\Omega_\tau} (\alpha_j \partial_t v_j + \beta_j \vec v \cdot \nabla v_j - \dv( d_j \nabla v_j) \Psi \dd (t, \vec x)
    \\ &\quad
    - \int_{\Sigma_\tau} \Psi d_j \partial_{\vec \nu} v_j \dd (t, \sigma(\vec x))
    + \int_\Omega \alpha_j v_j^0 \Psi(0,\cdot) \dd \vec x
    \\
    &= - \int_{\Omega_\tau} v_j \big( \alpha_j \partial_t \Psi + \beta_j \vec v \cdot \nabla \Psi + \dv(d_j \nabla \Psi) \big) \dd (t, \vec x)
     + \frac{d_j}{d} \int_{\Sigma_\tau} \eta v_j \dd (t, \sigma(\vec x)).
  \end{align*}
 Together, this gives
  \begin{align*}
   &\int_{\Omega_\tau} v \Theta \dd (t, \vec x)
    + \int_{\Sigma_\tau} v \eta \dd (t, \sigma(\vec x))
    \\
   &= \int_{\Omega_\tau} (f - \sum_{j=1}^m f_j) \Psi \dd (t, \vec x)
    + \int_{\Sigma_\tau} (g^\Sigma - \sum_{j=1}^m g_j^\Sigma) \Psi \dd (t, \sigma(\vec x))
    + \int_\Omega (v^0 - \sum_{j=1}^m \alpha_j v_j^0) \Psi(0,\cdot) \dd \vec x
    \\ &\quad
    + \sum_{j=1}^m \big( \int_{\Omega_\tau} v_j \big( \alpha_j \partial_t \Psi + \beta_j \vec v \cdot \nabla \Psi + \dv(d_j \nabla \Psi) \big) \dd (t, \vec x)
    + \frac{d_j}{d} \int_{\Sigma_\tau} \eta v_j \dd (t, \sigma(\vec x)) \big)
    \\
   &\leq \norm{(f - \sum_{j=1}^m f_j)^+}_{\WW_{q'}^{(1,2)}(\Omega_\tau)^\ast} \norm{\Psi}_{\WW_{q'}^{(1,2)}(\Omega_\tau)}
    \\ &\qquad
    + \norm{(g^\Sigma - \sum_{j=1}^m g_j^\Sigma)^+}_{\mathring \WW_{q'}^{(1,2) \cdot (1 - \frac{1}{2q'})}(\Sigma_\tau)^\ast} \norm{\Psi|_\Sigma}_{\mathring \WW_{q'}^{(1,2) \cdot (1 - \frac{1}{2q'})}(\Sigma_\tau)}
     \\ &\qquad
    + \norm{(v^0 - \sum_{j=1}^m \alpha_j v_j^0)^+}_{\WW_{q'}^{2-2/q'}(\Omega)^\ast} \norm{\Psi(0,\cdot)}_{\WW_{q'}^{2-2/q'}(\Omega)}
     \\ &\qquad
    + \sum_{j=1}^m \norm{v_j}_{\LL_q(\Omega_\tau)} \norm{\alpha_j \partial_t \Psi + \beta_j \vec v \cdot \nabla \Psi + \dv(d_j \nabla \Psi)}_{\LL_{q'}(\Omega_\tau)}
    \\ &\qquad
    + \frac{d_j}{d} \norm{v_j^+}_{\mathring \WW_{q'}^{(1,2) \cdot (\frac{1}{2} - \frac{1}{2q'})}(\Sigma_\tau)^\ast} \norm{\eta}_{\WW_{q'}^{(1,2) \cdot (\frac{1}{2} - \frac{1}{2q'})}(\Sigma_\tau)},
  \end{align*}
 since we know that $\Psi \geq 0$ and $\eta \geq 0$.
 Here, by $\LL_{q'}$-maximal regularity, we find that
  \[
   \norm{\alpha_j \partial_t \Psi + \beta_j \vec v \cdot \nabla \Psi + \dv(d_j \nabla \Psi)}_{\LL_{q'}(\Omega_\tau)}
    \leq C \big( \norm{\Theta}_{\LL_{q'}(\Omega_\tau)} + \norm{\eta}_{\mathring \WW_{q'}^{(1,2) \cdot (\frac{1}{2} - \frac{1}{2q'})}(\Sigma_\tau)} \big),
  \]
 so that taking the supremum over all non-negative $\Theta$ and $\eta$ with norm less or equal $1$, we may conclude that
  \begin{align*}
   &\norm{v^+}_{\LL_q(\Omega_\tau)} + \norm{v^+|_\Sigma}_{\mathring \WW_{q'}^{(1,2) \cdot (\frac{1}{2} - \frac{1}{2q'})}(\Sigma_\tau)^\ast}
    \\
    &\leq C \big( \norm{(f - \sum_{j=1}^m f_j)^+}_{\WW_{q'}^{(1,2)}(\Omega_\tau)^\ast} + \norm{(g^\Sigma - \sum_{j=1}^m g_j^\Sigma)^+}_{\mathring \WW_{q'}^{(1,2) \cdot (1 - \frac{1}{2q'})}(\Sigma_\tau)^\ast}
     \\ &\quad
     + \norm{(v^0 - \sum_{j=1}^m \alpha_j v_j^0)^+}_{\WW_{q'}^{2-2/q'}(\Omega)^\ast} + \sum_{j=1}^m \norm{v_j^+|_\Sigma}_{\mathring \WW_{q'}^{(1,2) \cdot (1 - \frac{1}{2q'})}(\Sigma_\tau)^\ast} + \norm{v_j}_{\LL_q(\Omega_\tau)} \big).
  \end{align*}
\end{proof}
 
Similarly, a positive version of the dual estimate for parabolic equations on the surface $\Sigma = \partial \Omega$ can be derived:

\begin{lemma}[Positive dual estimate on the surface]
 \label{lem:pos_dual_estimate_surface}
 Let $\Sigma \subseteq \R^d$ be the boundary of a bounded $\CC^{3-}$-domain, $\mu, \beta_i \in \CC^1(\Sigma;(0,\infty))$ and $\alpha_i \in \R$, $i = 1, \ldots, m$.
 Moreover, let $p, q \in (1,\infty)$ and $T > 0$ be given.
 Assume that $v^\Sigma$, $v_i^\Sigma \in \WW_p^{(1,2)}(\Sigma_\tau)$ are solutions to the parabolic initial-value problems
  \[
   \begin{cases}
    \partial_t v^\Sigma - \dv_\Sigma (\mu \nabla_\Sigma v^\Sigma)
     = f^\Sigma
     \quad &
     \text{on } \Sigma_\tau,
     \\
    v^\Sigma(0,\cdot)
     = v^{\Sigma,0}
     \quad &
     \text{on } \Sigma
   \end{cases}
  \]
 and
  \[
   \begin{cases}
    \alpha_i \partial_t v_i^\Sigma - \dv_\Sigma(\beta_i \nabla_\Sigma v_i^\Sigma)
     = f_i^\Sigma
     \quad &
     \text{on } \Sigma_\tau,
     \\
    v_i^\Sigma(0,\cdot)
     = v_i^{\Sigma,0}
     \quad &
     \text{on } \Sigma
   \end{cases}
  \]
 for some $\tau \in (0,T]$ and $f^\Sigma, f_i^\Sigma \in \LL_p(\Sigma_\tau)$, $v^{\Sigma,0}$, $v_i^{\Sigma,0} \in \WW_p^{2-2/p}(\Sigma)$.
 \newline
 Then there is a constant $C = C(q,T)$, independent of $\tau \in (0,T]$ and the data, such that
  \begin{align*}
   &\norm{v^{\Sigma,+}}_{\LL_q(\Sigma_\tau)}
    \\
    &\leq C \big( \norm{(f^\Sigma - \sum_{j=1}^m f_j^\Sigma)^+}_{\WW_{q'}^{(1,2)}(\Sigma_\tau)^\ast} + \norm{(v^{\Sigma,0} - \sum_{j=1}^m \alpha_j v_j^{\Sigma,0})^+}_{\WW_{q'}^{2-2/q'}(\Sigma)^\ast} + \sum_{j=1}^m \norm{v_j^\Sigma}_{\LL_q(\Sigma_\tau)} \big).
  \end{align*}
\end{lemma}

\begin{proof}
 Similar to the bulk case, we consider an arbitrary non-negative function $\Theta^\Sigma \in \LL_{q'}(\Sigma_\tau)$ and the time-reversed dual problem
  \[
   \begin{cases}
    - \partial_t \Psi^\Sigma - \dv( \mu \nabla_\Sigma \Psi^\Sigma)
     = \Theta^\Sigma \quad
     &\text{on } \Sigma_\tau,
     \\
    \Psi^\Sigma(\tau,\cdot)
     = 0 \quad
     &\text{on } \Sigma
   \end{cases}
  \]
 which has the property of $\LL_{q'}$-maximal regularity with a maximal regularity constant which, thanks to the zero final conditions, can be chosen independent of the time horizon $\tau \in (0,T]$, i.e.\
  \[
   \norm{\Psi^\Sigma}_{\WW_{q'}^{(1,2)}(\Sigma_\tau)}
    \leq C \norm{\Theta^\Sigma}_{\LL_{q'}(\Sigma_\tau)}.
  \]
 Moreover, as $\Theta^\Sigma \geq 0$, by the parabolic positivity principle, $\Psi^\Sigma \geq 0$, and, therefore, we may estimate
  \begin{align*}
   &\int_{\Sigma_\tau} v^\Sigma \Theta^\Sigma \dd (t, \sigma(\vec x))
    \\
    &= - \int_{\Sigma_\tau} (\partial_t \Psi^\Sigma + \dv_\Sigma( \mu \nabla_\Sigma \Psi^\Sigma)) v^\Sigma \dd (t, \sigma(\vec x))
    \\
    &= \int_{\Sigma_\tau} (\partial_t v^\Sigma - \dv_\Sigma (\mu \nabla_\Sigma v^\Sigma)) \Psi^\Sigma \dd (t, \sigma(\vec x))
     + \int_\Sigma v^{\Sigma,0} \Psi^{\Sigma}(0,\cdot) \dd \sigma(\vec x)
     \\
    &= \int_{\Sigma_\tau} f^\Sigma \Psi^\Sigma \dd (t, \sigma(\vec x))
     + \int_{\Sigma} v^{\Sigma,0} \Psi^{\Sigma}(0,\cdot) \dd \sigma(\vec x)
     \\
    &= \int_{\Sigma_\tau} (f^\Sigma - \sum_{j=1}^m f_j^\Sigma) \Psi^\Sigma \dd (t, \sigma(\vec x))
     + \int_{\Sigma} (v^{\Sigma,0} - \sum_{j=1}^m \alpha_j v_j^{\Sigma,0}) \Psi^\Sigma(0,\cdot) \dd \sigma(\vec x)
     \\ &\quad
     + \sum_{j=1}^m \big( \int_{\Sigma_\tau} f_j^\Sigma \Psi^\Sigma \dd (t,\sigma(\vec x)) + \int_{\Sigma} \alpha_j v_j^{\Sigma,0} \Psi^\Sigma(0,\cdot) \dd \sigma(\vec x) \big).
  \end{align*}
 Here, we may estimate the terms in the last line by
  \begin{align*}
   &\int_{\Sigma_\tau} f_j^\Sigma \Psi^\Sigma \dd (t,\sigma(\vec x)) + \int_{\Sigma} \alpha_j v_j^{\Sigma,0} \Psi^\Sigma(0,\cdot) \dd \sigma(\vec x)
    \\
    &= \int_{\Sigma_\tau} \big( \alpha_j \partial_t v_j^\Sigma \cdot \Psi^\Sigma
     - \dv_\Sigma (\beta_j \nabla_\Sigma v_j^\Sigma) \Psi^\Sigma \big) \dd (t, \sigma(\vec x))
     + \int_\Sigma \alpha_j v_j^\Sigma \Psi^\Sigma(0,\cdot) \dd \sigma(\vec x)
    \\
    &= - \int_{\Sigma_\tau} \big( \alpha_j v_j^\Sigma \partial_t \Psi^\Sigma + v_j^\Sigma \dv_\Sigma (\beta_j \nabla_\Sigma \Psi^\Sigma) \big) \dd (t,\sigma(\vec x))
    \\
    &\leq C \norm{v_j^\Sigma}_{\LL_q(\Omega_\tau)} \norm{\Theta^\Sigma}_{\LL_{q'}(\Omega_\tau)}.
  \end{align*}
 Taking the supremum over all non-negative $\Theta^\Sigma \in \LL_{q'}(\Omega_\tau)$ with norm less or equal one, we deduce that
  \begin{align*}
   &\norm{v^{\Sigma,+}}_{\LL_q(\Sigma_\tau)}
    \\
    &\leq C \big( \norm{(f^\Sigma - \sum_{j=1}^m f_j^\Sigma)^+}_{\WW_{q'}^{(1,2)}(\Sigma_\tau)^\ast} + \norm{(v^{\Sigma,0} - \sum_{j=1}^m \alpha_j v_j^{\Sigma,0})^+}_{\WW_{q'}^{2-2/q'}(\Sigma)^\ast} + \sum_{j=1}^m \norm{v_j^\Sigma}_{\LL_q(\Sigma_\tau)} \big).
  \end{align*}
\end{proof}

 \section{Global-in-time existence}
 \label{sec:global-existence}

 By comparing with blow-up results for ODEs modelling reaction networks, it is clear that for homogeneous Neumann boundary conditions global-in-time well-posedness of reaction-diffusion systems cannot be expected, in general, and blow-up may occur (but see \cite{Souplet_2018} for a quite general result under mass dissipation and an entropy type dissipation condition).
 Therefore, we cannot expect global-in-time existence of solutions without any further restriction on the structure of the chemical reaction networks in the bulk phase and on the surface.
 Also note the blow-up result in \cite{PieSch97} for \emph{inhomogeneous} Dirichlet boundary conditions.
 \begin{example}[Blow-up in the Henry model]
  Consider $N = 1$ and the Henry model $s^\Sigma(c|_\Sigma, c^\Sigma) = c - c^\Sigma$, chemical reaction models $r^\Omega(c) = r^{\Sigma}(c) = c^2$, and initial data $c^0 \equiv 1$, $c^{\Sigma,0} \equiv 1$.
  (Here, for simplicity, we set all physical parameters to be $1$.)  
  Then the unique solution $(c, c^\Sigma)(t,\cdot) \equiv (\frac{1}{1-t}, \frac{1}{1-t})$, $t \in [0,1)$, blows up in finite time.
 \end{example} 
 Although this example is of purely academic nature, it hints that also for systems with coupling between bulk and surface reaction-diffusion-advection systems, further structural assumptions are needed to ensure global-in-time existence of solutions.
 Throughout this section we, therefore, additionally demand the following of the chemical reaction models $\vec r^\Omega$ in the bulk phase and $\vec r^\Sigma$ on the active surface, respectively.
 Note that for the case of Maxwell--Stefan diffusion models, where the diffusive flux is modelled by a indirect balance equation of forces under a summation constraint, $\LL_\infty$-bounds are trivial, but do, in general, not suffice to establish global-in-time existence of solutions, cf.\ \cite{Bothe2015}.
  \begin{assumption}[Triangular structure of $\vec r^\Omega$ and $\vec r^\Sigma$ / Intermediate sum condition]
  \label{assmpt:triangular_structure}
   There are lower triangular matrices $\bb Q^\Omega, \bb Q^{\Sigma} \in \R_+^{N \times N}$ with strictly positive diagonal entries and $C_\mathrm{tr}$, $C_\mathrm{tr}^\Sigma > 0$ such that
    \begin{alignat}{2}
     \bb Q^\Omega \vec r^\Omega(\vec y)
      &\leq C_\mathrm{tr} \left(  1 + \sum_{j=1}^N y_j \right)^{\mu^\Omega} \vec e
	  \qquad &
      &\text{for all }
      \vec y \in [0, \infty)^N,
      \tag{$A_S^\mathrm{bulk}$}
      \label{A_S^bulk}
      \\
     \bb Q^\Sigma \vec r^\Sigma(\vec y^\Sigma)
      &\leq C_\mathrm{tr}^\Sigma \left(  1 + \sum_{j=1}^N y^\Sigma_j \right)^{\mu^\Sigma} \vec e
	  \qquad &
      &\text{for all }
      \vec y^\Sigma \in [0, \infty)^N,
      \tag{$A_S^{\mathrm{ch}}$}
      \label{A_S^ch}
    \end{alignat}
   for some exponents $\mu^\Omega, \mu^\Sigma \geq 0$, and where $\vec e = (1, \ldots, 1)^\mathsf{T} \in \R^N$  and we write $\vec v \leq \vec w$ if $v_i \leq w_i$ for all components $i = 1, \ldots, N$.
  \end{assumption}
  
To make a boot-strap argument feasible, and eventually establish $\LL_q$-growth bounds for the solution, we need $\LL_p$--$\LL_q$-estimates as in the following lemma, which heavily exploits the triangular structure of the chemical reaction networks.

 \begin{lemma}[$\LL_p$--$\LL_q$-estimates]
 \label{lem:L_p-L_q-estimates}
  Assume that, additionally to the standing Assumption~\ref{assmpt:general}, also the intermediate sum condition Assumption~\ref{assmpt:triangular_structure} is satisfied.
  Then, for
   \[
    p \in (\max \{ 1, \frac{(\mu^\Omega - 1)(d + 2)}{2 \mu^\Omega}, \frac{(\mu^\Sigma - 1)(d + 1)}{2 \mu^\Sigma} \}, \infty), \quad
    q \in (1, \infty]
    \quad \text{and} \quad
    T_0 > 0,
   \]
  we set $T^\ast := \min \{ T_0, T_\mathrm{max} \} \in (0, \infty)$, and there is a constant $C^\ast = C^\ast(\vec c^0, \vec c^{\Sigma,0}, p,q,T^\ast)$ such that for all $\tau \in (0,T^\ast]$
   \[
    \norm{\vec c}_{\LL_q(\Omega_\tau;\R^N)}
     + \norm{\vec c|_\Sigma}_{\mathring {\fs G}_{q'}^\Sigma(\tau)^\ast}
     + \norm{\vec c^\Sigma}_{\LL_q(\Sigma_\tau)}
     \leq C^\ast \left( 1
      + \norm{\vec c}_{\LL_p(\Omega_\tau;\R^N)}
      + \norm{\vec c^\Sigma}_{\LL_p(\Sigma_\tau;\R^N)} \right)^\mu,
   \]
  where $\mu = \mu(\mu^\Omega, \mu^\Sigma, p, q)$ will, in general, depend on $\mu^\Omega$, $\mu^\Sigma$ as well as on $p$ and $q$.
 \end{lemma}
 \begin{remark} 
  In case $\mu^\Omega = \mu^\Sigma = 1$, i.e.\ a \emph{linear} intermediate sum condition, all $p, q \in (1, \infty)$ are admissible for the $\LL_p$--$\LL_q$-estimates in this lemma.
 \end{remark}
 \begin{proof}[Proof of Lemma \ref{lem:L_p-L_q-estimates}]
  First, we consider the case where $q > p$ are such that $2 - \frac{d+2}{p} > - \frac{d+1}{q}$, $2 - \frac{d+2}{p} > - \frac{d+2}{\mu^\Omega q}$ and $2 - \frac{d+1}{p} > - \frac{d+1}{\mu^\Sigma q}$, so that the continuous embeddings $\WW_p^{(1,2)}(\Omega_T) \hookrightarrow \LL_q(\Sigma_T)$ (via the boundary trace map), $\WW_p^{(1,2)}(\Omega_\tau) \hookrightarrow \LL_{\mu^\Omega q}(\Omega_\tau)$ and $\WW_p^{(1,2)}(\Sigma_\tau) \hookrightarrow \LL_{\mu^\Sigma q}(\Sigma_\tau)$ are available.
 Here, the assumption that $p > \max \{ \frac{(\mu^\Omega - 1)(d + 2)}{2 \mu^\Omega}, \frac{(\mu^\Sigma - 1)(d + 1)}{2 \mu^\Sigma} \}$ enables us to find $q > p$ such that the latter two conditions are met.
  \newline
  We employ the triangular structure of the chemical reaction network in the bulk phase and on the active surface, by which there are constants $C_\mathrm{tr}, C_\mathrm{tr}^\Sigma > 0$ such that for $i = 1, \ldots, N$ and $r_{ij} := \frac{q_{ij}}{q_{ii}}, r_{ij}^\Sigma = \frac{q_{ij}^\Sigma}{q_{ii}^\Sigma} \in \R$, the inequalities
   \begin{align*}
    r^\Omega_i(\vec c)
     &\leq C_\mathrm{tr} (1 + \sum_{j=1}^N c_j)^{\mu^\Omega}
      - \sum_{j < i} r_{ij} r^\Omega_j(\vec c)
      \\
     &= C_\mathrm{tr} (1 + \sum_{j=1}^N c_j)^{\mu^\Omega}
      - \sum_{j < i} r_{ij} (\partial_t c_j + \vec v \cdot \nabla c_j - d_j \Delta c_j)
      \\
     r^\Sigma_i(\vec c^\Sigma)
      &\leq C_\mathrm{tr}^\Sigma (1 + \sum_{j=1} c_j^\Sigma)^{\mu^\Sigma}
       - \sum_{j < i} r_{ij}^\Sigma r^\Sigma_j(\vec c^\Sigma)
       \\
      &= C_\mathrm{tr}^\Sigma (1 + \sum_{j=1} c_j^\Sigma)^{\mu^\Sigma}
       - \sum_{j < i} r_{ij}^\Sigma (\partial_t c^\Sigma_j - d_j^\Sigma \Delta_\Sigma c^\Sigma_j - s_j(c_j|_\Sigma,c_j^\Sigma))
   \end{align*}
  are valid.
  \newline
  We use these estimates in combination with the positive version of the duality estimates stated in Lemma~\ref{lem:pos-dual-estimate-bulk-variant} and Lemma~\ref{lem:pos_dual_estimate_surface}:
  Fix any $T_0 > 0$, set $T^\ast := \min \{ T_0, T_\mathrm{max} \}$ and let $\tau \in (0, T^\ast)$, where $T_\mathrm{max} \in (0,\infty]$ is the maximal time of existence for the solution to the reaction-diffusion-advection-sorption system \eqref{eqn:RDASS} in the class $\WW_p^{(1,2)}$ for some $p \in (1,\infty)$ (and then, as it will turn out, away from $t = 0$ for \emph{every} $p \in (1, \infty)$, see Lemma~\ref{lem:smoothing_effect} below).
  For each component $i = 1, \ldots, N$ of the system, we consider the following auxiliary parabolic initial-boundary value problems:
   \begin{alignat*}{2}
    \partial_t z_i + \vec v \cdot \nabla z_i - d_i \Delta z_i
     &= r^\Omega_i(\vec c)
	 \qquad &
     &\text{in } (0, \tau) \times \Omega,
     \\
    - d_i \partial_{\vec \nu} z_i
     &= - k_i^\mathrm{de} (1 + c_i^\Sigma)
	 \qquad &
     &\text{on } (0, \tau) \times \Sigma,
     \label{a}
     \tag{a}
     \\
    z_i(0,\cdot)
     &= c^0_i
	 \qquad &
     &\text{on } \Omega
     \\
     \intertext{in the bulk phase and the parabolic initial value problems}
    \partial_t z_i^\Sigma - d_i \Delta_\Sigma z_i^\Sigma
     &= r^\Sigma_i(\vec c^\Sigma) + k_i^\mathrm{ad} c_i|_\Sigma
	 \qquad &
     &\text{on } (0, \tau) \times \Sigma,
     \label{b}
     \tag{b}
     \\
    z_i^\Sigma(0,\cdot)
     &= c^{\Sigma,0}_i
	 \qquad &
     &\text{on } \Sigma
   \end{alignat*}
  on the boundary.
  Choosing $r \in (1, \min \{p, 3 \})$ arbitrarily, the initial data do not necessarily have to satisfy the compatibility conditions, so that we obtain unique solutions $z_i \in \WW_r^{(1,2)}(\Omega_\tau)$ and $z_i^\Sigma \in \WW_r^{(1,2)}(\Sigma_\tau)$ of \eqref{a} and \eqref{b}, respectively.
  Moreover, since the lower and upper linear bounds $-k_i^\mathrm{de} (1 + c_i^\Sigma) \leq r^\mathrm{sorp}_i(c_i, c_i^\Sigma) \leq k_i^\mathrm{ad} (1 + c_i)$ are valid by assumption \eqref{A_B^sorp}, the comparison principle Lemma~\ref{lem:comparison_principle} provides the estimates
   \[
    0 \leq c_i \leq z_i,
     \quad
    0 \leq c_i^\Sigma \leq z_i^\Sigma
     \quad
     \text{for } i = 1, \ldots N.
   \]
  We exploit that the systems \eqref{a} and \eqref{b} are quasi-linear and write $z_i = u_i + v_i + w_i$ and $z_i^\Sigma = u_i^\Sigma + v_i^\Sigma + w_i^\Sigma$ as a sum of solutions to the following quasi-autonomous, linear systems:
   \begin{alignat*}{2}
    \partial_t u_i + \vec v \cdot \nabla u_i - d_i \Delta u_i
     &= C_\mathrm{tr} (1 + \sum_{j=1}^N c_j)^{\mu^\Omega}
	 \qquad &
     &\text{in } (0, \tau) \times \Omega,
     \\
    - d_i \partial_{\vec \nu} u_i
     &= 0
	 \qquad &
     &\text{on } (0, \tau) \times \Sigma,
     \tag{a1}
     \label{a1}
     \\
    u_i(0,\cdot)
     &= 0
	 \qquad &
     &\text{in } \Omega
     \\
     \intertext{and}
    \partial_t v_i + \vec v \cdot \nabla v_i - d_i \Delta v_i
     &= r_i^\Omega(\vec c) - C_\mathrm{tr} (1 + \sum_{j=1}^N c_j)^{\mu^\Omega}
	 \qquad &
     &\text{in } (0, \tau) \times \Omega,
     \\
    - d_i \partial_{\vec \nu} v_i
     &= - k_i^\mathrm{de} (1 + c_i^\Sigma)
	 \qquad &
     &\text{on } (0, \tau) \times \Sigma,
     \tag{a2}
     \label{a2}
     \\
    v_i(0,\cdot)
     &= 0
	 \qquad &
     &\text{in } \Omega
     \\
     \intertext{and}
    \partial_t w_i + \vec v \cdot \nabla w_i - d_i \Delta w_i
     &= 0
	 \qquad &
     &\text{in } (0, \tau) \times \Omega,
     \\
    - d_i \partial_{\vec \nu} w_i
     &= 0
	 \qquad &
     &\text{on } (0, \tau) \times \Sigma,
     \tag{a3}
     \label{a3}
     \\
    w_i(0,\cdot)
     &= c^0_i
	 \qquad &
     &\text{in } \Omega
   \end{alignat*}
  in the bulk phase, whereas on the surface, we consider the problems
   \begin{alignat*}{2}
    \partial_t u_i^\Sigma - d_i^\Sigma \Delta_\Sigma u_i^\Sigma
     &= C_\mathrm{tr}^\Sigma (1 + \sum_{j=1}^N c_j^\Sigma)^{\mu^\Sigma} + k_i^\mathrm{ad} (1 + c_i|_\Sigma)
	 \qquad &
     &\text{on } (0, \tau) \times \Sigma,
     \tag{b1}
     \label{b1}
     \\
    u_i^\Sigma(0,\cdot)
     &= 0
	 \qquad &
     &\text{on } (0, \tau) \times \Sigma,
     \\
     \intertext{and}
    \partial_t v_i^\Sigma - d_i^\Sigma \Delta_\Sigma v_i^\Sigma
     &= r^\Sigma_i(\vec c^\Sigma) - C_\mathrm{tr}^\Sigma (1 + \sum_{j=1}^N c_j^\Sigma)^{\mu^\Sigma}
	 \qquad &
     &\text{on } (0, \tau) \times \Sigma,
     \tag{b2}
     \label{b2}
     \\
    v_i^\Sigma(0,\cdot)
     &= 0
	 \qquad &
     &\text{on } (0, \tau) \times \Sigma,
     \\
     \intertext{and}
    \partial_t w_i^\Sigma - d_i^\Sigma \Delta_\Sigma w_i^\Sigma
     &= 0
	 \qquad &
     &\text{on } (0, \tau) \times \Sigma,
     \tag{b3}
     \label{b3}
     \\
    w_i^\Sigma(0,\cdot)
     &= c^{\Sigma,0}_i
	 \qquad &
     &\text{on } (0, \tau) \times \Sigma.
   \end{alignat*}
  Since the strong solutions $c_j, c_j^\Sigma \geq 0$ are non-negative, by the comparison principle, we immediately find that the solutions $u_i, w_i, u_i^\Sigma, w_i^\Sigma$ to the initial (boundary) value problems \eqref{a1}, \eqref{a3}, \eqref{b1} and \eqref{b3} are non-negative, whereas for the solutions $v_i$ and $v_i^\Sigma$ to \eqref{a2} and \eqref{b2} non-negativity cannot be guaranteed. However, we always have the inequalities
   \[
    0
     \leq c_i \leq z_i \leq u_i + v_i^+ + w_i,
     \quad
    0
     \leq c_i^\Sigma \leq z_i^\Sigma \leq u_i^\Sigma + v_i^{\Sigma,+} + w_i
     \quad
     \text{for } i = 1, \ldots, N.
   \]
  Let us collect the information which we can extract from the duality estimates and positive duality estimates:
   \[
    \norm{u_i}_{\LL_q(\Omega_\tau)} + \norm{u_i|_\Sigma}_{\mathring \WW_{q'}^{(1,2) \cdot (\frac{1}{2} - \frac{1}{2q'})}(\Sigma_\tau)^\ast}
     \leq C_{p,q,T^\ast} (1 + \sum_{j=1}^N \norm{c_j^{\mu^\Omega}}_{\mathring \WW_{q'}^{(1,2)}(\Omega_\tau)^\ast}),
   \]
  where the constant $C_{p,q,T^\ast} > 0$ may be chosen independent of $\tau < T^\ast$, since we consider zero initial data in \eqref{a1};
   \begin{align*}
    &\norm{v_i^+}_{\LL_q(\Omega_\tau)} + \norm{v_i^+|_\Sigma}_{\mathring \WW_{q'}^{(1,2) \cdot (\frac{1}{2} - \frac{1}{2q'})}(\Sigma_\tau)^\ast}
     \\
     &\qquad
     \leq C_{p,q,T} (1 + \norm{c_i^\Sigma}_{\mathring \WW_{q'}^{(1,2) \cdot (\frac{1}{2} - \frac{1}{2q'})}(\Sigma_\tau)^\ast} + \sum_{j=1}^N \norm{c_j^{\mu^\Omega}}_{\mathring \WW_{q'}^{(1,2)}(\Omega_\tau)^\ast} ),
     \\
    &\norm{w_i}_{\LL_q(\Omega_\tau)} + \norm{w_i|_\Sigma}_{\mathring \WW_{q'}^{(1,2) \cdot (\frac{1}{2} - \frac{1}{2q'})}(\Sigma_\tau)^\ast}
     \leq C_{q,T^\ast}.
    \end{align*}
   Here, in the latter two estimates the constants $C_{p,q,T^\ast}$ and $C_{q,T^\ast}$ may (and will) depend on the initial data, but not on the particular choice of $\tau \in (0, T^\ast)$.
   Similarly, on the surface we obtain the estimates
    \begin{align*}
     \norm{u_i^\Sigma}_{\LL_q(\Sigma_\tau)}
      &\leq C_{p,q,T} (1 + \norm{c_i|_\Sigma}_{\mathring \WW_{q'}^{(1,2)}(\Sigma_\tau)^\ast} + \sum_{j=1}^N \norm{(c_j^\Sigma)^{\mu^\Sigma}}_{\mathring \WW_{q'}^{(1,2)}(\Sigma_\tau)^\ast})
     \\
     &\leq C_{p,q,T} (1 + \norm{c_i|_\Sigma}_{\mathring \WW_{q'}^{(1,2) \cdot (\frac{1}{2} - \frac{1}{2q'})}(\Sigma_\tau)^\ast} + \sum_{j=1}^N \norm{(c_j^\Sigma)^{\mu^\Sigma}}_{\mathring \WW_{q'}^{(1,2)}(\Sigma_\tau)^\ast} )
     \\
     \intertext{again exploiting the zero initial data;}
    \norm{v_i^{\Sigma,+}}_{\LL_q(\Sigma_\tau)}
     &\leq C_{p,q,T^\ast} \sum_{j<i} \norm{(c_j^\Sigma)^{\mu^\Sigma}}_{\LL_q(\Sigma_\tau)};
     \\
    \norm{w_i^\Sigma}_{\LL_q(\Sigma_\tau)}
     &\leq C_{p,q,T^\ast}.
   \end{align*}
  We now bring together these estimates, proceeding iteratively:
  \newline
  \textit{i=1:}
  For the surface term, we find
   \begin{align*}
    \norm{c_1^\Sigma}_{\LL_q(\Sigma_\tau)}
     &\leq \norm{z_1^\Sigma}_{\LL_q(\Sigma_\tau)}
     \leq \norm{u_1^\Sigma}_{\LL_q(\Sigma_\tau)} + \norm{v_1^{\Sigma,+}}_{\LL_q(\Sigma_\tau)} + \norm{w_1^\Sigma}_{\LL_q(\Sigma_\tau)}
     \\
     &\leq C_{p,q,T^\ast} \big( 1 + \sum_{j=1}^N \norm{(c_j^\Sigma)^{\mu^\Sigma}}_{\mathring \WW_{q'}^{(1,2)}(\Sigma_\tau)^\ast} + \norm{c_1|_\Sigma}_{\mathring \WW_{q'}^{(1,2) \cdot (\frac{1}{2} - \frac{1}{2q'})}(\Sigma_\tau)^\ast} \big)
   \end{align*}
  and, for the bulk term,
   \begin{align*}
    &\norm{c_1}_{\LL_q(\Omega_\tau)} + \norm{c_1|_\Sigma}_{\mathring \WW_{q'}^{(1,2) \cdot (\frac{1}{2} - \frac{1}{2q'})}(\Sigma_\tau)^\ast}
     \\
     &\leq \norm{z_1}_{\LL_q(\Omega_\tau)} + \norm{z_1|_\Sigma}_{\mathring \WW_{q'}^{(1,2) \cdot (\frac{1}{2} - \frac{1}{2q'})}(\Sigma_\tau)^\ast}
     \\
     &\leq \norm{u_1}_{\LL_q(\Omega_\tau)} + \norm{u_1|_\Sigma}_{\mathring \WW_{q'}^{(1,2) \cdot (\frac{1}{2} - \frac{1}{2q'})}(\Sigma_\tau)^\ast}
      + \norm{v_1^+}_{\LL_q(\Omega_\tau)}
      \\ &\quad
      + \norm{v_1^+|_\Sigma}_{\WW_{q'}^{(1,2) \cdot (\frac{1}{2} - \frac{1}{2q'})}(\Sigma_\tau)^\ast}
      + \norm{w_1}_{\LL_q(\Omega_\tau)} + \norm{w_1|_\Sigma}_{\mathring \WW_{q'}^{(1,2) \cdot (\frac{1}{2} - \frac{1}{2q'})}(\Sigma_\tau)^\ast}
     \\
     &\leq C_{p,q,T^\ast} \big( 1 + \sum_{j=1}^N \norm{(c_j)^{\mu^\Omega}}_{\mathring \WW_{q'}^{(1,2)}(\Omega_\tau)^\ast} + \norm{c_1^\Sigma}_{\mathring \WW_{q'}^{(1,2)}(\Sigma_\tau)^\ast} \big).
    \end{align*}
   Together, this gives the estimate
    \begin{align*}
     &\norm{c_1}_{\LL_q(\Omega_\tau)}
      + \norm{c_1|_\Sigma}_{\mathring \WW_{q'}^{(1,2) \cdot (\frac{1}{2} - \frac{1}{2q'})}(\Sigma_\tau)^\ast}
      + \norm{c_1^\Sigma}_{\LL_q(\Omega_\tau)}
       \\
       &\leq C_{q,T^\ast} \big( 1 + \sum_{j=1}^N \norm{c_j^{\mu^\Omega}}_{\mathring \WW_{q'}^{(1,2)}(\Omega_\tau)^\ast} + \norm{(c_j^\Sigma)^{\mu^\Sigma}}_{\mathring \WW_{q'}^{(1,2)}(\Sigma_\tau)\ast} \big).
    \end{align*}
  \textit{$i \geq 2$:}
  Now, assume that $i \geq 2$ and for $k = 1, \ldots, i-1$ we have shown the estimates
    \begin{align*}
     &\norm{c_k}_{\LL_q(\Omega_\tau)}
      + \norm{c_k|_\Sigma}_{\mathring \WW_{q'}^{(1,2) \cdot (\frac{1}{2} - \frac{1}{2q'})}(\Sigma_\tau)^\ast}
      + \norm{c_k^\Sigma}_{\LL_q(\Omega_\tau)}
       \\
       &\leq C_{q,T^\ast} \big( 1 + \sum_{j=1}^N \norm{c_j^{\mu^\Omega}}_{\mathring \WW_{q'}^{(1,2)}(\Omega_\tau)^\ast} + \norm{(c_j^\Sigma)^{\mu^\Sigma}}_{\mathring \WW_{q'}^{(1,2)}(\Sigma_\tau)^\ast} \big).
    \end{align*}
  We then find
   \begin{align*}
    \norm{c_i^\Sigma}_{\LL_q(\Sigma_\tau)}
     &\leq \norm{z_i^\Sigma}_{\LL_q(\Sigma_\tau)}
     \leq \norm{u_i^\Sigma}_{\LL_q(\Sigma_\tau)}
      + \norm{v_i^{\Sigma,+}}_{\LL_q(\Sigma_\tau)}
      + \norm{w_i^\Sigma}_{\LL_q(\Sigma_\tau)}
     \\
     &\leq C_{q,T^\ast} \big( 1 + \sum_{j=1}^N \norm{(c_j^\Sigma)^{\mu^\Sigma}}_{\mathring \WW_{q'}^{(1,2)}(\Sigma_\tau)^\ast}
      + \norm{c_i|_\Sigma}_{\mathring \WW_{q'}^{(1,2) \cdot (\frac{1}{2} - \frac{1}{2q'})}(\Sigma_\tau)^\ast}
      + \sum_{j < i} \norm{c_j^\Sigma}_{\LL_q(\Sigma_\tau)} \big).
   \end{align*}
  and
   \begin{align*}
    &\norm{c_i}_{\LL_q(\Omega_\tau)} + \norm{c_i|_\Sigma}_{\mathring \WW_{q'}^{(1,2) \cdot (\frac{1}{2} - \frac{1}{2q'})}(\Sigma_\tau)^\ast}
     \\
     &\leq \norm{z_i}_{\LL_q(\Omega_\tau)} + \norm{z_i|_\Sigma}_{\mathring \WW_{q'}^{(1,2) \cdot (\frac{1}{2} - \frac{1}{2q'})}(\Sigma_\tau)^\ast}
     \\
     &\leq \norm{u_i}_{\LL_q(\Omega_\tau)} + \norm{u_i|_\Sigma}_{\mathring \WW_{q'}^{(1,2) \cdot (\frac{1}{2} - \frac{1}{2q'})}(\Sigma_\tau)^\ast} + \norm{v_i^+}_{\LL_q(\Omega_\tau)}
      \\ &\quad
       + \norm{v_i^+|_\Sigma}_{\mathring \WW_{q'}^{(1,2) \cdot (\frac{1}{2} - \frac{1}{2q'})}(\Sigma_\tau)^\ast} + \norm{w_i}_{\LL_q(\Sigma_\tau)} + \norm{w_i|_\Sigma}_{\mathring \WW_{q'}^{(1,2) \cdot (\frac{1}{2} - \frac{1}{2q'})}(\Sigma_\tau)^\ast}
     \\
     &\leq C_{q,T^\ast} \big( 1 + \sum_{j=1}^N \norm{c_j^{\mu^\Omega}}_{\mathring \WW_{q'}^{(1,2)}(\Omega_\tau)^\ast}
      + \sum_{j < i} \big( \norm{c_j}_{\LL_q(\Omega_\tau)} + \norm{c_j|_\Sigma}_{\mathring \WW_{q'}^{(1,2) \cdot (\frac{1}{2} - \frac{1}{2q'})}(\Sigma_\tau)^\ast}
      + \norm{c_j^\Sigma}_{\LL_q(\Omega_\tau)} \big) \big).
   \end{align*}
  Together with the induction hypothesis, this gives
   \begin{align*}
     &\norm{c_i}_{\LL_q(\Omega_\tau)}
      + \norm{c_i|_\Sigma}_{\mathring \WW_{q'}^{(1,2) \cdot (\frac{1}{2} - \frac{1}{2q'})}(\Sigma_\tau)^\ast}
      + \norm{c_i^\Sigma}_{\LL_q(\Omega_\tau)}
       \\
       &\leq C_{q,T^\ast} \big( 1 + \sum_{j=1}^N \norm{c_j^{\mu^\Omega}}_{\mathring \WW_{q'}^{(1,2)}(\Omega_\tau)^\ast} + \norm{(c_j^\Sigma)^{\mu^\Sigma}}_{\mathring \WW_{q'}^{(1,2)}(\Sigma_\tau)^\ast} \big),
   \end{align*}
  so that by iteration we obtain the estimate
   \begin{align}
     &\sum_{j=1}^N \big( \norm{c_j}_{\LL_q(\Omega_\tau)}
      + \norm{c_j|_\Sigma}_{\mathring \WW_{q'}^{(1,2) \cdot (\frac{1}{2} - \frac{1}{2q'})}(\Sigma_\tau)^\ast}
      + \norm{c_j^\Sigma}_{\LL_q(\Omega_\tau)} \big)
       \nonumber \\ &\quad
       \leq C_{q,T^\ast} \big( 1 + \sum_{j=1}^N \big( \norm{c_j^{\mu^\Omega}}_{\mathring \WW_{q'}^{(1,2)}(\Omega_\tau)^\ast} + \norm{(c_j^\Sigma)^{\mu^\Sigma}}_{\mathring \WW_{q'}^{(1,2)}(\Sigma_\tau)^\ast} \big) \big).
      \label{eqn:general_estimate}
   \end{align}
  From the estimate \eqref{eqn:general_estimate}, we may now derive the $\LL_p$--$\LL_q$-estimates.
  First, let us consider the case that $p < q \leq \infty$ (hence, $1 \leq q' < p'$ for the corresponding Hölder conjugates) such that $\frac{1}{q} \geq \frac{1}{p} - \frac{d+2}{2}$, $2 - \frac{d+2}{p} \geq - \frac{d+2}{\mu^\Sigma q}$.
  Then, we may employ the continuous and dense embeddings
   \[
    \mathring \WW_{q'}^{(1,2)}(\Omega_\tau)
     \hookrightarrow \LL_{(\mu^\Omega p)'}(\Omega_\tau)
     \quad \text{and} \quad
    \mathring \WW_{q'}^{(1,2)}(\Sigma_\tau)
     \hookrightarrow \LL_{(\mu^\Sigma p)'}(\Sigma_\tau),
   \]
  and, hence, we may conclude that
   \[
    \LL_p(\Omega_\tau)
     \hookrightarrow \mathring \WW_{q'}^{(1,2)}(\Omega_\tau)^\ast
     \quad \text{and} \quad
    \LL_p(\Sigma_\tau)
     \hookrightarrow \mathring \WW_{q'}^{(1,2)}(\Sigma_\tau)^\ast
   \]
  are continuously embedded.
  Therefore, we obtain that
   \begin{align*}
     &\sum_{j=1}^N \big( \norm{c_j}_{\LL_q(\Omega_\tau)}
      + \norm{c_j^\Sigma}_{\LL_q(\Omega_\tau)}
      + \norm{c_j|_\Sigma}_{\mathring \WW_{q'}^{(1,2) \cdot (\frac{1}{2} - \frac{1}{2q'})}(\Sigma_\tau)} \big)
       \\
       &\leq C_{p,q,T^*} \big( 1 + \sum_{j=1}^N \big( \norm{c_j}_{\LL_p(\Omega_\tau)}^{\mu^\Omega} + \norm{c_j^\Sigma}_{\LL_p(\Sigma_\tau)}^{\mu^\Sigma} \big) \big)
     \quad \text{for } p < q \leq \infty \text{ s.t.\ } \frac{1}{p} \geq \frac{1}{q} - \frac{d+2}{2}.
   \end{align*}
  Iterating this procedure, if necessary, we obtain that
   \begin{align*}
     &\sum_{j=1}^N \big( \norm{c_j}_{\LL_q(\Omega_\tau)}
      + \norm{c_j^\Sigma}_{\LL_q(\Omega_\tau)}
      + \norm{c_j|_\Sigma}_{\mathring \WW_{q'}^{(1,2) \cdot (\frac{1}{2} - \frac{1}{2q'})}(\Sigma_\tau)} \big)
      \\
       &\leq C_{p,q,T^*} \big( 1 + \sum_{j=1}^N \big( \norm{c_j}_{\LL_p(\Omega_\tau)} + \norm{c_j^\Sigma}_{\LL_p(\Sigma_\tau)} \big) \big)^\mu
   \end{align*}
  for all $p > \max \{ 1, \frac{(\mu^\Omega - 1)(d + 2)}{2 \mu^\Omega}, \frac{(\mu^\Sigma - 1)(d + 1)}{2 \mu^\Sigma} \}, \, q \in (1,\infty)$,
  where $\mu \geq 0$ will, in general, be a multiple of $\mu^\Omega$ or $\mu^\Sigma$.
 \end{proof}
 Regarding the term $c_j|_\Sigma$, this is not yet the estimate in the norm which we are actually looking for.
 To improve on the norm, we consider the bulk reaction-diffusion-advection problems
  \begin{alignat*}{2}
   (\partial_t + \vec v \cdot \nabla - d_i \Delta) c_i
    &= f_i := r_i^\Omega(\vec c)
	\qquad &
    &\text{in } \Omega_\tau,
    \\
   c_i
    &= h_i := c_i|_\Sigma
	\qquad &
    &\text{on } \Sigma_\tau,
    \\
   c_i(0,\cdot)
    &= c_i^0
	\qquad &
    &\text{in } \Omega.
  \end{alignat*}
For the initial data, we may and will w.l.o.g.\ assume that $c_i^0 \in \CC^2(\overline{\Omega})$.
Moreover, by the estimates which we have already established and polynomial boundedness of the chemical reaction rates, we may conclude that $r_i^\Omega(\vec c) \in \bigcap\limits_{q \in (1,\infty)} \LL_q(\Omega_\tau)$ and $h_i = c_i|_\Sigma \in \bigcap\limits_{q \in (1,\infty)} \mathring \WW_{q'}^{(1,2) \cdot (\frac{1}{2} - \frac{1}{2q'})}(\Sigma_\tau)^\ast$.
\newline
To obtain an estimate of the form
  \[
  \norm{c_i|_\Sigma}_{\LL_q(\Sigma_\tau)}
   \leq C \big( \norm{\vec c}_{\LL_{q \gamma^\Omega}(\Omega_\tau)}^{\gamma^\Omega} + \norm{c_i|_\Sigma}_{\mathring \WW_{q'}^{(1,2) \cdot (\frac{1}{2} - \frac{1}{2q'})}(\Sigma_\tau)^\ast} \big)
   \quad
   \text{for sufficiently large } q,
 \]
we need to make some additional considerations:
As a first step, let us consider the known optimal regularity estimate
 \[
  \norm{v_i}_{\WW_q^{(1,2)}(\Omega_\tau)}
   \leq C \big( \norm{h_i}_{\WW_q^{(1,2) \cdot (1 - \frac{1}{2q})}(\Sigma_\tau)} + \norm{v_i^0}_{\WW_q^{2-2/q}(\Omega)} + \norm{f_i}_{\LL_q(\Omega_\tau)}  \big)
 \]
for the reaction-diffusion-advection systems with inhomogeneous Dirichlet data
 \begin{alignat*}{2}
  (\partial_t - \vec v \cdot \nabla - d_i \Delta) v_i
   &= f_i
   \qquad
   &&\text{in } \Omega_\tau,
   \\
  v_i|_\Sigma
   &= h_i
   \qquad
   &&\text{on } \Sigma_\tau,
   \\
  v_i(0,\cdot)
   &= v_i^0
   \qquad
   &&\text{in } \Omega.
 \end{alignat*}
Via its dual problem
 \begin{alignat*}{2}
  - (\partial_t + \vec v \cdot \nabla + d_i \Delta) \theta_i
   &= \Theta_i
   \qquad
   &&\text{in } \Omega_\tau,
   \\
  \theta_i
   &= g_i
   \qquad
   &&\text{on } \Sigma_\tau,
   \\
  \theta_i(\tau,\cdot)
   &= 0
   \qquad
   &&\text{in } \Omega
 \end{alignat*}
we obtain the estimate
 \[
  \norm{v_i}_{\LL_q(\Omega_\tau)}
   \leq C \big( \norm{h_i}_{\mathring \WW_{q'}^{(1,2) \cdot (1 - \frac{1}{2q'})}(\Sigma_\tau)^\ast} + \norm{v_i^0}_{\WW_{q'}^{2-2/q'}(\Omega)^\ast} + \norm{f_i}_{\WW_{q'}^{(1,2)}(\Omega_\tau)^\ast}  \big).
 \]
 We now want to interpolate between these estimates.
 If we would proceed directly with the real interpolation method, we can estimate the $( \LL_q(\Omega_\tau), \WW_q^{(1,2)}(\Omega_\tau) )_{\theta,q}$-norm (which is equivalent to $\norm{\cdot}_{\BB_q^{(1,2) \cdot \theta}(\Omega_\tau)}$) by the interpolation space norms
  \begin{align*}
   &\norm{h_i}_{( \WW_{q'}^{(1,2) \cdot (1 - \frac{1}{2q'})}(\Sigma_\tau)^\ast, \WW_q^{(1,2) \cdot (1 - \frac{1}{2q})})_{\theta,q}},
   \quad
   \norm{v_i^0}_{(\WW_{q'}^{2-2/q'}(\Omega)^\ast, \WW_q^{2-2/q}(\Omega))_{\theta,q}}
   \\
   &\quad \text{and} \quad
   \norm{f_i}_{(\WW_q^{(1,2)}(\Omega_\tau)^\ast, \LL_q(\Omega_\tau))_{\theta,q}}
  \end{align*}
 of $h_i$, $v_i^0$ and $f_i$, respectively.
 This direct approach, however, runs into the problem of how to characterize these interpolation spaces, e.g.\ by identifying them with suitable Besov or, more general, Triebel--Lizorkin spaces, if possible.
 Unfortunately, the test functions $\Dcal(\Omega) = \CC_c^\infty(\Omega)$ are \emph{not dense} in $\WW_{q'}^{2-2/q'}(\Omega)$ for sufficiently large $q' > 1$ (i.e.\ for small $q > 1$), and, therefore, the characterizations which are valid for the full space $\R^d$ do, in general, not remain true if one replaces $\R^d$ by an open, regular domain $\Omega \subseteq \R^d$.
 In fact, in this case the dual spaces also include elements which are \emph{no distributions}.
 We refer to \cite{Ama09} for results in the context of \emph{corners}, i.e.\ intersections of open or closed half-spaces, in $\R^d$.
 In our situation, however, we want to estimate an element in the dual space for which will \emph{a-priori} know that it is represented by an -- at least locally -- integrable function; in particular a distribution.
 To avoid the problems sketched above, we use that thanks to the regularity of the domain there exist \emph{universal extension operators} $\Ecal_{\Omega \rightarrow \R^d}: \LL_1(\Omega) \rightarrow \LL_1(\R^d)$ and $\Ecal_{\Omega_\tau \rightarrow \R^{1+d}}: \LL_1(\Omega_\tau) \rightarrow \LL_1(\R^{1+d})$, cf.\ \cite{Rychkov_1999}, such that
  \begin{align*}
   \Ecal_{\Omega \rightarrow \R^d}|_{\WW_r^s(\Omega)}:
    \quad
    &\WW_r^s(\Omega) \rightarrow \WW_r^s(\R^d),
    \\
   \Ecal_{\Omega_\tau \rightarrow \R^d}|_{\WW_r^{(1,2) \cdot s}(\Omega_\tau)}:
    &\WW_r^{(1,2) \cdot s}(\Omega_\tau) \rightarrow \WW_r^{(1,2) \cdot s}(\R^{1+d}),
    \quad
    r \in (1,\infty), \, s \in [0,2],
  \end{align*}
 i.e.\ for each $r \in (1,\infty)$, $s \in [0,2]$, these are \emph{bounded linear operators}, and the action of $\Ecal_{\Omega \rightarrow \R^d}$ resp.\ $\Ecal_{\Omega_\tau \rightarrow \R^{1+d}}$ does not depend on $r$ and $s$, e.g.\ $\Ecal_{\Omega \rightarrow \R^d}^{(r,s)} u = \Ecal_{\Omega \rightarrow \R^d}^{(r',s')} u$, whenever $u \in \WW_r^s(\Omega) \cap \WW_{r'}^{s'}$ for some $r, r' \in (1,\infty)$ and $s, s' \in [0,2]$.
 Thanks to the regularity of the domain $\Omega$, these exist.
 Moreover, for the Dirichlet trace function, we consider a smooth decomposition $\{ \psi_k: \, k = 1, \ldots, m \}$ of unity in a neighbourhood subordinate to an open covering $\{ U_k: k = 1, \ldots, m \}$ of the compact set $\overline{\Omega}$, and $\CC^2$-isomorphisms $\Phi_k: U_k \rightarrow V_k \subseteq \R^d$ such that $\Phi_k(U_k \cap \Omega) = V_k \cap (\R^{d-1} \times (0,\infty))$ and $\Phi_k(U_k \cap \Sigma) = V_k \cap (\R^{d-1} \times \{0\})$, and the map $\Tcal$ defined by
  \[
   (\Tcal u)_k(t,\vec x')
    = \begin{cases}
     (\psi_k u)(t, \Phi_k^{-1}(\vec x',0)),
     &\text{if } (x',0) \in V_k,
     \\
     0,
     &\text{otherwise}.
    \end{cases}
  \]
 Note that $\mathcal{T}|_{\WW_r^{(1,2) \cdot s}(\Sigma_\tau)}$ is a bounded linear operator from $\WW_r^{(1,2) \cdot s}(\Sigma_\tau)$ to $[\WW_r^{(1,2) \cdot s}(\R \times \R^{d-1})]^m$ for each $r \in (1, \infty)$ and $s \in [0,2]$.
 In fact, it is a topologic linear isomorphism between $\WW_r^{(1,2) \cdot s}(\Sigma_\tau)$ and $\Tcal(\WW_r^{(1,2) \cdot s}(\Sigma_\tau)) \subseteq \WW_r^{(1,2) \cdot s}(\R \times \R^{d-1})$ for each $r \in (1,\infty)$ and $s \in [0,2]$.
 We then find that
  \begin{align*}
   &\norm{\Ecal_{\Omega_\tau \rightarrow \R^{1+d}} v_i}_{\LL_q(\R^{1+d})}
    \leq \norm{v_i}_{\LL_q(\Omega_\tau)}
    \\
    &\leq C \left( \norm{h_i}_{\mathring \WW_{q'}^{(1,2)} \cdot (1-\frac{1}{2q'})}(\Sigma_\tau) + \norm{v_i^0}_{\WW_{q'}^{2-2/q'}(\Omega)^\ast} + \norm{f_i}_{\WW_{q'}^{(1,2)}(\Omega_\tau)} \right)
    \\
    &\leq C \left( \norm{\Tcal h_i}_{\WW_{q}^{(1,2) \cdot (- \frac{1}{2q})}(\R^{1+d})} + \norm{\Ecal_{\Omega \rightarrow \R^d} v_i^0}_{\WW_q^{(1,2) \cdot (- 2/q)}(\R^d)} + \norm{\Ecal_{\Omega_\tau \rightarrow \R^{1+d}} f_i}_{\WW_q^{(1,2) \cdot (-1)}(\R^{1+d})} \right)
  \end{align*}
 and
  \begin{align*}
   &\norm{\Ecal_{\Omega_\tau \rightarrow \R^{1+d}} v_i}_{\WW_q^{(1,2)}(\Omega_\tau)}
    \leq C \norm{v_i}_{\WW_q^{(1,2)}(\Omega_\tau)}
    \\
    &\leq C \left( \norm{h_i}_{\WW_q^{(1,2) \cdot (1 - \frac{1}{2q})}(\Sigma_\tau)} + \norm{v_i^0}_{\WW_q^{2-2/q}(\Omega)} + \norm{f_i}_{\LL_q(\Omega_\tau)} \right)
    \\
    &\leq C \left( \norm{\Tcal h_i}_{\WW_q^{(1,2) \cdot (1 - \frac{1}{2q})}(\R^{1+(d-1)};\R^N)} + \norm{\Ecal_{\Omega \rightarrow \R^d} v_i^0}_{\WW_q^{2-2/q}(\R^d)} + \norm{\Ecal_{\Omega_\tau \rightarrow \R^{1+d}} f_i}_{\LL_q(\R^{1+d})} \right).
  \end{align*}
 We may now interpolate between these inequalities with the real interpolation functor $(\cdot,\cdot)_{\theta,q}$ for some $\theta \in (0,1)$ and obtain that
  \begin{align*}
   &\norm{\Ecal_{\Omega \rightarrow \R^{1+d}} v_i}_{\BB_q^{(1,2) \cdot \theta}(\R^{1+d})}
    \\
    &\leq C \left( \norm{\Tcal h_i}_{\BB_q^{(1,2) \cdot (\theta - \frac{1}{2q})}(\R^{1+d};\R^N)} + \norm{\Ecal_{\Omega \rightarrow \R^d} v_i^0}_{\BB_q^{2(\theta-\frac{1}{q})}(\R^d)} + \norm{f_i}_{\BB_q^{(1,2) \cdot (\theta - 1)}(\R^{1+d})} \right).
  \end{align*}
 For the particular choice $\theta = \frac{1}{2q}$, we find that the spatial trace space of $\BB_q^{(1,2) \cdot \frac{1}{2q}}(\R^{1+d})$ for the restriction to $\R \times \Sigma$ continuously embeds into $\LL_q(\R \times \Sigma)$, so that we may conclude that
  \begin{align*}
   &\norm{(\Ecal_{\Omega_\tau \rightarrow \R^{1+d}} v_i)|_\Sigma}_{\LL_q(\R \times \Sigma)}
    \leq C \norm{\Ecal_{\Omega_\tau \rightarrow \R^{1+d}} v_i}_{\BB_q^{(1,2) \cdot \theta}(\R^{1+d})}
    \\
    &\leq C \left( \norm{\Tcal h_i}_{\BB_q^{(1,2) \cdot (\theta - \frac{1}{2q})}(\R^{1+(d-1)}; \R^N)} + \norm{\Ecal_{\Omega \rightarrow \R^d} v_i^0}_{\BB_q^{2(\theta - \frac{1}{q})}(\R^d)} + \norm{f_i}_{\BB_q^{(1,2) \cdot (\theta-1)}(\R^{1+d})} \right)
    \\
    &\leq C \left( \norm{\Tcal h_i}_{\LL_q(\R^{1+(d-1)}; \R^N)} + \norm{\Ecal_{\Omega \rightarrow \R^d} v_i^0}_{\WW_q^{(1,2)}(\R^d)} + \norm{f_i}_{\LL_q(\R^{1+d})} \right)
    \\
    &\leq C \left( 1 + \norm{c_i|_\Sigma}_{\WW_{q'}^{(1,2) \cdot (\frac{1}{2} - \frac{1}{2q'})}(\Sigma_\tau)^\ast} + \norm{v_i^0}_{\WW_q^{(1,2)}(\Omega)} + \norm{c_i}_{\LL_{q\gamma^\Omega}(\Omega_\tau)}^{\gamma^\Omega} \right).
  \end{align*}
 
 \begin{lemma}[Parabolic smoothing]
  \label{lem:smoothing_effect}
  Let $(\vec c, \vec c^\Sigma) \in \WW_p^{(1,2)}(\Omega_T)^N \times \WW_p^{(1,2)}(\Sigma_T)^N$ be the solution to the reaction-diffusion-advection-sorption system \eqref{eqn:RDASS} with Langmuir adsorption for some $p \in ( \max \{1, \frac{\mu^\Omega - 1}{\mu^\Omega} \frac{d+2}{2}, \frac{\mu^\Sigma - 1}{\mu^\Sigma} \frac{d + 1}{2} \}, \infty )$ and some $T > 0$.
  Then, $(\vec c, \vec c^\Sigma) \in \WW_q^{(1,2)}((\varepsilon,T) \times \Omega)^N \times \WW_q^{(1,2)}((\varepsilon,T) \times \Sigma))^N$ for all $\varepsilon \in (0,T)$ and $q \in (1, \infty)$.
 \end{lemma}
 \begin{remark}
  One might suspect that the smoothing property of the reaction-diffusion-advection-sorption system under consideration is valid \emph{independently} of the particular triangular structure imposed on the reactive terms.
  Note, however, that we also include the case of rather low values of $p \in (1, \infty)$.
  In contrast to the case $p > \frac{d+2}{2}$, the solution space $\WW_p^{(1,2)}(\Omega_\tau) \times \WW_p^{(1,2)}(\Sigma_\tau)$ does, in general, not continuously embed into those spaces which would be required to employ Hölder regularity theory of reaction-diffusion equations.
  In fact, if $\WW_p^{(1,2)}(\Omega_\tau) \hookrightarrow \CC^{(1,2) \cdot \alpha}(\overline{\Omega_\tau})$, $\WW_p^{(1,2)}(\Omega_\tau) \hookrightarrow \CC^{(1,2) \cdot (\alpha + \frac{1}{2})}(\overline{\Sigma_\tau})$ (via the boundary trace map), and $\WW_p^{(1,2)}(\Sigma_\tau) \hookrightarrow \CC^{(1,2) \cdot \alpha}(\overline{\Sigma_\tau})$ embed continuously for some $\alpha \in (0, \frac{1}{2})$ (which is always true if $p > \frac{d+2}{2}$), we could exploit the theory of quasi-autonomous, linear parabolic equations to deduce that actually $(\vec c, \vec c^\Sigma) \in \CC^{(1,2) \cdot (1 + \alpha)}((\varepsilon,\tau);\R^N) \times \CC^{(1,2) \cdot (1 + \alpha)}((\varepsilon, \tau) \times \Sigma; \R^N)$ for every $\varepsilon > 0$.
  \newline
  In \cite{ShaMor16}, the smoothing property has been demonstrated for the case $p > d$.
  In fact, from their results if follows that if $p > d$ and $(\vec c^0, \vec c^{\Sigma,0}) \in \WW_p^2(\Omega;\R^N) \times \WW_p^2(\Sigma;\R^N)$ (satisfying the compatibility conditions), there is a unique classical solution in the class
   \begin{align*}
    (\vec c, \vec c^\Sigma)
     &\in \CC([0,\tau]; \LL_p(\Omega; \R^N) \times \LL_p(\Sigma;\R^N))
      \cap \CC^1((0,\tau]; \CC(\overline{\Omega};\R^N) \times \CC(\Sigma; \R^N))
      \\ &\qquad
      \cap \CC((0,\tau]; \CC^2(\overline{\Omega}; \R^N) \times \CC^2(\Sigma; \R^N)).
   \end{align*}
  Note that the condition $p > d$ is sufficient to imply that $\WW_p^2(\Omega) \hookrightarrow \CC^{1+\alpha}(\overline{\Omega})$ for $\alpha := 1 - \frac{d}{p} \in (0,1)$, and then $\WW_p^2(\Sigma) \hookrightarrow \CC^{1+\alpha}(\Sigma)$ as well, so that the tools from Hölder regularity theory for parabolic equations become applicable.
 \end{remark} 
 \begin{proof}[Proof of Lemma \ref{lem:smoothing_effect}]
 By the $\LL_p$--$\LL_q$-estimates of Lemma~\ref{lem:L_p-L_q-estimates} (which crucially depend on the triangular structure!), we know that $\vec c \in \LL_q(\Omega_T;\R^N)$ and $\vec c|_\Sigma$, $\vec c^\Sigma \in \LL_q(\Sigma_T;\R^N)$ for all $q \in (1, \infty)$.
 By the polynomial boundedness of the chemical reaction terms $\vec r^\Omega(\vec c)$ and $\vec r^\Sigma(\vec c^\Sigma)$ as well as the sorption model $\vec r^\mathrm{sorp}(\vec c|_\Sigma, \vec c^\Sigma)$, we may, thus, conclude that these are functions in $\LL_q(\Omega_T)$ and $\LL_q(\Sigma_T)$ for all $q \in (1, \infty)$ as well.
 Therefore, the solutions $c_i^\Sigma: [0,T] \rightarrow \R$ of the problems
  \begin{alignat*}{2}
    \partial_t c_i^\Sigma - d_i^\Sigma \Delta_\Sigma c_i^\Sigma
     &= F_i^\Sigma + G_i^\Sigma
     \qquad
     &&\text{on } \Sigma_T,
     \\
    c_i^\Sigma(0,\cdot)
     &= c_i^{\Sigma,0}
     \qquad
     &&\text{on } \Sigma,  
  \end{alignat*}
 where $F_i^\Sigma = r_i^\Sigma(\vec c^\Sigma)$ and $G_i^\Sigma = s_i^\Sigma(c_i, c_i^\Sigma) \in \LL_q((\varepsilon,T) \times \Sigma)$ lie in $\WW_q^{(1,2)}((\varepsilon,T) \times \Sigma)$ for all $q \in (1, \infty)$ and $\varepsilon \in (0,T)$.
 In fact, we may employ the representation
  \[
   c_i^\Sigma(t,\cdot)
    = \ee^{t A_i^\Sigma} c_i^{\Sigma,0} + \int_0^t \ee^{(t-s) A_i^\Sigma} (F_i^\Sigma + G_i^\Sigma)(s,\cdot) \dd s,
  \]
 where the inhomogeneous term is in $\WW_q^{(1,2)}(\Sigma_T)$ since $F_i^\Sigma + G_i^\Sigma \in \LL_q(\Sigma_T)$ and by $\LL_q$-maximal regularity, for all $q \in (1,\infty)$.
 For the homogeneous part $\ee^{t A_i^\Sigma} c_i^{\Sigma,0}$ we may use the smoothing effect of the analytic $C_0$-semigroup $(\ee^{t A_i^\Sigma})$ on the space $\LL_p(\Sigma)$, so that this term lies in $\WW_q^{(1,2)}((\varepsilon,T) \times \Sigma)$ for all $\varepsilon \in (0,T)$.
 In particular, $c_i^\Sigma \in \CC((0,T];\WW_q^{2-2/q}(\Sigma))$ for all $q \in (1,\infty)$.
 Concerning the bulk term, we see that $c_i|_{t \in [\varepsilon,T]}$ is the solution to the linear parabolic initial-boundary value problem
  \[
   \begin{cases}
    \partial_t c_i + \vec v \cdot \nabla c_i - d_i \Delta c_i
     = F_i
     &\text{in } (\varepsilon,T) \times \Omega,
     \\
    k_i^\mathrm{ad} (1 - \frac{c_i^{\Sigma,+}}{c_i^{\Sigma,\infty}})^+ + d_i \partial_{\vec \nu} c_i
     = G_i
     &\text{on } (\varepsilon,T) \times \Sigma,
     \\
    c_i(\varepsilon,0)
     = c_i^\varepsilon
     &\text{in } \overline{\Omega}
   \end{cases}
  \]
 for $c_i^\varepsilon := c_i(\varepsilon,0) \in \WW_p^{2-2/p}(\Omega)$, and $F_i = r_i^\Omega(\vec c) \in \LL_q((\varepsilon,T) \times \Omega)$ for all $q \in (1, \infty)$, $G_i = k_i^\mathrm{de} c_i^\Sigma \in \WW_q^{(1,2)}((\varepsilon,T) \times \Sigma)$ for all $q \in (q, \infty)$.
 Now, we may write $c_i|_{t \in [\varepsilon,T]} = u_i + v_i + w_i$, where $u_i \in \WW_q^{(1,2)}((\varepsilon,T) \times \Omega$) for all $q \in (1,\infty)$ with $\frac{1}{q} \geq \frac{1}{p} - \frac{1}{d+2}$ is the solution to the parabolic initial-boundary value problem
  \[
   \begin{cases}
    \partial_t u_i - d_i \Delta u_i
     = F_i - \vec v \cdot \nabla c_i
     &\text{in } (\varepsilon,T) \times \Omega,
     \\
   k_i^\mathrm{ad} (1 - \frac{c_i^{\Sigma,+}}{c_i^{\Sigma,\infty}})^+ u_i + d_i \partial_{\vec \nu} u_i
     = G_i
     &\text{on } (\varepsilon,T) \times \Sigma,
     \\
    u_i(\varepsilon,\cdot)
     = 0
     &\text{on } \overline{\Omega},
   \end{cases}
  \]
 and where we may use the fact that $c_i \in \WW_p^{(1,2)}(\Omega_T) \hookrightarrow \WW_q^{(1,2) \cdot \frac{1}{2}}(\Omega_T) \hookrightarrow \LL_q((\varepsilon,T);\LL_q(\Sigma))$ for $q \in (1,\infty)$ such that $\frac{1}{q} \geq \frac{1}{p} - \frac{1}{d+2}$.
 Moreover, the functions $v_i$ and $w_i$ are constructed by taking some $w_i^\varepsilon \in \bigcap_{q \in (1,\infty)} \WW_q^{2-2/q}(\Omega)$ such that $- d_i \partial_{\vec \nu} w_i = - k_i^\mathrm{ad} (1 - \frac{c_i^{\Sigma,+}(\varepsilon,\cdot)}{c_i^{\Sigma,\infty}})^+ (c_i(\varepsilon,0) + w_i^\varepsilon)$ and then defining $v_i$ and $w_i$ as the solutions to the initial-boundary value problems
   \[
    \begin{cases}
     \partial_t v_i - d_i \Delta v_i
      = 0
      &\text{in } (\varepsilon,T) \times \Omega,
      \\
     d_i \partial_{\vec \nu} v_i
      = 0
      &\text{on } (\varepsilon,T) \times \Sigma,
      \\
     v_i(\varepsilon,\cdot)
      = c_i(\varepsilon,0) - w_i^\varepsilon
      &\text{on } \overline{\Omega}
    \end{cases}
   \]
  and
   \[
    \begin{cases}
    \partial_t w_i - d_i \Delta w_i
     = 0
     &\text{in } (\varepsilon,T) \times \Omega,
     \\
    k_i^\mathrm{ad} (1 - \frac{c_i^{\Sigma,+}}{c_i^{\Sigma,\infty}})^+ w_i + d_i \partial_{\vec \nu} w_i
     = k_i^\mathrm{ad} (1 - \frac{c_i^{\Sigma,+}}{c_i^{\Sigma,\infty}})^+ v_i
     &\text{on } (\varepsilon,T) \times \Sigma,
     \\
    w_i(\varepsilon,\cdot)
     = w_i^\varepsilon
     &\text{on } \Sigma.
    \end{cases}
   \]
  Here, by $\LL_q$-maximal regularity, the latter problem has a solution $w_i \in \WW_q^{(1,2)}((\varepsilon,T) \times \Omega)$ for all $q \in (1,\infty)$.
  Moreover, by classical smoothing properties of the Neumann semigroup on $\LL_p(\Omega)$, we obtain $v_i \in \WW_q^{(1,2)}((2\varepsilon,T) \times \Omega)$ for all $q \in (1,\infty)$.
  Putting things together, this means that $c_i \in \WW_q^{(1,2)}((2\varepsilon,T) \times \Omega)$ for every $\varepsilon > 0$, and, in particular, $c_i \in \CC((0,T];\WW_q^{2-2/q}(\Omega))$ for every $q \in (1,\infty)$ such that $\frac{1}{q} \geq \frac{1}{p} - \frac{1}{d+2}$.
  Repeating the argument for $p' \in (1,\infty)$ such that $\frac{1}{p'} \geq \frac{1}{p} - \frac{1}{d+2}$ and possibly multiple iterations, we obtain that actually $c_i \in \WW_q^{(1,2)}((\varepsilon,T) \times \Omega)$ for all $\varepsilon > 0$, and thus $c_i \in \CC((0,T];\WW_q^{2-2/q}(\Omega))$ for \emph{all} $q \in (1, \infty)$. 
 \end{proof}
 
\begin{corollary}
\label{cor:max_existence_for_all_q}
 If the strong $\WW_p^{(1,2)}$-solution $(\vec c, \vec c^\Sigma) \in \WW_p^{(1,2)}(\Omega_T;\R^N) \times \WW_p^{(1,2)}(\Sigma_T;\R^N)$ exists for some $0 < T < \infty$ and $p \in (1, \infty)$, then its restriction to times $t \in (\varepsilon, T)$ is a $\WW_q^{(1,2)}$-solution $(\vec c, \vec c^\Sigma) \in \WW_q^{(1,2)}((\varepsilon, T) \times \Omega;\R^N) \times \WW_p^{(1,2)}((\varepsilon,T) \times \Sigma;\R^N)$ for all $q \in (1,\infty)$. In particular, the maximal time of existence $T_\mathrm{max} \in (0, \infty]$ does not depend on the choice of $p \in (1, \infty)$.
 In particular, if
  \[
   \sup_{t \in [\varepsilon,T]} \norm{(\vec c(t,\cdot), \vec c^\Sigma(t,\cdot))}_{\WW_p^{2-2/p}} 
    < \infty
  \]
for some $p \in (1,\infty)$, then $\sup_{t \in [\varepsilon,T]} \norm{(\vec c(t,\cdot), \vec c^\Sigma(t,\cdot))}_{\WW_q^{2-2/q}} < \infty$ for all $q \in (1, \infty)$.
\end{corollary}

 Once we have the $\LL_p$--$\LL_q$-estimates at hand, we may easily infer exponential growth bounds for the $\LL_p$-norms of the solution on $\Omega_\tau$ and $\Sigma_\tau$, respectively.
 
 \begin{lemma}[$\LL_p$-growth bounds]
 \label{lem:exponential_bounds}
  For every $T_0 > 0$, $p \in (1, \infty]$, there are $M = M(p,T_0) > 0$, $\omega = \omega(p,T_0) \in \R$ (also depending on the initial data) such that
   \[
    \sum_{i=1}^N \big( \norm{c_i}_{\LL_p(\Omega_\tau)}
      + \norm{c_i|_\Sigma}_{\LL_p(\Sigma_\tau)}
      + \norm{c_i^\Sigma}_{\LL_p(\Sigma_\tau)} \big)
     \leq M \ee^{\omega \tau},
     \quad
     \tau \in (0, \min \{T_0, T_\mathrm{max}\}).
   \]
 \end{lemma}
 \begin{proof}
  We start by employing the $\LL_p$--$\LL_q$-estimates of Lemma~\ref{lem:L_p-L_q-estimates} for the special case $p \in (1, \infty)$ and $q = \infty$, resulting in
   \begin{align*}
    &\norm{\vec c(\tau,\cdot)}_{\LL_p(\Omega)}^p
     + \norm{\vec c^\Sigma(\tau,\cdot)}_{\LL_p(\Sigma)}^p
     \\
     &\leq C \sup_{t \in [0, \tau]} \big(
      \norm{\vec c(t,\cdot)}_{\LL_p(\Omega)}^p
       + \norm{\vec c^\Sigma(t,\cdot)}_{\LL_p(\Sigma)}^p \big)
     \\
     &\leq C \big(
       \norm{\vec c}_{\LL_\infty(\Omega_\tau)}^p
       + \norm{\vec c^\Sigma}_{\LL_\infty(\Sigma_\tau)}^p \big)
     \\
     &\leq C \big( 1
      + \norm{\vec c}_{\LL_p(\Omega_\tau)}^p
      + \norm{\vec c^\Sigma}_{\LL_p(\Sigma_\tau)}^p \big)
     \\
     &\leq C \big[ 1
      + \int_0^\tau \big( \norm{\vec c(t,\cdot)}_{\LL_p(\Omega)}^p
      + \norm{\vec c^\Sigma(t,\cdot)}_{\LL_p(\Sigma)}^p \big) \dd t \big],
     \quad
     \tau \in (0, T_0] \cap (0, T_\mathrm{max}).
   \end{align*}
  By Gronwall's Lemma, see e.g.\ \cite[Lemma VII.1.15]{AE2}), there are, thus, constants $M' = M'_{p,T_0} > 0$ and $\omega' = \omega'_{p,T_0} > 0$ (also depending on the initial data through the constant $C$ in the $\LL_p$--$\LL_q$-estimates) such that
   \begin{equation}
    \norm{\vec c(t,\cdot)}_{\LL_p(\Omega)}^p
     + \norm{\vec c^\Sigma(t,\cdot)}_{\LL_p(\Sigma)}^p
     \leq M' \ee^{\omega' t},
     \quad
     t \in (0,T_0] \times (0, T_\mathrm{max}).
     \label{eqn:exponential-bounds-1}
   \end{equation}
  Integrating over $t \in (0, \tau)$, this proves the claim for $p < \infty$.
  The case $p = \infty$ follows by integrating equation \eqref{eqn:exponential-bounds-1} in time for some finite $p \in (1, \infty)$ and then employing the $\LL_p$--$\LL_\infty$-estimate provided by Lemma~\ref{lem:L_p-L_q-estimates}.
  \newline
  Finally, we may proceed just as in the remarks after the proof of Lemma~\ref{lem:L_p-L_q-estimates} to estimate the $\LL_p$-norm of $\vec c|_\Sigma$ as well.
 \end{proof}
 
\subsection{Statement and proof of the global existence result}
 
 After these preparations, we are ready to prove the following theorem on global-in-time existence of solutions to the bulk-surface reaction-diffusion-advection system \eqref{eqn:RDSS} for suitable initial data. 
 
 \begin{theorem}[Global-in-time well-posedness]
  \label{thm:global_existence}
  Let $J = \R_+ = [0, \infty)$, $p \in [\frac{d+2}{2}, \infty)$ and let Assumptions~\ref{assmpt:general} and~\ref{assmpt:triangular_structure} be valid.
   Moreover, assume that the sorption is modelled by the Langmuir sorption model.
  Then, for all compatible initial data
   \[
    (\vec c^0, \vec c^{\Sigma,0}) \in \fs I_p^\Omega \times \fs I_p^\Sigma
   \]
  with $(\vec c^0, \vec c^{\Sigma}) \geq (0, 0)$, the reaction-diffusion-advection-sorption system
   \begin{alignat*}{2}
    \partial_t c_i + \vec v \cdot \nabla c_i - \dv (d_i \nabla c_i)
     &= r^\Omega_i(\vec c)
	 \qquad &
     &\text{on } (0,\infty) \times \Omega,
     \\
    \partial_t c_i^\Sigma - \dv_\Sigma(d_i^\Sigma \nabla_\Sigma c_i^\Sigma)
     &= s_i^\Sigma(c_i, c_i^\Sigma) + r_i^\Sigma(\vec c^\Sigma)
	 \qquad &
     &\text{on } (0, \infty) \times \Sigma,
     \\
    - d_i \partial_\nu c_i
     &= s_i^\Sigma(c_i, c_i^\Sigma)
	 \qquad &
     &\text{on } (0, \infty) \times \Sigma,
     \\
    c_i(0, \cdot)
     &= c^0_i
     \qquad &
     &\text{in } \Omega,
     \\
    c_i^\Sigma(0, \cdot)
     &= c^{\Sigma,0}_i
	 \qquad &
     &\text{on } \Sigma,
   \end{alignat*}
  for $i = 1, \ldots, N$, has a unique global-in-time solution
   \[
    (\vec c, \vec c^\Sigma)
     \in \WW^1_{p,\mathrm{loc}}([0,\infty); \LL_p(\Omega;\R^N) \times \LL_p(\Sigma;\R^N)) \cap \LL_{p,\mathrm{loc}}([0, \infty); \WW_p^2(\Omega;\R^N) \times \WW_p^2(\Sigma;\R^N)).
   \]
  Moreover, for every $q \in (1, \infty)$,
   \[
    (\vec c, \vec c^\Sigma)
     \in \WW^1_{q,\mathrm{loc}}((0,\infty); \LL_q(\Omega;\R^N) \times \LL_q(\Sigma;\R^N)) \cap \LL_{q,\mathrm{loc}}((0, \infty); \WW_q^2(\Omega;\R^N) \times \WW_q^2(\Sigma;\R^N)).
   \]
 \end{theorem}
 
  The result is based on the comparison principle and the dual estimates for scalar reaction-diffusion-advection equations, i.e.\ a technique heavily used especially by the French school around Michel Pierre, cf.\ the survey \cite{Pierre_2010} and related papers.
  Note that we need regularity of class $\WW_p^{2-2/p}$ for some $p \geq \frac{d+2}{2}$ to get a local-in time solution, which then immediately regularises, so that for any positive times $t > 0$ the solution lies in all spaces of class $\WW_q^{2-2/q}$ for all $q \in (1, \infty)$.

 \begin{proof}
  We employ the following strategy to establish the global-in-time existence result, where thanks to the preparations done up to now, we will be able to go through each of these steps at fast pace:
   \begin{enumerate}
    \item
     We start with smooth initial data $(\vec c, \vec c^\Sigma) \in \bigcap_{q>1} (\fs I_q^\Omega \times \fs I_q^\Sigma)$. More general initial data can be handled later on by exploiting the \emph{smoothing effect} of the parabolic system, cf.\ Lemma~\ref{lem:smoothing_effect}.
    \item
     We fix an arbitrary time horizon $T_0 > 0$ and consider $T^\ast = \min \{T_0, T_\mathrm{max} \}$ for the maximal time of existence $T_\mathrm{max} \in (0,\infty]$ from the local-in-time existence result.
     By the unique-continuation property, it suffices to show that the solution extends to $[0,T^\ast]$ for all such $T^\ast > 0$.
    \item
     As a preliminary step, we demonstrate that the solution is in $\LL_q(\Omega_{T^\ast};\R^N) \times \LL_q(\Sigma_{T^\ast};\R^N)$ for all $q \in (1,\infty)$.
    \item
     In combination with polynomial boundedness of the chemical reaction and surface ad- and desorption models, we may then derive $\LL_p$-growth bounds for the reactive and sorption terms $\vec r^\Omega(\vec c)$, $\vec r^\Sigma(\vec c^\Sigma)$ and $\vec s^\Sigma(\vec c|_\Sigma, \vec c^\Sigma)$.
    \item
     By $\LL_q$-maximal regularity of the linearised reaction-diffusion system on the surface, we can then show that the surface part actually lies in the class $\vec c^\Sigma \in \fs E_q^\Sigma(T^\ast)$ for all $q \in (1, \infty)$.
    \item
     Only thereafter, we may exploit the $\LL_q$-maximal regularity of the reaction-diffusion-advection system in the bulk phase with inhomogeneous Neumann boundary conditions which arises from linearisation, and obtain that also $\vec c \in \fs E_q^\Omega(T^\ast)$ for all $q \in (1,\infty)$.
     \item 
      We may then invoke the blow-up criterion from the local-in-time existence result to find that $T^\ast = T_0$ for every choice of $T_0 > 0$, thus $T_\mathrm{max} = \infty$, i.e.\ global-in-time existence.
   \end{enumerate}
  Let us, for the moment, assume that $(\vec c, \vec c^\Sigma) \in \fs I_q^\Omega \times \fs I_q^\Sigma$ for all $q \in (1, \infty)$.
  To show global existence, i.e.\ $T_\mathrm{max} = + \infty$, it suffices to show that there is a strong solution on each finite time interval $[0,T_0]$, where $T_0 > 0$ can be chosen arbitrary large.
  Thus, let $T_0 > 0$ be arbitrary and consider $T^\ast = \min \{ T_0, T_\mathrm{max} \}$.
  Then, by definition of $T_\mathrm{max}$, the strong solution exists on $[0,T^\ast)$, and for each $T \in (0,T^\ast)$, its restriction to times $t \in [0,T]$ is in the class $\fs E_p^\Omega(T) \times \fs E_p^\Sigma(T)$.
  Moreover, from the local-in-time existence result for our class of reaction-diffusion-advection-sorption systems, employing the blow-up criterion and Corollary~\ref{cor:max_existence_for_all_q}, we may conclude that
   \[
    T^\ast
     = T_\mathrm{max}
      \quad \Leftrightarrow \quad      
      \limsup_{t \rightarrow T^\ast} \norm{(\vec c, \vec c^\Sigma)(t,\cdot)}_{\WW_q^{2-2/q}(\Omega) \times \WW_q^{2-2/q}(\Sigma)}
       = \infty
       \quad
       \text{for some } q \in (1, \infty)
   \]
  and, therefore, it remains to prove that $(\vec c, \vec c^\Sigma) \in \BUC([0,T^\ast); \WW_q^{2-2/q}(\Omega)^N \times \WW_q^{2-2/q}(\Sigma)^N)$ for some $q \in (1, \infty)$.
  \newline
  As a preliminary step, note that $(\vec c, \vec c^\Sigma) \in \BUC([0,T^\ast); \LL_q(\Omega;\R^N) \times \LL_q(\Sigma;\R^N))$ for all $q \in (1,\infty)$.
  In fact, by Lemma~\ref{lem:L_p-L_q-estimates} and Lemma~\ref{lem:exponential_bounds}, for every $p \in (1,\infty]$ there is a constant $C_{p,T^\ast} > 0$ such that
   \[
    \norm{\vec c}_{\LL_p(\Omega_\tau)} + \norm{\vec c|_\Sigma}_{\LL_p(\Sigma_\tau)} + \norm{\vec c^\Sigma}_{\LL_p(\Sigma_\tau)}
     \leq C_{p,T^\ast}
     \quad
     \text{for every } \tau \in (0,T^\ast).
   \]
  By the polynomial boundedness of the chemical reaction rates and Remark~\ref{rem:Nemytskii_Operators}, thus
   \begin{alignat*}{2}
    \norm{\vec r^\Omega(\vec c)}_{\LL_p(\Omega_\tau)}
     &\leq M^\Omega (1 + \norm{\vec c}_{\LL_{p\gamma^\Omega}(\Omega_\tau)})^{\gamma^\Omega}
     \leq C_{p\gamma^\Omega, T^\ast},
     \quad &
     &\tau \in (0,T^\ast),
     \\
    \norm{\vec r^\Sigma(\vec c^\Sigma)}_{\LL_p(\Sigma_\tau)}
     &\leq M^\Omega (1 + \norm{\vec c^\Sigma}_{\LL_{p\gamma^\Omega}(\Sigma_\tau)})^{\gamma^\Sigma}
     \leq C_{p\gamma^\Sigma,T^\ast},
     \quad &
     &\tau \in (0, T^\ast).
   \end{alignat*}
 Moreover, by the partial lower and upper linear bounds on the sorption model,
  \[
   \norm{s_i^\Sigma(c_i|_\Sigma, c_i^\Sigma)}_{\LL_p(\Sigma_\tau)}
    \leq k_i^\mathrm{ad} \norm{c_i|_\Sigma}_{\LL_p(\Sigma_\tau)} + k_i^\mathrm{de} \norm{c_i^\Sigma}_{\LL_p(\Sigma_\tau)}
    \leq C_{p,T^\ast},
    \quad
    \tau \in (0, T^\ast).
  \]   
  By $\LL_p$-maximal regularity of the linearised reaction-diffusion-system on the surface
   \[
    \begin{cases}
     \partial_t c_i^\Sigma - d_i \Delta_\Sigma c_i^\Sigma
      = F_i^\Sigma + G_i^\Sigma
      &\text{on } \Sigma_\tau,
      \\
     c_i^\Sigma(0,\cdot)
      = c_i^{\Sigma,0}
      &\text{on } \Sigma
    \end{cases}
   \]
  for $F_i^\Sigma := r_i^\Sigma(\vec c^\Sigma)$ and $G_i^\Sigma := s_i^\Sigma(c_i|_\Sigma, c_i^\Sigma) \in \LL_q(\Sigma_\tau)$, we may conclude that
   \[
    \norm{c_i^\Sigma}_{\WW_{q}^{(1,2)}(\Omega_\tau)}
     \leq C \big( \norm{c_i^{\Sigma,0}}_{\WW_q^{2-2/q}(\Sigma)} + \norm{F_i^\Sigma + G_i^\Sigma}_{\LL_q(\Sigma_\tau)} \big)
     \leq C_{q,T^\ast}
     \quad
     \text{for all } \tau \in (0, T^\ast).
   \]
  Hence, $c_i^\Sigma \in \WW_q^{(1,2)}(\Sigma_{T^\ast})$ for all $q \in (1, \infty)$ and $i = 1, \ldots, N$, which, in particular, already implies that
   \[
    \sup_{t \in [0,T^\ast)} \norm{c_i^\Sigma(t,\cdot)}_{\WW_q^{2-2/q}(\Sigma)}
     < \infty,
   \]
  as $\WW_q^{(1,2)}(\Sigma_{T^\ast}) \hookrightarrow \BUC([0,T^\ast); \WW_q^{2-2/q}(\Omega))$ embeds continuously.
  \newline
  With this knowledge at hand, let us take a look on the reaction-diffusion-advection-system in the bulk phase
   \[
    \begin{cases}
     \partial_t c_i + \vec v \cdot \nabla c_i - d_i \Delta c_i
      = F_i
      &\text{in } \Omega_{T^\ast},
      \\
     k_i^\mathrm{ad} (1 - \frac{c_i^{\Sigma,+}}{c_i^{\Sigma,\infty}})^+ c_i + d_i \partial_{\vec \nu} c_i
      = G_i
      &\text{on } \Sigma_{T^\ast},
      \\
     c_i(0,\cdot)
      = c_i^0
      &\text{on } \overline{\Omega},
    \end{cases}
   \]
 where $F_i := r_i^\Omega(\vec c)$ and $G_i := k_i^\mathrm{ad} c_i^\Sigma$.
 Since $\beta_i := k_i^\mathrm{ad} (1 - \frac{c_i^{\Sigma,+}}{c_i^{\Sigma,\infty}})^+ \in \bigcap_{q > 1} \WW_q^1(0,T^\ast;\LL_q(\Sigma)) \cap \LL_q(0,T^\ast;\WW_\infty^1(\Sigma))$ (as we already know that $c_i^\Sigma \in \WW_{q}^{(1,2)}(\Sigma_{T^\ast})$ for all $q \in (1,\infty)$), by \cite{DeHiPr07} this problem admits $\LL_q$-maximal regularity for all $q \in (1,\infty)$:
 Indeed, to satisfy condition (SB) in \cite{DeHiPr07}, we need to ensure that there are $s_{i0}, r_{i0} \geq p$ with $\frac{2}{s_{i0}} + \frac{d-1}{r_{i0}} < 1 - \frac{1}{p} - 1 = - \frac{1}{p}$ such that $\beta_i \in \WW_{s_{i0}}^{\frac{1}{2} - \frac{1}{2p}}((0,T^\ast); \LL_{r_{i0}}(\Sigma)) \cap \LL_{s_{i0}}((0,T^\ast); \WW_{r_{i0}}^{1 - \frac{1}{p}}(\Sigma))$, which in our case is obviously true.
 Therefore, $c_i \in \WW_q^{(1,2)}(\Omega_{T^\ast})$ with
  \[
   \norm{c_i}_{\WW_q^{(1,2)}(\Omega_{T^\ast})}
    \leq C \big( \norm{F_i}_{\LL_q(\Omega_{T^\ast})} + \norm{c_i^0}_{\WW_q^{2-2/q}(\Omega)} + \norm{c_i^\Sigma}_{\WW_q^{(1,2) \cdot (\frac{1}{2} - \frac{1}{2q})}(\Sigma_{T^\ast})} \big)
    \leq C_{p,T^\ast},
    \quad q > 1.
  \]
 This, in particular, implies that
  \[
   (c_i, c_i^\Sigma)
    \in \BUC([0,T^\ast);\WW_q^{2-2/q}(\Omega) \times \WW_q^{2-2/q}(\Sigma)).
  \]
 Hence, by the blow-up criterion, $T^\ast = T_0 < T_\mathrm{max}$ for each choice of $T_0 > 0$, leading to $T_\mathrm{max} = \infty$.
 \newline
 For more general initial data note that, by Lemma~\ref{lem:smoothing_effect}, $(\vec c(\varepsilon,\cdot), \vec c^\Sigma(\varepsilon,\cdot)) \in \bigcap_{q \in (1, \infty)} \fs I_q^\Omega \times \fs I_q^\Sigma$.
 Thus we may simply consider the problem for initial data at time $t = \varepsilon > 0$.
 \end{proof}

 \section*{Funding}
 The authors gratefully acknowledge financial support by the Deutsche Forschungsgemeinschaft (DFG, German Research Foundation), Project-ID 265191195, SFB 1194.

\end{document}